\documentclass{notes}

\usepackage[english]{babel}
\usepackage{pdfpages}
\usepackage{minitoc}
\usepackage{stmaryrd}
\usepackage{enumerate}

\newcommand{\rI}{{\rm I}}

\newcommand{\cC}{\mathcal{C}}

\newcommand{\cF}{\mathcal{F}}

\newcommand{\cH}{\mathcal{H}}

\newcommand{\cL}{\mathcal{L}}

\newcommand{\cR}{\mathcal{R}}
\newcommand{\cS}{\mathcal{S}}
\newcommand{\cT}{\mathcal{T}}
\newcommand{\cU}{\mathcal{U}}

\newcommand{\fo}{{\mathfrak o}}

\newcommand{\fs}{{\mathfrak s}}

\newcommand{\R}{\mathbb{R}}
\newcommand{\C}{\mathbb{C}}

\newcommand{\su}{\mathfrak{su}}
\newcommand{\so}{\mathfrak{so}}

\renewcommand{\O}{{\rm O}}
\newcommand{\SO}{{\rm SO}}
\newcommand{\Sp}{{\rm Sp}}
\newcommand{\Spin}{{\rm Spin}}
\newcommand{\SU}{{\rm SU}}
\newcommand{\GL}{\mathrm{GL}}

\newcommand{\U}{{\rm U}}

\newcommand{\FR}{\rm Fr}

\renewcommand{\det}{\mathop\mathrm{det}\nolimits}

\newcommand{\End}{{\mathrm{End}}}

\newcommand{\del}{\partial}

\newcommand{\id}{\mathrm{id}}
\newcommand{\im}{\mathop{\mathrm{im}}}

\newcommand{\tr}{\mathop{\mathrm{tr}}\nolimits}

\newcommand{\Sym}{\mathrm{Sym}}

\newcommand{\qandq}{\quad\text{and}\quad}

\newcommand{\qforq}{\quad\text{for}\quad}

\def\<{\mathopen{}\left<}
\def\>{\right>\mathclose{}}
\def\({\mathopen{}\left(}
\def\){\right)\mathclose{}}

\usepackage{multicol, color}

\definecolor{gold}{rgb}{0.85,.66,0}
\definecolor{cherry}{rgb}{0.9,.1,.2}
\definecolor{burgundy}{rgb}{0.8,.2,.2}
\definecolor{orangered}{rgb}{0.85,.3,0}
\definecolor{orange}{rgb}{0.85,.4,0}
\definecolor{olive}{rgb}{.45,.4,0}
\definecolor{lime}{rgb}{.6,.9,0}
\definecolor{green}{rgb}{.2,.7,0}
\definecolor{grey}{rgb}{.4,.4,.2}
\definecolor{brown}{rgb}{.4,.3,.1}

\def\om{\omega}
\def\ip{\raise1pt\hbox{\large$\lrcorner\ \!$}} 
\def\w{\wedge}

\newcommand{\ph}{\varphi}
\newcommand{\ps}{\psi}

\newcommand{\wt}{\widetilde}
\newcommand{\wh}{\widehat}

\newcommand{\triv}{\underline}
\newcommand{\ol}{\overline}

\newcommand{\eps}{\epsilon}

\renewcommand{\l}{\ell}

\DeclareMathOperator{\LC}{\mathsf{LC}}
\DeclareMathOperator{\W}{\mathsf{W}}
\DeclareMathOperator{\G}{\mathsf{G}}
\newcommand\tRc{\operatorname{Rc}}

\newcommand{\tRm}{\mathrm{Rm}}
\newcommand{\vol}{\mathsf{vol}}
\DeclareMathOperator\grad{grad}
\DeclareMathOperator\Div{div}
\newcommand{\hk}{\mathbin{\! \hbox{\vrule height0.3pt width5pt depth 0.2pt \vrule height5pt width0.4pt depth 0.2pt}}}

\newcommand{\Gt}{\mathrm{G}_2}
\newcommand{\lot}{\ell ot}
\newcommand{\real}{\mathrm{Re}}
\newcommand{\imag}{\mathrm{Im}}

\begin{document}

\renewcommand{\thepage}{\roman{page}}
\setcounter{page}{1}

\title{BRIDGES Lectures: \\ \, \\ Flows of geometric structures, \\ especially $\Gt$-structures}

\author{Spiro Karigiannis \\ Department of Pure Mathematics \\ University of Waterloo \\ \tt{karigiannis@uwaterloo.ca}}
\date{May 2025}

\maketitle

\newpage

\abstract{The \emph{BRIDGES meeting in gauge theory, extremal structures, and stability} was held in June 2024 at l'Institut d'\'Etudes Scientifiques de Carg\`ese in Corsica, France, organized by Daniele Faenzi, Eveline Legendre, Eric Loubeau, and Henrique S\'a Earp. The first week of the meeting was a summer school which consisted of four independent but related lecture series by Oscar Garc\'ia-Prada, Spiro Karigiannis, Laurent Manivel, and Ruxandra Moraru. The present document consists of notes for the lecture series by Spiro Karigiannis on ``Flows of geometric structures, especially $\Gt$-structures''. Some assistance in the preparation of these notes by the author was provided by several participants of the summer school. Details can be found in the acknowledgements section of the introduction.

The main theme is short time existence (STE) and uniqueness for geometric flows. We first introduce geometric structures on manifolds and geometric flows of such structures. We discuss some qualitative features of geometric flows, and consider the notions of strong and weak parabolicity. We focus on the Ricci flow, explaining carefully the DeTurck trick to establish short-time existence and uniqueness, an argument which we then extend to a general class of geometric flows of Riemannian metrics, previewing similar ideas for flows of $\Gt$-structures.

Finally, we consider geometric flows of $\Gt$-structures. We review the basics of $\Gt$-geometry and survey several different geometric flows of $\Gt$-structures. In particular, we clarify in what sense STE results for the $\Gt$ Laplacian flow differ from STE results for other geometric flows. We conclude with a summary of some recent results by the author with Dwivedi and Gianniotis, including a classification of all possible heat-type flows of $\Gt$-structures, and a sufficient condition for such a flow to admit STE and uniqueness by a modified DeTurck trick.}

\newpage

\dominitoc
\tableofcontents

\clearpage
\renewcommand{\thepage}{\arabic{page}}
\setcounter{page}{1}

\chapter{Introduction}

These lecture notes are a slightly expanded version of the lectures that the author gave in June 2024 at the \emph{BRIDGES meeting in gauge theory, extremal structures, and stability} at the Institut d'\'Etudes Scientifiques de Carg\`ese in Corsica. They are written in an informal style, and are aimed at young mathematicians who are just starting to learn about flows of geometric structures. The main theme is the idea of \emph{short time existence} (STE) and uniqueness, although several other aspects of the theory of geometric flows are briefly discussed in passing.

We begin with an introduction to the notion of a \emph{geometric structure} on a manifold, and of a \emph{geometric flow}. We discuss some qualitative features of geometric flows, and give a precise definition of the principal symbol and the property of strong parabolicity, which most geometric flows \emph{do not} possess. Strongly parabolic geometric flows, such as the harmonic map heat flow, enjoy short-time existence and uniqueness of solutions by standard PDE theory. We then focus on the Ricci flow of Riemannian metrics, and explain carefully the \emph{DeTurck trick} for demonstrating short-time existence and uniqueness by establishing a correspondence between the Ricci flow and a different flow which is strongly parabolic. This idea is then applied to a more general class of geometric flows of Riemannian metrics, as a precursor to similar ideas for flows of $\Gt$-structures.

Finally, we restrict our attention to $\Gt$-structures on a $7$-manifold, and geometric flows of $\Gt$-structures. We review the basic differential geometric properties of $\Gt$-structures, including the relation between curvature and torsion, and give a survey of several different geometric flows of $\Gt$-structures, including the $\Gt$ Laplacian flow. In particular, we clarify in what sense the STE results for the $\Gt$ Laplacian flow differ from the STE results for other geometric flows. We conclude with a summary of some of the results from the recent paper~\cite{DGK-flows2} by the author with Dwivedi and Gianniotis, including a classification of all possible heat-type flows of $\Gt$-structures, and a sufficient condition for such a flow to admit STE and uniqueness by a modified DeTurck trick.

While our bibliographic coverage of flows of $\Gt$-structures is fairly extensive, it is certainly not complete, as it omits much work on flows and solitons of flows on homogeneous manifolds. Moreover, our bibliographic coverage for other geometric flows is very very far from complete.

\textbf{Acknowledgements.} I gratefully acknowledge the hospitality of the Institut d'\'Etudes Scientifiques de Carg\`ese in Corsica, and in particular thank the organizers Daniele Faenzi, Eveline Legendre, Eric Loubeau, and Henrique S\'a Earp for putting together an excellent week of lectures for young researchers. For the preparation of these lecture notes, I am indebted to several participants of the BRIDGES meeting for their very careful proofreading of the manuscript, and for their many useful suggestions and recommendations for improving the clarity of these notes. I gratefully acknowledge Beatrice Brienza, Udhav Fowdar, Alfredo Llosa, Ra\'ul Gonz\'alez Molina, Andr\'es Moreno, Henrique S\'a Earp, and Jakob Stein for their contributions. In particular, much of the material appearing in the exercises for Chapters~\ref{chapter:structures} and~\ref{chapter:G2} was originally prepared by Henrique, Udhav, and Andr\'es for the problem sessions during the week in Corsica, and Jakob was very helpful in debugging the source code to get the exercise pictograms to display correctly. The BRIDGES meeting was supported by the \emph{BRIDGES collaboration: Brazil-France interactions in gauge theory, extremal structures and stability}, grants \#2021/04065-6, S\~{a}o Paulo Research Foundation (FAPESP) and ANR-21-CE40-0017, Agence Nationale de la Recherche (ANR). The author also thanks Anton Iliashenko and Alex Pawelko for spotting some minor errors in earlier versions of these notes.

\textbf{Notation.} Throughout these lectures, we let $M^n$ be a smooth oriented $n$-manifold. We always assume that $M$ is compact and without boundary. Compactness is not required for all results we discuss, but is essential for most of the analytic results so we just assume it throughout. If $x^1, \ldots, x^n$ are local coordinates on $M$, then we write $\partial_i$ for $\frac{\partial}{\partial x^i}$. We write $\partial_t$ for the differential operator $\frac{\partial}{\partial t}$. We use $\mathfrak{X}$ for the space $\Gamma(TM)$ of smooth vector fields, and $\Omega^k$ for the space $\Gamma(\Lambda^k T^* M)$ of smooth $k$-forms. We use $\cT^2$ for the space $\Gamma( T^*M \otimes T^* M)$ of smooth (covariant) $2$-tensors on $M$. We use $\cS^2$ for the space $\Gamma( S^2 T^* M)$ of smooth symmetric $2$-tensors, and $\cS^0$ for those which are trace-free with respect to a fixed metric $g$. In particular we have $\cS^2 = \langle g \rangle \oplus \cS^2_0$ and $\cT^2 = \cS^2 \oplus \Omega^2$, the splittings being (pointwise) orthogonal with respect to $g$. Here $\langle g \rangle = \{ f g : f \in \Omega^0 \}$.

We use $( e_1, \ldots, e_n )$ for the standard ordered basis of $\R^n$, and we always equip $\R^n$ with its standard positive definite inner product and standard orientation, for which $(e_1, \ldots, e_n)$ is an oriented orthonormal basis.

For $\alpha = (\alpha_1, \ldots, \alpha_n)$ a multi-index of nonnegative integers, we set $|\alpha| = \sum_{i = 1}^n \alpha_i$. If $\xi = (\xi_1, \ldots, \xi_n) \in \R^n$ then $\xi^{\alpha} := \xi_1^{\alpha_1} \cdots \xi_n^{\alpha_n}$, and if $f$ is a smooth function then
\begin{equation*}
   \partial^{\alpha} f := \frac{\partial^{|\alpha|}}{(\partial x^1)^{\alpha_1} \cdots (\partial x^n)^{\alpha_n}} f.
\end{equation*}

We use the Einstein convention throughout given a pair of matching upper/lower indices, for expressions involving local coordinates or a general local frame. Given a metric $g$ on $M$, we often, but not always, use the musical isomorphism determined by the metric to identify vector fields with $1$-forms, and more generally to identify type $(k, \l)$ tensors with (covariant) $(k+\l)$-tensors. In particular, whenever we use a local \emph{orthonormal frame} with respect to $g$, then all indices are subscripts, $g_{ij} = \delta_{ij}$, and repeated indices are summed.

Regarding our curvature conventions, if $g$ is a metric on $M$, then in local coordinates we have the Christoffel symbols
\begin{equation} \label{eq:Christoffel}
    \Gamma^k_{ij} = \frac{1}{2} g^{k \l} \left( \frac{\partial g_{i \l}}{\partial x^j} + \frac{\partial g_{j \l}}{\partial x^i} - \frac{\partial g_{i j}}{\partial x^{\l}} \right)
\end{equation}
yielding the Riemann curvature tensor $R_{ijk \l}$, the Ricci tensor $R_{jk}$, and the scalar curvature $R$ via
\begin{equation} \label{eq:curvatures}
\begin{aligned}
    R_{ijk\l} & = g_{m \l} R^m_{ijk} = g_{m \l} \left( \frac{\partial \Gamma^m_{jk}}{\partial x^i} - \frac{\partial \Gamma^m_{ik}}{\partial x^j} + \Gamma^m_{ia} \Gamma^a_{jk} - \Gamma^m_{ja} \Gamma^a_{ik}\right) \\
    & = g(\nabla_{\partial_i}(\nabla_{\partial_j}\partial_k)-\nabla_{\partial_j}(\nabla_{\partial_i}\partial_k), \partial_l), \\
    R_{jk} & = g^{i \l} R_{ijk \l}, \\
    R & = g^{jk} R_{jk}.
\end{aligned}
\end{equation}
We often use the fact that $\partial_m g^{ij} = - g^{ip} g^{jq} \partial_m g_{pq}$ which follows from differentiation of $g^{ij} g_{jk} = \delta^i_k$.

The \emph{divergence} on symmetric $2$-tensors with respect to the metric $g$ is the linear map $\Div \colon \cS^2 \to \Omega^1$ given by
\begin{equation*}
    (\Div h)_k = g^{pq} \nabla_p h_{qk}.  
\end{equation*}
Its formal adjoint is the linear map $\Div^* \colon \Omega^1 \to \cS^2$ given by
\begin{equation} \label{eq:divstar-defn}
X \mapsto \Div^* X = -\frac{1}{2} \cL_X g.
\end{equation}
In terms of a local frame, we have
\begin{equation} \label{eq:divstar}
    (\Div^* X)_{ij} = -\frac{1}{2} (\nabla_i X_j + \nabla_j X_i).   
\end{equation}

We have occasion to use the following ``diamond operator''. Let $\gamma \in \Omega^k$. Consider the linear map $\cT^2 \to \Omega^k$ given by $A \mapsto A \diamond \gamma$, where $A \diamond \gamma = \frac{d}{dt}\big|_{t=0} (e^{tA})^* \gamma$. This is the induced Lie algebra action of $\mathfrak{gl}(n, \R)$ on $\Omega^k$. Explicitly, we have
\begin{equation*}
    (A \diamond \gamma)(X_1, \ldots, X_k) = \gamma(AX_1, X_2, \ldots, X_k) + \gamma(X_1, AX_2, \ldots, X_k) + \cdots + \gamma(X_1, X_2, \ldots, AX_k).
\end{equation*}
In terms of a local orthonormal frame, this is
\begin{equation} \label{eq:diamond-defn}
    (A \diamond \gamma)_{i_1 \cdots i_k} = A_{i_1 m} \gamma_{mi_2 \cdots i_k} + A_{i_2 m} \gamma_{i_1 m i_3 \cdots i_k} + \cdots + A_{i_k m} \gamma_{i_1 \cdots i_{k-1} m}.
\end{equation}

We use $\triv{\R}^k_U$ to denote the trivial $\R^k$ vector bundle over $U$. Thus, if $E$ is an $\R^k$ vector bundle over $M$, then a local trivialization of $E$ over an open set $U \subset M$ is a vector bundle isomorphism $\phi_{U} \colon E|_U \to \triv{\R}^k_U$, which is equivalent to a \emph{local frame} $( s_1, \ldots, s_k )$ for $E$ over $U$ given by $s_j = \phi_{U}^{-1} \circ e_j$, where $e_j$ is the smooth section of $\triv{\R}^k_U$ given by $e_j (p) = (p, e_j)$.

\chapter{Geometric structures and geometric flows} \label{chapter:structures}

In this chapter we study \emph{geometric structures} and \emph{geometric flows} of such structures. So the first thing we need to do is understand what these words mean.

\section{Geometric Structures} \label{sec:structures}

The \emph{frame bundle} $\FR(M)$ of $M$ is the principal $\GL(n, \R)$-bundle of frames of the tangent bundle $TM$. What does this mean? Let $p \in M$, so $T_p M$ is the tangent space to $M$ at $p$. This is an $n$-dimensional real vector space, so the linear isomorphisms $\R^n \cong T_p M$ are in one-to-one correspondence with the set of (ordered) bases of $T_p M$. A \emph{point} in $\FR(M)$ corresponds to a pair $(p, \cF_p)$ where $\cF_p \colon \R^n \to T_p M$ is a linear isomorphism (equivalently, a \emph{frame} or \emph{basis} of $T_p M$) via the correspondence $\cF_p \leftrightarrow (\cF_p e_1, \ldots, \cF_p e_n)$. It is easy to see that $\GL(n, \R)$ acts smoothly on $\FR(M)$ via $A \cdot (p, \cF_p) = (p, \cF_p \circ A)$, and that this action is free and transitive on the fibres. This means that any two isomorphisms $\cF_p \colon \R^n \to T_p M$ and $\wt \cF_p \colon \R^n \to T_p M$ are related by a unique $A \in \GL(n, \R)$ such that $\cF_p = \wt \cF_p \circ A$.

In the language of vector bundles and transition functions, we say that the vector bundle $TM$ has ``structure group'' $\GL(n, \R)$, because we can cover $M$ by open sets $\{ U_{\alpha} \}$ over which we have \emph{local trivializations} $\phi_{\alpha}$ (equivalently, local frames for $TM$) such that the transition functions $\phi_{\beta} \circ \phi_{\alpha}^{-1}$ are vector bundle isomorphisms of $\triv{\R}^n_{U_{\alpha} \cap U_{\beta}}$ given by $p \mapsto (p, g_{\beta \alpha}(p))$ for some smooth map $g_{\beta \alpha} \colon U_{\alpha} \cap U_{\beta} \to \GL(n, \R)$.

If our manifold $M$ is equipped with additional ``structure'', then we can distinguish a \emph{preferred} class of frames at each $p \in M$. This is what we mean by a \emph{geometric structure}. Let's look at several examples.

\begin{ex} \label{ex:On}
Suppose $M$ is equipped with a Riemannian metric $g$. This is a smoothly varying choice of positive definite inner products on the fibres of $TM$. Then at each point $p \in M$, we can consider those frames which are \emph{orthonormal} with respect to the inner product $g_p$ on $T_p M$. This is equivalent to considering those isomorphisms $\cF_p \colon \R^n \to T_p M$ which are \emph{isometries}. Any two such frames are related by a unique element $A \in \O(n)$. Thus we can cover $M$ by open sets over which we have local trivializations of $TM$ where the transition functions lie in the proper subgroup $O(n)$ of $\GL(n, \R)$. We say that we have \emph{reduced the structure group} from $\GL(n, \R)$ to $\O(n)$.

It is a standard fact (using partitions of unity) than any manifold $M$ admits a Riemannian metric. That is, the tangent bundle of any manifold $M^n$ admits a reduction of structure group from $\GL(n, \R)$ to $\O(n)$.
\end{ex}

\begin{ex} \label{ex:GLplus}
Suppose $M$ is equipped with an \emph{orientation}. This is a smoothly varying choice of orientations on the fibres of $TM$. Then at each point $p \in M$, we can consider those frames which are \emph{oriented} with respect to this orientation on $T_p M$. This is equivalent to considering those isomorphisms $\cF_p \colon \R^n \to T_p M$ which are \emph{orientation preserving}. Any two such frames are related by a unique element $A \in \GL^+(n, \R)$. Thus we can cover $M$ by open sets over which we have local trivializations of $TM$ where the transition functions lie in the proper subgroup $\GL^+(n, \R)$ of $\GL(n, \R)$. We say that we have \emph{reduced the structure group} from $\GL(n, \R)$ to $\GL^+(n, \R)$.

In this case, such a reduction of structure group \emph{does not always exist}. In fact, an orientation on $M$ corresponds to an equivalence class $[\mu]$ of nowhere vanishing $n$-forms on $M$, where $\mu \sim \nu$ if and only if $\mu = f \nu$ for some \emph{positive} $f \in C^{\infty} (M)$. Thus an orientation exists if and only if the real line bundle $\Lambda^n (T^* M)$ is \emph{trivial}. This is a \emph{topological obstruction}. Such manifolds are called \emph{orientable}. In terms of this description, the ``oriented'' local frames $( X_1, \ldots, X_n )$ for $TM$ are those whose dual frames $( \alpha_1, \ldots, \alpha_n )$ for $T^* M$ satisfy $\alpha_1 \wedge \cdots \wedge \alpha_n = f \mu$ for some $f > 0$.
\end{ex}

\begin{ex} \label{ex:SOn}
If $M$ is orientable, we can simultaneously perform the reductions in Examples~\ref{ex:On} and~\ref{ex:GLplus} to obtain a reduction of structure group to $\GL^+ (n, \R) \cap \O(n) = \SO(n)$. Such a manifold $(M, g, [\mu])$ is an \emph{oriented Riemannian manifold} and in this case there is a unique representative $\vol$ of the class $[\mu]$ for which $\vol(X_1, \ldots, X_n) = 1$ for any oriented orthonormal frame $(X_1, \ldots, X_n)$, called the \emph{Riemannian volume form}.
\end{ex}

\begin{ex} \label{ex:GLC}
Suppose $M$ is equipped with an \emph{almost complex structure}. This is a smooth section $J \in \Gamma(TM \otimes T^* M) = \Gamma( \End TM )$ satisfying $J^2 = - I$. It allows us to identify each $T_p M$, which is a priori just a real vector space, with a \emph{complex vector space}, where the scalar multiplication by $\sqrt{-1}$ is given by $J_p$. Clearly we must have $n = 2m$ is even for this to have any hope of existing. However, even if $n=2m$, an almost complex structure need not exist. For instance, one can show that $M$ must necessarily be orientable (but this is not in general sufficient). This is because the standard inclusion of $\GL(m, \C)$ in $\GL(2m, \R)$ actually lies in $\GL^+ (2m, \R)$. (See Exercise~\ref{exer:GLmC} at the end of this chapter.) Existence of an almost complex structure is in general a very subtle topological question.

Given an almost complex structure $J$ on $M^{2m}$, at each point $p$ we can consider those frames $(X_1, \ldots, X_{2m})$ which are \emph{compatible} with this complex structure on $T_p M$, in the sense that they are of the form
\begin{equation*}
(X_1, J_p X_1, \ldots, X_m, J_p X_m)  
\end{equation*}
for some frame $(X_1, \ldots, X_m)$ of the \emph{complex} vector space $(T_p M, J_p)$. This is equivalent to considering those isomorphisms $\cF_p \colon \R^{2m} \to T_p M$ which are \emph{complex linear}, where $\R^{2m}$ is viewed as the underlying real vector space of $\C^m$. Any two such frames are related by a unique element $A \in \GL(m, \C)$. Thus we can cover $M$ by open sets over which we have local trivializations of $TM$ where the transition functions lie in the proper subgroup $\GL(m, \C)$ of $\GL(2m, \R)$. We say that we have \emph{reduced the structure group} from $\GL(2m, \R)$ to $\GL(m, \C)$.
\end{ex}

\begin{ex} \label{ex:Um}
Suppose $M^{2m}$ is equipped with an almost complex structure $J$. In particular, $M$ is orientable. Then we can always find a Riemannian metric $g$ on $M$ such that $g(JX, JY) = g(X, Y)$ for all $X, Y \in \Gamma(TM)$. This is done by replacing any metric $\wt{g}$ by $g = \wt g + J^* \wt g$. In this case, $J_p$ is an isometry with respect to $g_p$ for all $p \in M$. It is easy to see that $\omega (X, Y) = g(JX, Y)$ is a $2$-form on $M$, called the \emph{K\"ahler form} of $(M, g, J)$.

Thus given $J$, we can always further reduce to the structure group $\GL(m, \C) \cap \O(2m) = \U(m)$, the \emph{unitary} group of $\C^m$. Such a triple $(g, J, \omega)$ is called a $\U(m)$-structure on $M^{2m}$, and is also called an \emph{almost Hermitian structure}.
\end{ex}

\begin{ex} \label{G2}
A $7$-manifold $M^7$ is said to be equipped with a $\Gt$-structure if we can reduce the structure group from $\GL(7, \R)$ to the exceptional $14$-dimensional Lie group $\Gt \subset \SO(7)$. In particular, such a manifold also has both a metric and an orientation. It turns out that a $\Gt$-structure is completely encoded by a special kind of ``nondegenerate'' $3$-form $\ph$ on $M$. On a given $7$-manifold, a $\Gt$-structure exists if and only if $M$ is both orientable and \emph{spinnable}. (We do not assume the reader is familiar with spin geometry, so you don't need to know what this means.) A $\Gt$-structure on $M^7$ is essentially a smoothly varying identification of the tangent spaces of $M$ with the \emph{imaginary octonions}, together with their induced algebraic structure of the \emph{cross product} in seven dimensions, inner product, and orientation.
\end{ex}

We have seen that for a subgroup $G \subset \GL (n, \R)$, a $G$-structure on $M^n$ corresponds to a smoothly varying choice of point in each fibre of the frame bundle $\FR(M)$, subject to the fact that $(p,\mathcal{F}_p)$ is identified with $(p, \mathcal{F}_p \circ A)$ for any $A \in G$. Hence, specifying a $G$-structure is equivalent to choosing a smooth section of the quotient  $\FR(M)/G$, which is a fibre bundle over $M$ with typical fibre $\GL (n, \R) / G$. The fact that a $G$-structure may not exist on $M$ corresponds to the fact that the fibre fundle $\FR(M)/G$ may not have any global smooth sections.

We have now seen many examples of ``geometric structures'' on a manifold. Let $G$ be a Lie subgroup of $\GL(n, \R)$. When a manifold $M^n$ admits a $G$-structure, it usually admits \emph{many} $G$-structures. A natural question is: ``What is the \emph{best} $G$-structure on $M$?'' Of course, the answer depends entirely on what we mean by \emph{best}, which can depend on the context. Some examples:
\begin{itemize}
    \item A ``best'' Riemannian metric $g$ might be a metric with \emph{constant sectional curvature}. We know the topology of any manifold which admits a constant sectional curvature metric. Topologically, it must be a discrete quotient of either $S^n$ or $\R^n$, by the Killing--Hopf Theorem. Thus most manifolds cannot admit such ``best'' metrics, so this condition is too strong.
    
    At the other extreme, one can consider \emph{constant scalar curvature}. By the celebrated \emph{Yamabe problem}, solved by Rick Schoen, if $M$ is compact then any conformal class of metrics admits a constant scalar curvature representative. (See Lee--Parker~\cite{Lee-Parker} for an excellent survey.) So this condition is too weak.

    The natural intermediate condition is for the metric to be \emph{Einstein}. This means that $\tRc = \lambda g$ for some constant $\lambda$, and can be thought of as a type of ``constant Ricci curvature'' metric since $g$ is parallel.

    Note that each of these natural curvature conditions is a \emph{second order nonlinear partial differential equation} on the metric $g$.

    \item A ``best'' almost complex structure $J$ might be an \emph{integrable} almost complex structure, which means that its \emph{Nijenhuis tensor} $N$ vanishes. This is a $TM$-valued $2$-form on $M$ given by
    \begin{equation} \label{eq:Nijenhuis}
    N(X, Y) = [X, Y] + J[JX,Y] + J[X,JY] - [JX,JY].
    \end{equation}
    The celebrated Newlander-Nirenberg Theorem says that $J$ is integrable ($N=0$) if and only if there exists a \emph{holomorphic atlas} on $M$ inducing $J$, making $(M, J)$ into a \emph{holomorphic manifold}. (Most authors call this a ``complex manifold''.) This is a (nonlinear) first order PDE on $J$, but since there is no metric in this story, this particular example does not fit as easily into the same framework as the other examples. In this context it is worth pointing out that there are known examples of compact $4$-manifolds which admit almost complex structures but \emph{cannot} admit integrable complex structures. However, in dimension $2m \geq 6$, this is still an open question. Indeed, the most famous case is that of $S^6$. The $6$-sphere admits a ``canonical'' \emph{non-integrable} $J$, and it is unknown whether or not it admits an integrable $J$. It is a theorem~\cite{LeBrun} of Claude LeBrun that such a hypothetical integrable $J$ on $S^6$ cannot be compatible with the standard round metric.

    \item A ``best'' $\U(m)$-structure $(g, J, \omega)$ on $M^{2m}$ could be defined to be one for which $\nabla J = 0$, or equivalently $\nabla \omega = 0$. Such a manifold $(M, g, J, \omega)$ is called a \emph{K\"ahler} manifold, and the holonomy of its metric $g$ lies in $\U(m)$. Classically, and in algebraic geometry, it is more common to see the equivalent definition of K\"ahlerity as: $N = 0$ and $d \omega = 0$. More generally, for $m \geq 3$, the tensor $\nabla \omega$, called the \emph{torsion} of the $\U(m)$-structure, decomposes into four independent components under the Lie algebra action of $\U(m)$, and each of these four components could be zero or nonzero. This gives $2^4 = 16$ distinct ``classess'' of $\U(m)$-structure, called the Gray--Hervella classes. The most interesting non-K\"ahler $\U(m)$-structure is the case when $\nabla \omega$ is totally skew-symmetric, and hence equals $3 d\omega$. These are called \emph{nearly K\"ahler} manifolds, and when $2m=6$ they are intimately related to $\Gt$-structures.
    
    The curvature of a K\"ahler metric is significantly constrained, but there exist nevertheless a veritable zoo of distinct K\"ahler manifolds. Indeed, the Fubini-Study metric on complex projective space $\C \mathbb{P}^m$ is K\"ahler, and since a holomorphic submanifold of a K\"ahler manifold itself inherits a K\"ahler structure by restriction, it follows that any \emph{projective variety} has a natural K\"ahler structure. It thus makes sense to impose further restrictions, such as asking for a K\"ahler metric which is \emph{also} Einstein, naturally called a K\"ahler-Einstein metric. In particular, when the Einstein constant $\lambda = 0$, so $g$ is K\"ahler and Ricci-flat, then $g$ is called a \emph{Calabi--Yau} metric, and its holonomy is contained in $\SU(m)$. The \emph{Calabi-Yau Theorem} gives necessary and sufficient conditions for a compact K\"ahler  manifold $(M, g, J, \omega)$ to admit a Ricci-flat K\"ahler metric $\tilde g$ whose associated K\"ahler form $\tilde \omega ( \cdot, \cdot ) = \tilde g (J \cdot, \cdot)$ lies in the same cohomology class as $\omega$. The reader is referred to textbooks on K\"ahler geometry, such as Huybrechts~\cite{Huybrechts} or Moroianu~\cite{Moroianu}. (See the exercises at the end of this chapter for a study of $\SU(2)$-structures and their torsion.)

    \item Similarly we will see that if $\ph$ is a $\Gt$-structure on $M^7$, then $\nabla \ph$, called the \emph{torsion} of the $\Gt$-structure, decomposes into four independent components, so there are also $16$ classes of $\Gt$-structures. When $\nabla \ph = 0$, we say that $\ph$ is \emph{torsion-free}.
\end{itemize}

Let $\alpha$ be a covariant $k$-tensor corresponding to a $G$-structure on $M$, where $G \subset \SO(n)$. Then the \emph{torsion} of $\alpha$ is defined to be $\nabla \alpha$, and we denote it by $T$. Then we have $\nabla_X \alpha = T_X$, so $\nabla_X (\nabla_Y \alpha) = \nabla_X (T_Y) = (\nabla_X T)_Y + T_{\nabla_X Y}$. It follows that
\begin{equation*}
    R^{\nabla}(X, Y) \alpha = \nabla_X (\nabla_Y \alpha) - \nabla_Y (\nabla_X \alpha) - \nabla_{[X, Y]} \alpha = (\nabla_X T)_Y - (\nabla_Y T)_X.
\end{equation*}
Such a relationship between the Riemann curvature $R^{\nabla}$ of $g$ and the torsion $T$ of the $G$-structure $\alpha$ is sometimes called a \emph{$G$-Bianchi identity}, because it can also be proved by considering the linearization of the diffeomorphism invariance of the torsion of the $G$-structure, in the same way that the classical Bianchi identities of Riemannian geometry arise from the diffeomorphism invariance of the curvature. (See Chow--Knopf~\cite[Chapter 3, Sec.\ 2.2]{Chow-Knopf}.)

\begin{remark} \label{rmk:general-torsion}
If $G \subset \SO(n)$, we can give an equivalent definition of the torsion of a $G$-structure, as follows. The metric gives us the identification $\Lambda^2 (\R^n) = \fs\fo(n)$, so from the Lie algebra inclusion $\mathfrak{g} \subset \mathfrak{so}(n)$ we get an orthogonal decomposition $\Lambda^2 (\R^n) = \mathfrak{g} \oplus \mathfrak{g}^{\perp}$. Let $v_1, \ldots, v_n$ be a local orthonormal frame. The Levi-Civita connection $\nabla$ corresponds to an $\mathfrak{so}(n)$-valued $1$-form $A$ via $\nabla v_i = A_{ij} v_j$. The fact that $A$ is skew-symmetric follows from $\nabla g = 0$, since 
\begin{equation*}
0 = \nabla (g(v_j, v_k)) = g(\nabla v_j, v_k) + g(v_j, \nabla v_k) = g( A_{jl} v_l, v_k) + g(v_j, A_{kl} v_l) = A_{jk} + A_{kj}. 
\end{equation*}
Using the decomposition of $\mathfrak{so}(n) = \mathfrak{g} \oplus \mathfrak{g}^{\perp}$, we can write $A = \widehat{A} + T$, where $\widehat{A}$ is a $\mathfrak{g}$-valued $1$-form and $T$ a $\mathfrak{g}^\perp$-valued $1$-form. This splitting is in fact global. That is, $\nabla = \widehat{\nabla} + T$, where $\widehat{\nabla}$ is the projection of the Levi-Civita connection to $\mathfrak{g}$. This is sometimes called the canonical $G$-connection, although in~\cite{DGK-flows2} it is called the \emph{optimal} $G$-connection. The tensor $T$ is called the \emph{intrinsic torsion} of the $G$-structure.

Suppose $G$ is precisely the stabilizer of a finite set of differential forms $\gamma_1, \ldots, \gamma_N$. Since $P^* \gamma_a = \gamma_a$ for all $P \in G$, it follows that $\widehat{A} \diamond \gamma_a = 0$ in the notation of~\eqref{eq:diamond-defn}, so $\widehat{\nabla} \gamma_a = 0$. In particular, we have
\begin{equation*}
\nabla_X \gamma_a = \widehat{\nabla}_X \gamma_a + T(X) \diamond \gamma_a = T(X) \diamond \gamma_a.
\end{equation*}
The map $B \mapsto (B \diamond \gamma_1, \ldots, B \diamond \gamma_N)$ is injective on $\mathfrak{g}^{\perp}$ by construction, as its kernel is exactly $\mathfrak{g}$. It follows that $\nabla \gamma_a = 0$ for all $1 \leq a \leq N$ if and only if $T = 0$. Thus the intrinsic torsion as defined here is essentially equivalent to the torsion in the examples of $G$-structures discussed above. (For example, $N=1$ for both $G = \U(m)$ and $G = \Gt$, where we have $\gamma_1 = \omega$ or $\gamma_1 = \ph$, respectively.)
\end{remark}

\section{Geometric Flows} \label{sec:flows}

We have seen what a geometric structure is, which from now on will be a $G$-structure for some $G \subset \SO(n)$, and we have briefly discussed how some $G$-structures might be considered ``better'' than others. How do we find a ``best'' $G$-structure on $M$? There are two approaches, which can (very informally and somewhat imprecisely) be referred to as the \emph{elliptic} and \emph{parabolic} approaches. These are roughly as follows:
\begin{itemize}
    \item In the \emph{elliptic} approach, we attempt to directly solve the equation that determines our notion of ``best'' $G$-structure. This is usually a second order nonlinear PDE for the unknown $G$-structure $\alpha$. In best cases, it is an \emph{elliptic} equation and many tools from geometric analysis are available to study such problems, including: Sobolev and H\"older spaces, elliptic regularity, the Fredholm alternative, and the continuity method. Some examples of results of this flavour (not all stated in their full generality or completely precisely) include the following (these will not necessarily all make sense to you):
    \begin{itemize}
        \item The resolution of the Yamabe problem: on a compact manifold, any conformal class of Riemannian metrics admits a representative with constant scalar curvature. (See Aubin~\cite{Aubin} or Lee--Parker~\cite{Lee-Parker}, for example.)
        \item The Calabi--Yau Theorem: on a compact K\"ahler manifold $M$, each K\"ahler class admits a unique Ricci-flat K\"ahler metric if and only if the first Chern class of $M$ vanishes. (See Yau~\cite{Yau}, Aubin~\cite{Aubin}, or Joyce~\cite{Joyce} for more details.)
        \item The Donaldson--Uhlenbeck--Yau Theorem: let $E$ be a holomorphic vector bundle over a compact K\"ahler manifold. Then $E$ admits a Hermitian--Einstein connection if and only if $E$ is (semi-)stable. (See Donaldson--Kronheimer~\cite{DK}, for example.)
        \item Given a closed $\Gt$-structure $\ph$ on $M^7$, which is sufficiently close to torsion-free, then there exists a unique torsion-free $\Gt$-structure $\wt\ph$ in the same cohomology class and close to $\ph$. (See Joyce~\cite{Joyce}, for example.)
    \end{itemize}
    \item In the \emph{parabolic} approach, we start with a $G$-structure $\alpha$ that does \emph{not} satisfy our ``best'' criterion, and attempt to \emph{evolve or flow} it in the space of $G$-structures, with the hope that it constantly gets \emph{closer and closer} to being ``best''. This is usually a (weakly) parabolic nonlinear PDE which qualitatively exhibits \emph{diffusive phenomena} similar to the standard heat equation. This essentially means that it tends to spread out and dissipate any deviations from ``best''. Typical tools that are used to study such problems include: the maximum principle, Harnack inequalities, Shi-type estimates, and soliton solutions to model formation of singularities. (We will not make most of these ideas more precise in these lectures.) Some well-known examples of geometric flows which have been studied, together with some suggested references for further reading, are:
    \begin{itemize}
        \item The harmonic map heat flow~\cite{Eells-Sampson} of smooth maps between Riemannian manifolds.
        \item The Ricci flow~\cite{Chow-Knopf, CLN, Topping} of Riemannian metrics.
        \item The mean curvature flow~\cite{Zhu} (or Lagrangian mean curvature flow) of submanifolds.
        \item Various curvature flows in Hermitian geometry, such as the K\"ahler--Ricci flow~\cite{Cao}, the Hermitian curvature flow, and the anomaly flow.
        \item The Yang--Mills gradient flow~\cite{DK} of connections.
        \item The $\Gt$ Laplacian flow~\cite{Lotay-survey} of $\Gt$-structures.
    \end{itemize}
\end{itemize}
The goal of these lectures is to present the reader with a gentle introduction to some of the ideas involved in flowing $G$-structures on manifolds, using flows of Riemannian metrics (especially the Ricci flow) and flows of $\Gt$-structures as the main examples. The main question we focus on is the following:
\begin{itemize}
    \item Does the flow have short-time existence (STE) and uniqueness? This will lead us to consider the \emph{DeTurck trick} for proving STE for the Ricci flow and some other flows.
\end{itemize}
There are many other questions about geometric flows that are of supreme importance, which we discuss only in passing. Such questions include:
\begin{itemize}
    \item Does the flow have long-time existence (LTE)? For most (but not all) geometric flows this is in general \emph{not} the case. That is, \emph{most geometric flows develop singularities in finite time}.
    \item What \emph{characterizes} finite-time singularities? That is, what is happening geometrically to the manifold as we approach the singular time $\tau$? Usually some scalar quantity which is a norm of some tensor, is blowing up to infinity.
    \item Can we say more about singularities? We can \emph{blow up} at the singular point and time. That is, we ``zoom in'' by rescaling in space and translating in time. (This is not to be confused with blowing up in algebraic geometry.) Then in many cases, what we observe is a \emph{soliton} solution to the flow, which is a solution that is evolving by rescaling and diffeomorphisms. Thus, knowledge about existence and properties of soliton solutions tells us something about possible singularity formation. 
    \item If the flow has LTE, does it converge? The flow may exist for all time, but that does not mean that it gets closer and closer to a limit as $t \to \infty$.
    \item Does the flow exhibit \emph{stability}? That is, if we know for some reason that there exists a ``best'' $G$-structure $\ol{\alpha}$ sufficiently close to $\alpha_0$, will the flow starting at $\alpha_0$ exist for all time and converge to $\ol{\alpha}$, or at least to a ``best'' $G$-structure somehow equivalent to $\ol{\alpha}$, such as in the same diffeomorphism orbit?
\end{itemize}

\section{Some qualitative discussion of flows} \label{sec:qualitative}

Most reasonable geometric flows exhibit diffusive type qualitative behaviour, similar to the classical \emph{heat equation} on the Euclidean space $\R^n$, which is
\begin{equation*}
    \partial_t u = \Delta u.
\end{equation*}
Here $u \colon \R^n \times [0, \infty) \to \R$, and $\Delta u = \sum_{k=1}^n \frac{\partial^2 u}{\partial x_k \partial x_k}$ is the \emph{Laplacian} operator.

Let's try to qualitatively understand why the heat equation tends to \emph{diffuse} the function $u$, by looking at the case $n=1$. In this case the heat equation is $u_t = u_{xx}$. If $p \in \R$ is a point at which $u$ attains a local minimum in space, then $u_{xx} (p) \geq 0$, so $u_t(p) \geq 0$ and $u(p)$ will \emph{increase} in time. Similarly if $q \in \R$ is a point at which $u$ attains a local maximum in space, then $u_{xx} (q) \leq 0$, so $u_t(q) \leq 0$ and $u(q)$ will \emph{decrease} in time. Thus, the function $u(\cdot, t)$ should approach the average value of the initial function $u(\cdot, 0)$ as $t \to \infty$. This can of course be made rigourous.

Suppose $F \colon \R^k \to \R^k$ is a smooth function. We know from the theory of \emph{ordinary differential equations} that, given an initial condition $u(0) = u_0$, then \emph{there exists a unique solution} $u \colon [0, \infty)] \to \R^k$ to the first order ODE $\frac{\partial}{\partial t} u = F(u(t))$, at least for some short time $t \in [0, \eps)$. This is sometimes called the Picard--Lindel\"of Theorem, or the \emph{fundamental theorem of ODE}. Moreover, when $F$ is \emph{linear}, then we get \emph{existence for all time}. That is, we can take $\eps = \infty$. It is absolutely crucial, however, that \emph{in general long time existence fails for nonlinear equations}. The most elementary example is $\frac{\partial}{\partial t} u = u^2$, which has solution $u(t) = \frac{u_0}{1 - u_0 t}$, which exists for all time if $u_0 \leq 0$ but only for $t \in [0, u_0)$ if $u_0 > 0$. We say that $u$ \emph{blows up in finite time}. This is very important is because
\begin{equation*}
    \text{\emph{most natural geometric flows are \emph{nonlinear}, and hence in general do not have long time existence.}}
\end{equation*}
Moreover, the reason most natural geometric flows are nonlinear is because they usually depend on some notion of \emph{curvature}, and in general the various notions of curvature that arise in geometry are nonlinear second order partial differential expressions.

Thus consider, for example, the nonlinear PDE on $\R^1$ given by
\begin{equation*}
    u_t = u_{xx} + u^2,
\end{equation*}
which is a perturbation of the usual heat equation by the addition of a ``nonlinear reaction term''. Analyzing as above, at a point $p \in \R$ where $u$ attains a local min, we have $u_{xx} (p) \geq 0$ and of course $u(p)^2 \geq 0$ as well, so at such points $u(p)$ will increase in time. However, at a point $q \in \R$ where $u$ attains a local max, we have $u_{xx} (q) \leq 0$ but $u(q)^2 \geq 0$. So it is not clear what happens at such a point. It depend on ``which of the two terms wins the fight''.

\section{The ``heat equation'' on Riemannian manifolds} \label{sec:heat}

Before looking at other geometric flows, let's first consider the generalization of the classical heat equation to an arbitrary Riemannian manifold $(M, g)$. There are \emph{many} such generalizations, which schematically often take the form
\begin{equation*}
    \partial_t u = \Delta u,
\end{equation*}
where $u$ is some kind of geometric object (usually a section of some tensor bundle), and $\Delta$ is some kind of \emph{Laplacian}. There are many notions of Laplacian.

\begin{ex} \label{ex:LB}
The \emph{Laplace-Beltrami operator} on functions is given by
\begin{equation*}
   \Delta u = \Div (\grad u),
\end{equation*}
where $\grad u$ is the \emph{gradient} of the function $u$, which is the vector field metric dual to the $1$-form $du$, and $\Div X$ is the \emph{divergence} of the vector field $X$, which is the function determined by $\cL_X \vol = (\Div X) \vol$. If $g_{ij}$ are the components of the metric tensor with respect to some local coordinates $x^1 \ldots, x^n$, with $|g| = \det (g_{ij})$, then $\grad u = (du)^{\sharp} = g^{ij} \frac{\partial u}{\partial x^i} \frac{\partial}{\partial x^j}$ and $\Div X = \frac{1}{\sqrt{|g|}} \frac{\partial}{\partial x^i} \left( \sqrt{|g|} X^j \right)$, so
\begin{equation} \label{eq:LB-coords1}
    \Delta u = \Div (\grad u) = \frac{1}{\sqrt{|g|}} \frac{\partial}{\partial x^i} \left( \sqrt{|g|} \frac{\partial u}{\partial x^j} \right).
\end{equation}
Using the Chrisfoffel symbols $\Gamma^k_{ij}$ for $g$ from~\eqref{eq:Christoffel}, equation~\eqref{eq:LB-coords1} can also be written in the useful form
\begin{equation} \label{eq:LB-coords2}
    \Delta u = g^{ij} \left( \frac{\partial^2 u}{\partial x^i \partial x^j} - \Gamma^k_{ij} \frac{\partial u}{\partial x^k} \right).
\end{equation}
Note that in Riemannian normal coordinates centred at a point $p \in M$, we have $g_{ij}(p) = \delta_{ij}$,  $\frac{\partial g_{ij}}{\partial x^k}(p) = 0$, and $\Gamma^k_{ij} (p) = 0$, from which it follows from either expression above that $(\Delta u)(p) = \sum_{k=1}^n \frac{\partial u}{\partial x^k \partial x^k} (p)$ as for the Euclidean metric on $\R^n$.
\end{ex}

\begin{ex} \label{ex:AL}
Let $E$ be a vector bundle on $M$ equipped with a fibre metric $h$, and let $\nabla$ be a connection on $E$ compatible with $h$. The \emph{connection Laplacian} or \emph{rough Laplacian} or \emph{analyst's Laplacian} $\Delta$ on sections of $E$ is given by
\begin{equation*}
    \Delta = - \nabla^* \nabla,
\end{equation*}
where $\nabla^* \colon \Gamma(T^* M \otimes E) \to \Gamma (E)$ is the formal adjoint of $\nabla \colon \Gamma (E) \to \Gamma(T^* M \otimes E)$. In terms of a local frame $v_1, \ldots, v_n$ of $TM$, we have
\begin{equation} \label{eq:AL}
    \Delta s = - \nabla^* \nabla s = g^{ij} \nabla_i \nabla_j s
\end{equation}
where $\nabla_i := \nabla_{v_i}$. (Note that many authors use $\nabla^* \nabla$, which is a \emph{positive} operator, when they say rough Laplacian or connection Laplacian. But everyone agrees that the ``analyst's Laplacian'' has a minus sign, which is a \emph{negative} operator. The reader must always be careful to identify the sign convention for a Laplacian, similar to the problem of differing sign conventions for the Riemann curvature tensor.)
    
We remark that when $E$ is the trivial real line bundle over $M$, then the rough Laplacian equals the Laplace-Beltrami operator.
\end{ex}

\begin{ex} \label{ex:HL}
Let $d \colon \Omega^k \to \Omega^{k+1}$ be the exterior derivative, and let $d^* \colon \Omega^k \to \Omega^{k-1}$ be its formal adjoint. (One can show that $d^* = (-1)^{nk+n+1} \star d \star$ on $\Omega^k$, where $\star$ is the Hodge star operator.) The \emph{Hodge Laplacian} $\Delta_d$ is defined to be
\begin{equation*}
   \Delta_d = d d^* + d^* d.
\end{equation*}
The map $\Delta_d \colon \Omega^k\to \Omega^k$ is linear and formally self-adjoint. A $k$-form $\alpha$ is called \emph{harmonic} if $\Delta_d \alpha = 0$. It is easy to see that on a compact manifod, $\alpha$ is harmonic if and only if $d \alpha = 0$ and $d^* \alpha = 0$.

The classical \emph{Weitzenb\"ock formula} reveals that
\begin{equation} \label{eq:BW}
   \Delta_d = \nabla^* \nabla + \cR, 
\end{equation}
where $\cR \colon \Omega^k (M) \to \Omega^k (M)$ is a linear operator constructed from the Riemann curvature tensor of $g$ which is \emph{algebraic}, in the sense that it does not involve any differentiation. When $k=0$ this curvature term is zero, and thus on $0$-forms (which are functions) we have $\Delta_d = \nabla^* \nabla$ is the negative of the Laplace-Beltrami operator or analyst's Laplacian.
\end{ex}

There are many other Laplacians, such as the \emph{Lichnerowicz Laplacian} on symmetric $2$-tensors (which will play a role later in these lectures), and the \emph{Dirac Laplacian} associated to the Dirac operator on a spinor bundle, just to name a few.

One of the most fundamental results in Riemannian geometry is the classical \emph{Hodge Theorem}. Indeed, it is the author's opinion that the Hodge Theorem is the best introduction to \emph{geometric analysis}, especially since there are both \emph{elliptic and parabolic} proofs of this theorem. The precise statement is:
\begin{theorem}[The Hodge Theorem]
Let $\cH^k = \ker (\Delta_d|_{\Omega^k})$ be the vector space of harmonic $k$-forms on $M$. There is an $L^2$-orthogonal decomposition
\begin{equation*}
    \Omega^k = \cH^k \oplus d (\Omega^{k-1}) \oplus d^* (\Omega^{k+1}).
\end{equation*}
Moreover, we have
\begin{align*}
    \cH^k = \ker (d|_{\Omega^k}) \cap \ker (d^*|_{\Omega^k}), \quad \ker (d|_{\Omega^k}) = \cH^k \oplus d (\Omega^{k-1}), \quad \ker (d^*|_{\Omega^k}) = \cH^k \oplus d^* (\Omega^{k+1}).
\end{align*}
It follows that every de Rham cohomology class in $H^k (M, \R)$ has a \emph{unique harmonic representative} and that there is an isomorphism $\cH^k \cong H^k (M, \R)$ given by $\alpha \mapsto [\alpha]$.
\end{theorem}
\begin{proof}[Idea of proof]
See Warner~\cite{Warner} for a proof using elliptic methods. For our purposes, we are more interested in a proof using parabolic methods (that is, via a geometric flow.) Such a proof can be found in Rosenberg~\cite{Rosenberg}. The idea is to consider the equation
\begin{equation} \label{eq:heat-forms}
    \partial_t \alpha = - \Delta_d \alpha,
\end{equation}
which is a type of \emph{heat equation} on forms. From the Weitzenb\"ock formula~\eqref{eq:BW}, the right hand side of the above equation equals, up to lower order terms, the analyst's Laplacian $- \nabla^* \nabla$, so we should expect this equation to behave qualitatively like the classical heat equation.

One shows that~\eqref{eq:heat-forms} has \emph{long time existence and convergence}, and that if the initial $k$-form $\alpha_0$ is closed, then $[\alpha_t] = [\alpha_0]$ for all $t \geq 0$. That is, the de Rham cohomology class is preserved. In the limit as $t \to \infty$, the solution $\alpha(t)$ converges to the unique harmonic $k$-form in the class $[\alpha_0]$.
\end{proof}

We remark that the Hodge Theorem gives a proof that the de Rham cohomology of a compact orientable manifold is finite-dimensional, by using the fact that the kernel of an elliptic operator on a compact manifold is always finite-dimensional.

While the above situation is not an example of a $G$-structure, the Hodge Theorem \emph{does} give an excellent example of a situation where we ask for a ``best'' type of geometric object. In this case, given a de Rham cohomology class in $H^k (M, \R)$, we ask for the ``best representative'', which given a choice of metric $g$, is the unique \emph{harmonic} representative. One can also show easily that the harmonic representative has the smallest $L^2$-norm amongst all representatives.

It is worth making the following important remark about the geometric flow approach to the Hodge Theorem: the equation~\eqref{eq:heat-forms} is a \emph{linear} partial differential equation. Our metric $g$ is fixed, and the Laplacian is determined by $g$. This is (morally, at least) why we can get long time existence. By contrast, we will see for example that the Ricci flow can be thought of as a type of heat flow, where the Laplacian is changing in time, because it depends on the (evolving) metric $g(t)$. In this sense the Ricci flow (and most geometric flows) are \emph{nonlinear}, and thus as we remarked above, one should not expect to have long time existence in general. (Although it does sometimes happen, as we will discuss later.)

\section{Principal symbols of linear differential operators} \label{sec:symbols}

As we discussed earlier, if we want to use a \emph{geometric evolution equation} or \emph{geometric flow} to evolve some geometric object (such as a $G$-structure) into (ideally) a ``best'' one, then a natural type of equation is one which is a ``heat type'' of equation. Before we can make this precise, we need to discuss \emph{linear differential operators} and their \emph{principal symbols}.

Let $E, \wt{E}$ be real vector bundles of rank $k, \tilde k$ respectively over the Riemannian manifold $(M, g)$, and suppose $E, \wt E$ are both equipped with fibre metrics.

\begin{definition} \label{defn:LDO}
Let $P \colon \Gamma(E) \to \Gamma(\wt{E})$ be a \emph{linear} map. We say that $P$ is a \emph{differential operator} of order $m \geq 0$ if, whenever $\{ s_1, \ldots, s_k \}$ is a local frame for $E$, $( \tilde s_1, \ldots, \tilde s_{\tilde k} )$ is a local frame for $\wt E$, and $x^1, \ldots, x^n$ are local coordinates for $M$, we have
\begin{equation} \label{eq:LDO}
    P ( f^b s_b ) = \sum_{|\alpha| \leq m} P_{\alpha b}^c \partial^{\alpha} f^b \tilde s_c,
\end{equation}
where $1 \leq b \leq k$ and $1 \leq c \leq \tilde k$. Here the $P_{\alpha b}^c$ are smooth functions of $x^1, \ldots, x^n$ such that at least one $P_{\alpha b}^c$ with $|\alpha| = m$ is nonzero. (It is an easy exercise to verify that this characterization is independent of the choices of local coordinates on $M$ and local frame on $E$.)
\end{definition}

\begin{definition} \label{defn:symbol}
Let $P \colon \Gamma(E) \to \Gamma(\wt{E})$ be a linear differential operator of order $m$ as in Definition~\ref{defn:LDO}. Let $\xi = \xi_i dx^i \in T^*_p M \setminus \{ 0 \}$. The \emph{principal symbol} $\sigma_m (P) (\xi)$ of $P$ at $\xi$ is the linear endomorphism of the fibre $E_p$ given by
\begin{equation} \label{eq:symbol}
    \big( \sigma_m (P) (\xi) \big) (f^b s_b) = \sum_{|\alpha| = m} P_{\alpha b}^c \xi_1^{\alpha_1} \cdots \xi_n^{\alpha_n} f^b \tilde s_c.
\end{equation}
That is, in $P$ we replace the operator $\partial_i$ with multiplication by $\xi_i$, and keep only the highest order terms where $|\alpha| = m$. It is a good exercise using the chain rule to verify that this is well-defined, independent of the choice of local coordinates on $M$ and a local frame for $E$. For this it is \emph{crucial} that we keep only the highest order terms. The same ``operation'' on the full expression of $P$ is \emph{not} well-defined. This is an illustration of the fact that the highest order terms in a linear differential operator are (morally) the only terms that contribute to the qualitative behaviour of the differential equation $P s = 0$ or the evolution equation $\partial_t s = Ps$.
\end{definition}

Many authors use the convention (which comes from the relation to the Fourier transform) that we replace $\partial_i$ with $\sqrt{-1} \xi_i$. We won't do this in these lectures. Also note that one can express the notion of being a linear partial differential operator of order $m$ in terms of an arbitrary connection $\nabla$ on $E$, by replacing all the $\partial_i$ in~\eqref{eq:LDO} by $\nabla_i := \nabla_{\partial_i}$. It is easy to see that this is equivalent and gives the same notion of principal symbol.

\begin{remark} \label{rmk:symbol-coordinate-free}
It is possible to give a coordinate-free definition of the principal symbol, as follows. Given $\xi \in T^*_p M \setminus \{ 0 \}$, let $\eta \in C^{\infty} (M)$ be such that $(d \eta)_p = \xi$. Then it is an easy exercise to show that at the point $p$, we have
\begin{equation*}
    \lim_{r \to \infty} r^{-m} e^{-r \eta} P( e^{r \eta} s) = (\sigma_m (P) (\xi)) s,
\end{equation*}
independent of the choice of $\eta$ with this property.
\end{remark}

\begin{remark} \label{rmk:symbol-comp}
Suppose $P_1 \colon \Gamma(E_1) \to \Gamma(E_2)$ and $P_2 \colon \Gamma(E_2) \to \Gamma(E_3)$ are linear differential operators of orders $m_1$, $m_2$ respectively. Then it is easy to check, using the chain rule, that $P_2 \circ P_1 \colon \Gamma(E_1) \to \Gamma(E_3)$ is a linear differential operator of order \emph{at most} $m_1 + m_2$, and moreover that
\begin{equation} \label{eq:symbol-comp}
    \sigma_{m_1 + m_2}(P_2 \circ P_1) (\xi) = \sigma_{m_2} (P_2) (\xi) \circ \sigma_{m_1} (P_1) (\xi).
\end{equation}
That is, \emph{the principal symbol of the composition is the composition of the principal symbols}. Equation~\eqref{eq:symbol-comp} still holds when $P_2 \circ P_1$ is of lower order than $m_1 + m_2$, in the sense that the right hand side of~\eqref{eq:symbol-comp} must vanish. This observation is crucial for understanding \emph{diffeomorphism invariance} of geometric flows in Section~\ref{sec:diff-inv}.
\end{remark}

\section{Strongly parabolic flows} \label{sec:parabolic}

In this section we make precise what we mean by a ``heat type'' equation (or \emph{strongly parabolic} equation). More details can be found in Topping~\cite{Topping}.

\begin{definition} \label{defn:parabolic}
Let $P \colon \Gamma(E) \to \Gamma(E)$ be a linear \emph{second order} differential operator. The evolution equation
\begin{equation*}
   \partial_t s = Ps 
\end{equation*}
is called \emph{strongly parabolic} (sometimes also called \emph{strictly parabolic} or just \emph{parabolic}) if and only if there exists $C > 0$ such that
\begin{equation} \label{eq:parabolic}
    \langle \sigma_2 (P) (\xi) s, s \rangle \geq C \, |\xi|^2 \, |s|^2
\end{equation}
for all $\xi \in \Omega^1$ and all $s \in \Gamma(E)$. In particular, this says that for every $p \in M$ and every $\xi \in T_p^* M \setminus \{ 0 \}$, the principal symbol $\sigma_2 (P) (\xi)$ has the property that its self-adjoint part is positive, and  that this positivity holds \emph{uniformly} on $M$.
\end{definition}

\begin{ex} \label{ex:symbol-Lap}
Let $P = - \nabla^* \nabla$ be the analyst's Laplacian on $\Gamma(E)$ of Example~\ref{ex:AL}. Then
\begin{equation*}
    \sigma_2 (P) (\xi) = g^{ij} \xi_i \xi_j \rI = |\xi|^2 \rI,
\end{equation*}
where $\rI$ is the identity operator. Thus~\eqref{eq:parabolic} is satisfied with $C = 1$, so $\partial_t s = - \nabla^* \nabla s$ is a parabolic equation. By the Weitzenb\"ock formula~\eqref{eq:BW}, the same is true of the heat flow~\eqref{eq:heat-forms} on forms which is used to prove the Hodge Theorem.
\end{ex}

We now remove the overly restrictive assumption of \emph{linearity}, but we require $P$ to at least be \emph{quasilinear}, which means it is linear in the highest order derivatives. The precise statement is the following.

\begin{definition} \label{defn:QLDO}
Let $\cU$ be an open set in $\Gamma(E)$ with respect to the $C^{\infty}$ topology, and let $P \colon \cU \to \Gamma(E)$ be a map. Let $s \in \Gamma(E)$. We say that $P$ is a \emph{nonlinear but quasilinear} differential operator of order $m \geq 0$ if, whenever $\{ s_1, \ldots, s_k \}$ is a local frame for $E$ and $x^1, \ldots, x^n$ are local coordinates for $M$, we have
\begin{equation} \label{eq:QLDO}
    P ( f^b s_b ) = \sum_{|\alpha| = m} P_{\alpha b}^c \partial^{\alpha} f^b s_c + Q^b s_b,
\end{equation}
where the $P_{\alpha b}^c$ and $Q^b$ are smooth functions of $x^1, \ldots, x^n$ \emph{and all partial derivatives of $f^1, \ldots, f^k$ up to order $m-1$}, such that at least one $P_{\alpha b}^c$ is nonzero. (Again, it is easy to check that this characterization is independent of the choices of local coordinates on $M$ and local frame on $E$.)
\end{definition}

\begin{ex} \label{ex:curvature-QL}
By~\eqref{eq:curvatures}, the Riemann curvature tensor, Ricci curvature tensor, and scalar curvature tensor are all examples of second order nonlinear but \emph{quasilinear} differential operators on the space of Riemannian metrics on $M$, which is an open set $\cU$ in $\Gamma (S^2 (T^* M))$ with respect to the $C^{\infty}$ topology~\cite{GMM}.
\end{ex}

Let $P \colon \cU \to \Gamma(E)$ be a \emph{nonlinear but quasilinear} differential operator of order $m$. The \emph{linearization} of $P$ at $s \in \cU$, which is a \emph{linear} map $D_s P \colon \Gamma(E) \to \Gamma(E)$, is defined to be
\begin{equation*}
(D_s P) u = \left. \frac{d}{d \eps} \right|_{\eps=0} P(\gamma(\eps))
\end{equation*}
for \emph{any} smooth curve $\gamma$ in $\cU$ satisfying $\gamma(0) = s$ and $\gamma'(0) = u$. (Often, but not always, it is convenient to take $\gamma(\eps) = s + \eps u$.) Let us consider three examples which will be important later.

\begin{ex} \label{ex:linearization-Ricci}
Let $\cU$ denote the open subset (with respect to the $C^{\infty}$ topology) of $\cS^2 = \Gamma(S^2 (T^*M ))$ consisting of the positive definite sections. (That is, $\cU$ is the space of Riemannian metrics on $M$.) The map $\tRc \colon \cU \to \cS^2$ that sends a metric $g$ to its Ricci tensor $\tRc_g$ is a second order nonlinear but quasilinear operator, by Example~\ref{ex:curvature-QL}. It is a good exercise to compute that its \emph{linearization} $D_g \tRc \colon \cS^2 \to \cS^2$ at $g$ is given by
\begin{equation} \label{eq:linearization-Ricci}
    ((D_g \tRc)(h))_{jk} = \frac{1}{2} g^{ab} ( \nabla_a \nabla_j h_{bk} + \nabla_a \nabla_k h_{bj} - \nabla_a \nabla_b h_{jk} - \nabla_j \nabla_k h_{ab} ).
\end{equation}
(See~\cite[Lemma 3.5]{Chow-Knopf}, for example.)
\end{ex}

\begin{ex} \label{ex:linearixation-scalar}
Let $\cU$ be as in Example~\ref{ex:linearization-Ricci}. The map $R g \colon \cU \to \cS^2$ that sends a metric $g$ to $Rg$, where $R$ is its scalar curvature, is a second order nonlinear but quasilinear operator,  by Example~\ref{ex:curvature-QL}. Let $g_{\eps}$ be a smooth curve in $\cU$ with $g_0 = g$ and $\frac{d}{d\eps}|_{\eps = 0} g_t = h$. Using~\eqref{eq:linearization-Ricci} we compute that
\begin{align*}
    (D_g (Rg))(h) & = \left. \frac{d}{d\eps} \right|_{\eps = 0} g_t^{ij} (\tRc_{g_t})_{ij} g_t \\
    & = - g^{ik} h_{k\l} g^{lm} (\tRc_g)_{ij} g + g^{ij} ((D_g \tRc)(h))_{ij} g + g^{ij} (\tRc_g)_{ij} h \\
    & = - \langle h, \tRc_g \rangle g + \tr_g ((D_g \tRc)(h) ) g + R h \\
    & = - \langle h, \tRc_g \rangle g + g^{ab} g^{jk} (\nabla_a \nabla_j h_{bk} - \nabla_a \nabla_b h_{jk}) g + R h,
\end{align*}
which further simplifies to
\begin{equation} \label{eq:linearization-scalar}
    (D_g (Rg))(h) = - \Delta (\tr h) g + \Div (\Div h) g - \langle h, \tRc \rangle g + R h.
\end{equation}
\end{ex}

\begin{ex} \label{ez:linearization-LC}
Let $\cU$ be as in Example~\ref{ex:linearization-Ricci}. In this example, let $\nabla^0$ be the Levi-Civita connection of $g_0 \in \cU$. Recall that the difference of two connections is an $\End (TM)$-valued $1$-form. Consider the map $\LC^0 \colon \cU \to \Gamma(T^* M \otimes \End (TM))$ given by $\LC^0(g) = \nabla - \nabla^0$. It is not hard to see from~\eqref{eq:Christoffel} that this is a first order nonlinear but quasilinear differential operator, and that its linearization at $g$ is given by
\begin{equation} \label{eq:linearization-LC}
    ((D_g \LC^0) (h))^k_{ij} = \frac{1}{2} g^{k\l} (\nabla_i h_{j\l} + \nabla_j h_{i\l} - \nabla_{\l} h_{ij}).
\end{equation}
The tensorial equation~\eqref{eq:linearization-LC} can be most easily obtained from~\eqref{eq:Christoffel} by using Riemannian normal coordinates for $g$ at a point, so that covariant differentiation and partial differentiation agree at that point. Note that the linearization $D_g \LC^0$ is independent of the choice of ``reference metric'' $g_0$, as expected.
\end{ex}

\begin{definition} \label{defn:QL-parabolic}
Let $P$ be a nonlinear but quasilinear \emph{second order} differential operator on some open set $\cU$ in $\Gamma(E)$. The evolution equation
\begin{equation*}
   \partial_t s = P s 
\end{equation*}
is called \emph{strongly parabolic at $s \in \cU$} if and only if the linearized evolution equation $\partial_t u = (D_s P) u$ is strongly parabolic in the sense of Definition~\ref{defn:parabolic}.
\end{definition}

The fundamental result on strongly parabolic evolution equations is the following theorem on short-time existence and uniqueness. Precise statements and detailed proofs are notoriously hard to find in the literature. (See also Remark~\ref{rmk:parabolic-defn} below.) One paper which appears to give a complete proof is Mantegazza--Martinazzi~\cite{MM}.

\begin{theorem} \label{thm:parabolic-STE}
Let $P \colon \cU \to \Gamma(E)$ be a second order nonlinear but quasilinear differential operator such that $\partial_t s = P s$ is a strongly parabolic evolution equation. Given $s_0 \in \cU$, there exists $\eps > 0$ and a \emph{unique} smooth solution to $\partial_t s = P s$ for $t \in [0, \eps)$ with initial condition $s(0) = s_0$. That is, a strongly parabolic flow has \emph{short time existence and uniqueness}.
\end{theorem}

\begin{remark} \label{rmk:self-adjoing-SP}
The condition of strong parabolicity only involves the \emph{self-adjoint} part of the principal symbol. More precisely, let $B = \sigma_2 (P)(\xi)$, and let $B = B^+ + B^-$ be the orthogonal decomposition of $B$ into self-adjoint and skew-adjoint endomorphisms. Then $\langle Bs, s \rangle = \langle B^+ s, s \rangle$, so strong parabolicity is equivalent to the self-adjoint part $B^+$ of the principal symbol to be \emph{positive definite}. Note that if $B$ has a kernel, then it cannot be strongly parabolic. But being invertible is certainly not sufficient.
\end{remark}

\begin{remark} \label{rmk:parabolic-defn}
Some references define strong parabolicity to be equivalent to all the (complex) eigenvalues of $B$ having \emph{strictly positive real part}. It is not hard to see that this definition is \emph{not equivalent} to our Definition~\ref{defn:parabolic}. (For example, see Remark~\ref{rmk:RB-STE} below.) The proofs of Theorem~\ref{thm:parabolic-STE} that the author is aware of all seem to use our Definition~\ref{defn:parabolic}. Perhaps one can also prove STE for this inequivalent definition of strong parabolicity, but the author is unaware if this is true. Of course, if $B$ is self-adjoint, then both definitions agree, and all the eigenvalues are strictly real in that case anyway. (Compare with Remarks~\ref{rmk:B-not-self-adjoint} and~\ref{rmk:B-self-adjoint} below.)
\end{remark}

\section{The harmonic map heat flow} \label{sec:harmonic-map-flow}

In this section we consider an important example of a geometric flow which is strongly parabolic, namely the \emph{harmonic map heat flow} introduced by Eells--Sampson~\cite{Eells-Sampson}. Indeed, to the author's knowledge this was the \emph{first} significant example of a geometric flow. It also plays an important role in proving short time existence for the Ricci flow.

Let $(M^n, g)$ and $(N^k, h)$ be a pair of compact Riemannian manifolds. We want to consider the notion of a ``best'' smooth map $F \colon M \to N$. To do this, we define the \emph{map Laplacian} $\Delta_{g, h} F$ of $F$ with respect to the two metrics, as follows. First, for any $p \in M$, the differential $(dF)_p = (f_*)_p$ is a linear map from $T_p M$ to $T_{F(p)} N = (F^* TN)_{p}$. Thus, the differential $dF$ is in fact a smooth section of the bundle $T^* M \otimes F^* TN$. Explicitly, if $x^1, \ldots, x^n$ are local coordinates for $M$ and $y^1, \ldots, y^k$ are local coordinates for $N$, then
\begin{equation*}
    dF = \frac{\partial F^{\alpha}}{\partial x^j} dx^j \otimes \frac{\partial}{\partial y^{\alpha}}.
\end{equation*}

We can take the covariant derivative of $dF$ using the connection induced on $T^* M \otimes F^* TN$ by the Levi-Civita connections of $g$ and $h$. That is, if $(\Gamma_g)^k_{ij}$ and $(\Gamma_h)^{\alpha}_{\beta \gamma}$ are the Christoffel symbols for $g$ and $h$, respectively, then 
\begin{equation*}
    \nabla_i dx^j = - \Gamma^j_{ik} dx^k \qquad \text{and} \qquad \nabla_i \frac{\partial}{\partial y^{\alpha}} = \frac{\partial F^{\beta}}{\partial x^i} \nabla_{\beta} \frac{\partial}{\partial y^{\alpha}} = \frac{\partial F^{\beta}}{\partial x^i} (\Gamma_h)^{\gamma}_{\beta \alpha} \frac{\partial}{\partial y^{\gamma}}.     
\end{equation*}
A computation then yields that $\nabla dF \in \Gamma(T^* M \otimes T^* M \otimes F^* TN)$ is given by
\begin{equation*}
    \nabla dF = \nabla_i dF dx^i = \left( \frac{\partial^2 F^{\alpha}}{\partial x^i \partial x^j } - (\Gamma_g)^k_{ij} \frac{\partial F^{\alpha}}{\partial x^k} + \frac{\partial F^{\beta}}{\partial x^i} \frac{\partial F^{\gamma}}{\partial x^j} (\Gamma_h)^{\alpha}_{\beta \gamma} \circ F \right) dx^i \otimes dx^j \otimes \frac{\partial}{\partial y^{\alpha}}.
\end{equation*}
We now define the \emph{map Laplacian} of $F$ to be the smooth section $\Delta_{g, h} F$ of the pullback tangent bundle $F^* TN$ given by the trace of $\nabla dF$ with respect to $g$. That is,
\begin{equation} \label{eq:map-Laplacian}
    (\Delta_{g, h} F)^{\alpha} = g^{ij} \left( \frac{\partial^2 F^{\alpha}}{\partial x^i \partial x^j } - (\Gamma_g)^k_{ij} \frac{\partial F^{\alpha}}{\partial x^k} + \frac{\partial F^{\beta}}{\partial x^i} \frac{\partial F^{\gamma}}{\partial x^j} (\Gamma_h)^{\alpha}_{\beta \gamma} \circ F \right).
\end{equation}
A smooth map $F \colon (M, g) \to (N, h)$ is called a \emph{harmonic map} if $\Delta_{g, h} F = 0$. In some sense, this is the ``best'' type of map between Riemannian manifolds.

\begin{remark}
The notion of a harmonic map generalizes two familiar concepts. When the codomain $N$ is $\R$ equipped with the standard flat metric, then a harmonic map $F \colon M \to \R$ is just a harmonic function. That is, $F \in C^{\infty} (M)$ is in the kernel of the Laplace-Beltrami operator. When the domain $M = \R$ or $S^1$ equipped with the standard flat metric, then a harmonic map $F \colon M \to N$ is a \emph{geodesic} on $N$.
\end{remark}

The bundle $T^* M \otimes F^* TN$ is equipped with a fibre metric $g^{-1} \otimes F^* h$ induced by $g$, $h$, and $F$. This allows us to define the \emph{energy density} $|dF|^2$ of the map $F$ by
\begin{equation*}
    | dF |^2 = \frac{\partial F^{\alpha}}{\partial x^i} \frac{\partial F^{\beta}}{\partial x^j} g^{ij} (h_{\alpha \beta} \circ F),
\end{equation*}
and it is a smooth function on $M$. One can show tha the harmonic map equation $\Delta_{g, h} F = 0$ is the Euler-Lagrange equation for the \emph{energy functional} $F \mapsto \int_M | dF |^2 \vol_M$.

In fact, the \emph{negative gradient flow} of the energy functional is the geometric flow
\begin{equation} \label{eq:HMHF}
    \partial_t F = \Delta_{g, h} F,
\end{equation}
and is called the \emph{harmonic map heat flow}. It should be clear from the expression~\eqref{eq:map-Laplacian} that the principal symbol of the nonlinear but quasilinear second order differential operator $F \mapsto \Delta_{g, h} F$ is positive definite self-adjoint, and thus the harmonic map heat flow~\eqref{eq:HMHF} should be strongly parabolic and hence have short time existence and uniqueness. This is indeed the case, but there is a subtlety because the domain of this operator is the space $\mathrm{Map} (M, N)$ of smooth maps from $M$ to $N$, which is not the space of sections of a vector bundle, but rather an infinite-dimensional manifold. Nevertheless, Theorem~\ref{thm:parabolic-STE} can indeed be adapted to this setting.

In Section~\ref{sec:DeTurck}, in order to establish short time existence and uniqueness of the Ricci flow,  we will need the following special case of the map Laplacian.

\begin{prop} \label{prop:map-Laplacian-diffeo}
Let $F \colon (M, g) \to (N, h)$ be a \emph{diffeomorphism}. Then the map Laplacian $\Delta_{g, h} F$ of $F$ can be written as
\begin{equation} \label{eq:map-Laplacian-diffeo}
    (\Delta_{g, h} F)^{\gamma} = \big[ ((F^{-1})^* g)^{\alpha \beta} \left( - (\Gamma_{(F^{-1})^* g})^{\gamma}_{\alpha \beta} + (\Gamma_h)^{\gamma}_{\alpha \beta} \right) \big] \circ F.
\end{equation}
\end{prop}
\begin{proof}
In this proof, to keep the notation under control, it is understood that a function of the local coordinates $y^{\alpha}$ on $N$ is precomposed by $F$, so that it is actually a function of the local coordinates $x^i$ on $M$.

For simplicity, let $\tilde h = (F^{-1})^* g$, which is a metric on $N$ satisfying $F^* \tilde h = g$. Note that the Jacobian matrices $\frac{\partial (F^{-1})^i}{\partial y^{\alpha}}$ and $\frac{\partial F^{\alpha}}{\partial x^i}$ are inverses of each other. The Christoffel symbols $\Gamma_g$ for the metric $g = F^* \tilde h$ on $N$ satisfy
\begin{align*}
    (\Gamma_g)^k_{ij} \frac{\partial}{\partial x^k} = (\Gamma_{F^* \tilde h})^k_{ij} \frac{\partial}{\partial x^k} & = \nabla^{F^* \tilde h}_i \frac{\partial}{\partial x^j} = (F^* \nabla^{\tilde h})_i \frac{\partial}{\partial x^j} \\
    & = (F^{-1})_* \left( \nabla^{\wt h}_{F_* \frac{\partial}{\partial x^i}} \Big( F_* \frac{\partial}{\partial x^j} \Big) \right) \\
    & = (F^{-1})_* \left( \frac{\partial F^{\alpha}}{\partial x^i} \nabla^{\wt h}_{\alpha} \Big( \frac{\partial F^{\beta}}{\partial x^j} \frac{\partial}{\partial y^{\beta}} \Big) \right) \\
    & = (F^{-1})_* \left( \frac{\partial^2 F^{\beta}}{\partial x^i \partial x^j} \frac{\partial}{\partial y^{\beta}} + \frac{\partial F^{\alpha}}{\partial x^i} \frac{\partial F^{\beta}}{\partial x^j} (\Gamma_{\tilde h})^{\gamma}_{\alpha \beta} \frac{\partial}{\partial y^{\gamma}} \right) \\
    & = \left( \frac{\partial^2 F^{\gamma}}{\partial x^i \partial x^j} + \frac{\partial F^{\alpha}}{\partial x^i} \frac{\partial F^{\beta}}{\partial x^j} (\Gamma_{\tilde h})^{\gamma}_{\alpha \beta} \right) \frac{\partial (F^{-1})^k}{\partial y^{\gamma}} \frac{\partial}{\partial x^k},
\end{align*}
from which we deduce that
\begin{equation*}
    (\Gamma_g)^k_{ij} \frac{\partial F^{\gamma}}{\partial x^k} = \frac{\partial^2 F^{\gamma}}{\partial x^i \partial x^j} + \frac{\partial F^{\alpha}}{\partial x^i} \frac{\partial F^{\beta}}{\partial x^j} (\Gamma_{\tilde h})^{\gamma}_{\alpha \beta}.
\end{equation*}
We also have $g_{ij} = (F^* \tilde h)_{ij} = \tilde h_{\lambda \mu} \frac{\partial F^{\lambda}}{\partial x^i} \frac{\partial F^{\mu}}{\partial x^j}$, so
\begin{equation} \label{eq:map-Laplacian-diffeo-temp}
g^{ij} = \tilde h^{\lambda \mu} \frac{\partial (F^{-1})^i}{\partial y^{\lambda}} \frac{\partial (F^{-1})^j}{\partial y^{\mu}}.
\end{equation}
The above two expressions yield
\begin{align*}
    g^{ij} (\Gamma_g)^k_{ij} \frac{\partial F^{\gamma}}{\partial x^k} & = g^{ij}\frac{\partial^2 F^{\gamma}}{\partial x^i \partial x^j} + (\tilde h^{\lambda \mu} \circ F) \frac{\partial (F^{-1})^i}{\partial y^{\lambda}} \frac{\partial (F^{-1})^j}{\partial y^{\mu}} \frac{\partial F^{\alpha}}{\partial x^i} \frac{\partial F^{\beta}}{\partial x^j} (\Gamma_{\tilde h})^{\gamma}_{\alpha \beta} \\
    & = g^{ij}\frac{\partial^2 F^{\gamma}}{\partial x^i \partial x^j} + \tilde h^{\alpha \beta} (\Gamma_{\tilde h})^{\gamma}_{\alpha \beta}.
\end{align*}
Substituting the above into~\eqref{eq:map-Laplacian} and using~\eqref{eq:map-Laplacian-diffeo-temp} gives
\begin{align*}
    (\Delta_{g, h} F)^{\gamma} & = - \tilde h^{\alpha \beta} (\Gamma_{\tilde h})^{\gamma}_{\alpha \beta} + g^{ij} \frac{\partial F^{\alpha}}{\partial x^i} \frac{\partial F^{\beta}}{\partial x^j} (\Gamma_h)^{\gamma}_{\alpha \beta} \\
    & = \tilde h^{\alpha \beta} \left( - (\Gamma_{\tilde h})^{\gamma}_{\alpha \beta} + (\Gamma_h)^{\gamma}_{\alpha \beta} \right)
\end{align*}
as claimed.
\end{proof}

\section{Exercises} \label{sec:exercises-structures}

\begin{exercise} \label{exer:GLmC}
Consider the standard embedding $\GL(m, \C) \to \GL(2m, \R)$ given by
\begin{equation*}
    A + i B \to \begin{pmatrix} A & - B \\ B & A \end{pmatrix}.
\end{equation*}
Note that this is precisely the embedding which identifies multiplication by $i$ in $\C^m$ with left multiplication by the matrix
\begin{equation*}
    J = \begin{pmatrix} 0 & - I_m \\ I_m & 0 \end{pmatrix}
\end{equation*}
in $\R^{2m}$, where $J^2 = - I_{2m}$. We identify $\GL(m, \C)$ with the subgroup of $\GL(2m, \R)$ given by this embedding.
\begin{enumerate}[{$[$}a{$]$}]
\item Show that a matrix $P \in \GL(2m, \R)$ lies in $\GL(m, \C)$ if and only if $PJ = JP$.
\item Show that any $P \in \GL(m, \C)$ has $\det P > 0$, so that $\GL(m, \C) \subset \GL^+ (2m, \R)$.
\item Show explicitly that any two bases of $\R^{2m}$ of the form $(e_1, J e_1, \ldots, e_m, J e_m)$ are related by a matrix with positive determinant.
\end{enumerate}
\end{exercise}

\begin{exercise}
Let $(M^n, g, \vol)$ be an oriented Riemannian manifold of dimension $n$, where $\vol$ is the associated volume form. Let $d \colon \Omega^k \to \Omega^{k+1}$ be the exterior derivative, and let $d^* \colon \Omega^k \to \Omega^{k-1}$ be its formal adjoint, defined by $\int_M \langle d \alpha, \gamma \rangle \vol = \int_M \langle \alpha, d^* \gamma \rangle \vol$ for all $\alpha \in \Omega^k$ and $\gamma \in \Omega^{k+1}$. Recall that the Hodge star operator is defined by
\begin{equation*}
    \alpha \w \star \beta = \langle \alpha, \beta \rangle \vol \qquad \text{for $\alpha, \beta \in \Omega^k$.}
\end{equation*}
\begin{enumerate}[{$[$}a{$]$}]
\item Show that $\star^2 = (-1)^{k(n-k)}$ on $\Omega^k$.
\item Show that $\langle \star \alpha, \star \beta \rangle = \langle \alpha, \beta \rangle$, so $\star$ is an isometry.
\item Show that $d^* = (-1)^{nk+n+1} \star d \star$ on $\Omega^k$. \emph{Hint:} Consider $d (\alpha \w \star \beta)$ for $\alpha \in \Omega^{k-1}$ and $\beta \in \Omega^k$.
\end{enumerate}
\end{exercise}

\begin{exercise}
Let $(M, g)$ be a compact oriented Riemannian manifold. Recall that the Hodge Laplacian on $k$-forms is defined to be $\Delta_d = d d^* + d^* d$.
\begin{enumerate}[{$[$}a{$]$}]
\item Show that the Weitzenb\"ock formula for the Hodge Laplacian on $1$-forms is
\begin{equation*}
\Delta_d \alpha = \nabla^* \nabla \alpha + \tRc \alpha \qquad \text{for $\alpha \in \Omega^1$},    
\end{equation*}
where $(\tRc \alpha)(Y) = \tRc(\alpha, Y)$ and we identify $1$-forms with vector fields as usual.
\item If $\tRc \geq 0$, show that any harmonic $1$-form is parallel. Deduce in this case using the Hodge Theorem that the first Betti number $b_1 (M) = \dim H^1 (M, \R)$ is at most $n$.
\item If $\tRc \geq 0$ and $\tRc$ is \emph{strictly} positive at any point, show that any harmonic $1$-form is zero.
\end{enumerate}

\hskip -0.25in \textbf{Introduction to $\SU(2)$-structures.} The remaining exercises in this chapter concern the geometry of $\SU(2)$-structures. A good reference for the reader is Fowdar--S\'a Earp~\cite{FS}. An $\SU(2)$-structure on a $4$-manifold $M$ is encoded by a triple of $2$-forms $( \om_1,\om_2,\om_3)$ which is \emph{pointwise} modelled by the triple
\begin{equation} \label{eq: selfdual omega_i}
    \om_1 = e_0 \w e_1 + e_2 \w e_3, \quad \om_2 = e_0 \w e_2 + e_3 \w e_1, \quad \om_3 = e_0 \w e_3 + e_1 \w e_2,
\end{equation}
on $\R^4$ in terms of the standard basis $\{ e_0, e_1, e_2, e_3 \}$.
\end{exercise}

\begin{exercise} The triple~\eqref{eq: selfdual omega_i} defines a metric $g$ and a triple of orthogonal almost complex structures $(J_1, J_2, J_3)$ on $M^4$ as follows:
\begin{equation}
\label{eq: ACS J_i}
    J_iX \ip \om_j := X \ip \om_k, \qforq (i,j,k)\sim (1,2,3) \qandq X \in \mathfrak{X}(M).
\end{equation}
The compatible metric $g$ is then given by $\om_i(\cdot,\cdot) = g(J_i \cdot, \cdot)$. 
\begin{enumerate}[{$[$}a{$]$}]
    \item Convince yourself that each $J_i$ is indeed an orthogonal almost complex structure.
        
    \item What is the most elegant way to show that~\eqref{eq: selfdual omega_i} defines an $\SU(2)$-structure in the sense of the lectures? (That is, a reduction of the frame bundle $\FR(M)$ to an $\SU(2)$-bundle.)

    \item Observe from~\eqref{eq: selfdual omega_i} that $\frac{1}{2} \om_1 \w \om_1 = \frac{1}{2} \om_2 \w \om_2 = \frac{1}{2} \om_3 \w \om_3$ is a nowhere vanishing top form on $M^4$. Show that with respect to this orientation, this form is the Riemannian volume form $\vol$ of $g$. Let $\star$ denote the associated Hodge star operator. Defining $J_i \alpha : = J_i^* \alpha$, convince yourself that $\star$ and $J_i$ commute on differential forms.

    \item Verify that the triple~\eqref{eq: selfdual omega_i} give a global trivialization of the bundle of self-dual $2$-forms on $M^4$.

    \item Consider the bilinear map $\diamond \colon \Omega^2 \times \Omega^2 \to \Omega^2$ defined on decomposable $2$-forms as follows. If $\alpha=\alpha_1 \w \alpha_2$ and $\beta=\beta_1 \w \beta_2$ we set
        \begin{equation}
        \alpha \diamond \beta := \alpha_1 \w (\alpha_2 \ip \beta) -\alpha_2 \w (\alpha_1 \ip \beta),
        \label{equ: diamond on 2 forms}
        \end{equation}
        where as usual we identify vector fields and $1$-forms using the metric $g$. (Note that since $\alpha_1 \w \alpha_2 = \alpha_1 \otimes \alpha_2 - \alpha_2 \otimes \alpha_1$, this is consistent with our definition of $\diamond$ from~\eqref{eq:diamond-defn}.)
        
        Define $\bar{\om}_1, \bar{\om}_2, \bar{\om}_3$ by
        \begin{equation*}
        \bar{\om}_1 = e_0 \w e_1 - e_2 \w e_3, \quad \bar{\om}_2 = e_0 \w e_2 - e_3 \w e_1, \quad \bar{\om}_3 = e_0 \w e_3 - e_1 \w e_2,
        \end{equation*}
        By computing $\om_i \diamond \om_j$, $\om_i \diamond \bar{\om}_j$, and $\bar{\om}_i \diamond \bar{\om}_j$ for all $i,j=1,2,3$, is it fair to say that $\diamond$ behaves like a Lie bracket on $\so(4) = \Lambda^2 (\R^4)$? In fact you should be able to establish the classical Lie algebra direct sum decomposition $\so(4) = \su(2) \oplus \su(2)$.
\end{enumerate}
    \emph{Hint:} For all these computations it suffices to work in $\R^4$ with the $\SU(2)$-structure defined by~\eqref{eq: selfdual omega_i}.
\end{exercise}
 
\begin{exercise}
Denoting by $\nabla$ the Levi-Civita connection of the metric $g$, the \emph{torsion} tensor $T \in \Gamma( T^* M \otimes \Lambda^2 (T^* M))$ of an $\SU(2)$-structure is defined implicitly by
\begin{equation*}
    \nabla_X \underline{\om} =: T(X) \diamond \underline{\om},
\end{equation*}
where $\underline{\om} = (\om_1, \om_3, \om_3)$. Use the preliminary linear algebraic results of the previous problem to establish the following.
    \begin{enumerate}[{$[$}a{$]$}]
    \item Show that the torsion tensor can be expressed as 
    \begin{align*}
    T  & = -\tfrac{1}{4} \big( g (\nabla \om_2, \om_3) \otimes \om_1 + g (\nabla \om_3, \om_1) \otimes \om_2 + g (\nabla \om_1, \om_2) \otimes \om_3 \big) \\
    & = a_1 \otimes \om_1 + a_2 \otimes \om_2 + a_3 \otimes \om_3, \qforq a_i \in \Omega^1 (M^4),
    \end{align*}
    so that in fact
    \begin{equation*}
    T \in \Gamma( T^* M \otimes \Lambda^2_+ (T^* M)).
    \end{equation*}
         
    \item Equivalently, convince yourself that
    \begin{align*}
    a_1 & = \tfrac{1}{4} \big( \star d \om_1 + J_2 (\star d \om_3) - J_3 (\star d \om_2) \big), \\
    a_2 & = \tfrac{1}{4} \big( \star d \om_2 + J_3 (\star d \om_1) - J_1 (\star d \om_3) \big), \\
    a_3 & = \tfrac{1}{4} \big( \star d \om_3 + J_1 (\star d \om_2) - J_2 (\star d \om_1) \big).
    \end{align*}  
    From the above deduce that
    \begin{equation*}
        T = 0 \iff \nabla \om_1 = \nabla \om_2 = \nabla \om_3 = 0 \iff d \om_1 = d \om_2 = d \om_3 = 0.
    \end{equation*}
     (You will need to use the fact that $J_1 J_2 = - J_3$ on $1$-forms, and cyclic permutations thereof.) We call such a torsion-free $\SU(2)$-structure $(M, \om_1, \om_2, \om_3)$ a \emph{hyperK\"ahler} $4$-manifold.     
\end{enumerate}
\end{exercise}

\begin{exercise}
Recall that the Newlander--Nirenberg theorem says that an almost complex structure $J$ is integrable if and only if its Nijenhuis tensor $N_J$ vanishes, where as in~\eqref{eq:Nijenhuis} we have
    \begin{equation*}
    N_J(X,Y) = [X,Y] + J [JX, Y] + J [X, JY] - [JX, JY]. 
    \end{equation*}
    \begin{enumerate}[{$[$}a{$]$}]
        \item We decompose $T^* M \otimes \C = (T^* M)^{(1.0)} \oplus (T^* M)^{(0,1)}$ into $\pm i$-eigenspaces of $J$. Then we define the space of (complex-valued) forms on $M$ of type $(p,q)$ to be
        \begin{equation*}
        \Omega^{(p,q)} = \Gamma( \Lambda^p ((T^* M)^{(1.0)}) \otimes \Lambda^q ((T^* M)^{(0,1)}).
        \end{equation*}
        Show that $N_J \equiv 0 \iff d (\Omega^{1,0}) \subset \Omega^{2,0} \oplus \Omega^{1,1}$. 
    \item Prove that $(M^4, \om_1, \om_2, \om_3)$ is:
    \begin{enumerate}[(i)]
        \item $J_i$-Hermitian (that is, $J_i$ is integrable) iff $J_j a_j = J_k a_k$, where $(i,j,k)\sim (1,2,3)$.
        \item hyper-Hermitian (that is, $J_1, J_2, J_3$ are all integrable), iff $J_1a_1=J_2a_2=J_3a_3$.
        \item K\"ahler with respect to $\om_i$, iff $a_j = a_k = 0$, where $(i,j,k)\sim (1,2,3)$.
    \end{enumerate}
    \item An $\SU(2)$-structure is called \emph{harmonic} if $\mathrm{div}(T) : = g^{ij} \nabla_i T_{j, kl} = 0$. Show that
    \begin{equation*}
    \mathrm{div}(T) = 0 \iff d^* a_1 = d^* a_2 = d^*a_3 = 0.
    \end{equation*}
    
    \emph{Hint:} For a $1$-form $\alpha$, we have $d^* \alpha = -g^{ij} \nabla_i \alpha_j$.
\end{enumerate}
\end{exercise}

\chapter{The Ricci flow} \label{chapter:Ricci-flpw}

In this chapter, we focus on geometric flows of Riemannian metrics, particularly (but not exclusively) the Ricci flow. This is by far the most well-studied and well-understood flow of $G$-structures, where $G = \O(n)$. In addition to surveying a broad set of known results about the Ricci flow, we emphasize that this flow is \emph{not} strongly parabolic, due to \emph{diffeomorphism invariance}. We make this precise and explain the ``DeTurck trick'' to prove STE and uniqueness for the Ricci flow, by establishing a correspondence with a strongly parabolic flow. One can think of this technique as a type of ``gauge-fixing''. An excellent reference for the Ricci flow, which includes almost everything we discuss in this chapter, is Chow--Knopf~\cite{Chow-Knopf}.

\section{Motivation for the Ricci flow} \label{sec:RF-motivation}

Let $g_0$ be a Riemannian metric on $M$. We seek a ``reasonable'' \emph{geometric flow} of Riemannian metrics
\begin{equation*}
    \partial_t g = P g, \qquad g(0) = g_0.
\end{equation*}
What could this be? In analogy with the heat equation, we look for a ``heat type'' flow, which qualitatively exhibits \emph{diffusive} properties. That is, we want a flow which tends to ``diffuse'' the curvature of $g$ in some sense, or said differently to spread it out evenly and dissipate it away. The first na\"ive guess would be
\begin{equation*}
    \partial_t g = \Delta g = - \nabla^* \nabla g,
\end{equation*}
where $\nabla$ is the Levi-Civita connection of $g$. This cannot work, of course, because $\nabla g = 0$ since $\nabla$ is $g$-compatible. So we need to think a little harder. Since the flow would stop when the right hand side is zero, and we want it to stop when it has ``nice'' curvature, and since the various curvatures are indeed second order nonlinear but quasilinear differential expressions in the metric $g$, the obvious thing to do is to build the right hand side of our flow using (linear combinations of) various notions of curvature of $g$. Moreover, since $g$ is a symmetric $2$-tensor, we have to move $g$ within the space of symmetric $2$-tensors, so $\partial_t g$ must be a symmetric $2$-tensor. We have two obvious choices: the \emph{Ricci tensor} $\tRc$ of $g$, and $R g$, where $R$ is the scalar curvature of $g$.

In fact, it is well-known that the only second order differential invariant of a Riemannian metric $g$ is the Riemann curvature tensor $\tRm$. Moreover, there is an orthogonal decomposition
\begin{equation*}
    \tRm = c_n R g \owedge g + c'_n g \owedge \tRc^0 + W,
\end{equation*}
where $\tRc^0 = \tRc - \frac{1}{n} R g$ is the \emph{trace-free} part of the Ricci tensor, $W$ is the \emph{Weyl} tensor, and $c_n, c'_n$ are constants depending on the dimension $n$. Here $\owedge$ is the \emph{Kulkarni-Nomizu} product that takes two symmetric $2$-tensors and outputs a curvature-type tensor. See Besse~\cite{Besse} for more details. The upshot of this discussion is that the \emph{only} possibility for a flow of Riemannian metrics that has a chance to yield a ``heat type'' flow must be of the form
\begin{equation} \label{eq:flow-metrics}
    \partial_t g = a \tRc + b R g
\end{equation}
for some $a, b \in \R$. We'll return to the general case of~\eqref{eq:flow-metrics} in Section~\ref{sec:RB-flow} but for now we make the following definition.

\begin{definition} \label{defn:ricci-flow}
The \emph{Ricci flow} of Riemannian metrics is the geometric evolution equation
\begin{equation*} \label{eq:ricci-flow}
    \partial_t g = - 2 \tRc.
\end{equation*}
That is, we take $a = -2$ and $b = 0$ in~\eqref{eq:flow-metrics}. With respect to any local frame this becomes
\begin{equation} \label{eq:ricci-flow-coords}
    \partial_t g_{ij} = - 2 R_{ij}.
\end{equation}
\end{definition}

To justify the above choice, and to motivate the underlying issue with most geometric flows, let us consider what the Ricci flow~\eqref{eq:ricci-flow-coords} says at a given time $t$ in a particular special choice of local coordinates. Recall that local coordinates $x^1, \ldots, x^n$ for $M$ are in particular (local) smooth real-valued functions. We say that the coordinates are \emph{harmonic} with respect to $g$ if $\Delta x^i = 0$ for all $i$, where $\Delta$ is the Laplace-Beltrami operator of $g$. From~\eqref{eq:LB-coords2} with $u = x^{\l}$, we get that $x^1, \ldots, x^n$ are harmonic if and only if $g^{ij} \Gamma^{\l}_{ij} = 0$ for all $1 \leq \l \leq n$. It follows from the local solvability of the Laplace equation that local harmonic coordinates \emph{always exist}. (See~\cite[pages 90--92]{Chow-Knopf}.)

\begin{lemma} \label{lemma:harmonic-coords}
In local harmonic coordinates we have
\begin{equation} \label{eq:Ricci-harmonic}
    R_{jk} = - \frac{1}{2} \Delta g_{jk} + \lot,
\end{equation}
where $\lot$ denotes lower order terms. (That is, terms involving only one or no derivatives of the metric.)

Note that in the right hand side of~\eqref{eq:Ricci-harmonic} the expression $\Delta g_{jk}$ denotes the Laplace-Beltrami operator $\Delta$ acting on the \emph{function} $g_{jk}$. Thus~\eqref{eq:Ricci-harmonic} is \emph{not} a coordinate-independent (that is, tensorial) equation. This is expected, because the equation holds only for a \emph{particular} class of special coordinates.
\end{lemma}
\begin{proof}
In local harmonic coordinates we have $g^{ma} \Gamma^b_{ma} = 0$. Differentiating this equation with respect to $x^j$ and using~\eqref{eq:Christoffel}, we get
\begin{equation*}
0 = g^{ma} g^{b \l} \left( \frac{\partial^2 g_{m \l}}{\partial x^j \partial x^a} + \frac{\partial^2 g_{a \l}}{\partial x^j \partial x^m} - \frac{\partial^2 g_{ma}}{\partial x^j \partial x^{\l}} \right) + \lot.
\end{equation*}
Contracting the above with $g_{kb}$ we obtain
\begin{equation*}
g^{ma} \left( \frac{\partial^2 g_{mk}}{\partial x^j \partial x^a} + \frac{\partial^2 g_{ak}}{\partial x^j \partial x^m} - \frac{\partial^2 g_{ma}}{\partial x^j \partial x^k} \right) = \lot.
\end{equation*}
Interchanging $j, k$ and adding to the above gives
\begin{equation*}
g^{ma} \left( \frac{\partial^2 g_{mk}}{\partial x^j \partial x^a} + \frac{\partial^2 g_{mj}}{\partial x^k \partial x^a} + \frac{\partial^2 g_{ak}}{\partial x^j \partial x^m} + \frac{\partial^2 g_{aj}}{\partial x^k \partial x^m} \right) = 2 g^{ma} \frac{\partial^2 g_{ma}}{\partial x^j \partial x^k} + \lot.
\end{equation*}
Since $g^{ma}$ is symmetric, we can swap the roles of $m, a$ in the third and fourth terms above, finally obtaining
\begin{equation} \label{eq:temp-harmonic-coords}
g^{ma} \left( \frac{\partial^2 g_{mk}}{\partial x^j \partial x^a} + \frac{\partial^2 g_{mj}}{\partial x^k \partial x^a} \right) = g^{ma} \frac{\partial^2 g_{ma}}{\partial x^j \partial x^k} + \lot.
\end{equation}

From~\eqref{eq:Christoffel} and~\eqref{eq:curvatures} and the symmetry of $g^{ma}$ we also have
\begin{align*}
    R_{jk} & = R^m_{mjk} = \frac{\partial}{\partial x^m} \Gamma^m_{jk} - \frac{\partial}{\partial x^j} \Gamma^m_{mk} + \lot \\
    & = \frac{1}{2} g^{ma} \left( \frac{\partial^2 g_{ja}}{\partial x^m \partial x^k} + \frac{\partial^2 g_{ka}}{\partial x^m \partial x^j} - \frac{\partial^2 g_{jk}}{\partial x^m \partial x^a} \right) - \frac{1}{2} g^{ma} \left( \frac{\partial^2 g_{ma}}{\partial x^j \partial x^k} + \frac{\partial^2 g_{ka}}{\partial x^j \partial x^m} - \frac{\partial^2 g_{mk}}{\partial x^j \partial x^a} \right) + \lot \\
    & = \frac{1}{2} g^{ma} \left( \frac{\partial^2 g_{ja}}{\partial x^m \partial x^k} - \frac{\partial^2 g_{jk}}{\partial x^m \partial x^a} - \frac{\partial^2 g_{ma}}{\partial x^j \partial x^k} + \frac{\partial^2 g_{mk}}{\partial x^j \partial x^a} \right) + \lot \\
    & = \frac{1}{2} g^{ma} \left( \frac{\partial^2 g_{jm}}{\partial x^a \partial x^k} - \frac{\partial^2 g_{jk}}{\partial x^m \partial x^a} - \frac{\partial^2 g_{ma}}{\partial x^j \partial x^k} + \frac{\partial^2 g_{mk}}{\partial x^j \partial x^a} \right) + \lot.
\end{align*}
Substituting~\eqref{eq:temp-harmonic-coords} into the above, we obtain
\begin{equation*}
    R_{jk} = - \frac{1}{2} g^{ma} \frac{\partial^2 g_{jk}}{\partial x^m \partial x^a} + \lot,
\end{equation*}
which yields~\eqref{eq:Ricci-harmonic} because of~\eqref{eq:LB-coords2} and $g^{ma} \Gamma^k_{ma} = 0$.
\end{proof}

\begin{remark} \label{rmk:harmonic-coords}
It follows from Lemma~\ref{lemma:harmonic-coords} that the Ricci flow $\partial_t = - 2 \tRc$ has the property that in \emph{harmonic coordinates} with respect to $g(t_0)$, the component functions $g_{jk}$ of the metric are moving in the direction of their $g(t_0)$-Laplacians (plus lower order terms), so the evolution is in some sense ``heat like''. (But we emphasize that the coordinates $x^1, \ldots, x^n$ are \emph{not} necessarily harmonic with respect to $g(t_1)$ for any $t_1 \neq t_0$.) This discussion gives one justification for the choice of the constants $a=-2$ and $b=0$ in~\eqref{eq:flow-metrics}.
\end{remark}

\section{Weak parabolicity and diffeomorphism invariance} \label{sec:diff-inv}

In our discussion at the end of Section~\ref{sec:RF-motivation}, we implicitly assumed that the Ricci flow $\partial_t g = - 2 \tRc$ admits a (unique) solution. If the equation was strongly parabolic, this would follow from Theorem~\ref{thm:parabolic-STE}. Unfortunately, however, the Ricci flow is \emph{not} strongly parabolic. Rather it is ``weakly parabolic'', which informally means that it is ``parabolic transverse to the action of diffeomorphisms''. In this section we make this statement precise. Much more detail can be found in~\cite[Chapter 3]{Chow-Knopf}, although we present a significantly simplified argument.

The first thing we need to understand is in what sense the Ricci flow is \emph{diffeomorphism invariant}. Let $F \colon M \to M$ be a diffeomorphism. It is well-known (and easy to check) that
\begin{equation} \label{eq:Ricci-diffeo-inv}
    \tRc_{F^* g} = F^* (\tRc_g).
\end{equation}
That is, the Ricci tensor of the pullback metric $F^* g$ is the pullback of the Ricci tensor of $g$.

Let $\cU \subset \cS^2$ be the open subset of Riemannian metrics on $M$, and let $P \colon \cU \to \cS^2$ be the map $g \mapsto P(g) = - 2 \tRc_g$. Then the Ricci flow is $\partial_t g = P g$.

If $X$ is a vector field on $M$ with flow $\Theta_{\eps}$, then the curve $\gamma$ in $\cU$ given by $\gamma(\eps) = \Theta_{\eps}^* g$ passes through $g$ at $\eps = 0$ with initial velocity $\gamma'(0) = \cL_X g$. In this case, equation~\eqref{eq:Ricci-diffeo-inv} says $\tRc_{\Theta_{\eps}^* g} = \Theta_{\eps}^* (\tRc_g)$, and differentiating this with respect to $\eps$ at $\eps = 0$ yields
\begin{equation} \label{eq:symb-Ricci-1}
    (D_g \tRc) (\cL_X g) = \cL_X (\tRc_g).
\end{equation}
Using~\eqref{eq:divstar-defn}, equation~\eqref{eq:symb-Ricci-1} says that
\begin{equation*}
    (D_g \tRc \circ \Div^*) (X) = - \frac{1}{2} \cL_X (\tRc_g).
\end{equation*}
The right hand side above, as a map $X \mapsto - \frac{1}{2} \cL_X (\tRc_g)$, is only a \emph{first} order differential operator acting on $X$, and hence so is the composition on the left hand side. Taking principal symbols of both sides and using Remark~\ref{rmk:symbol-comp}, we deduce that
\begin{equation} \label{eq:symb-Ricci-2}
    \sigma_2 (D_g \tRc) (\xi) \circ \sigma_1 (\Div^*)( \xi) = 0.
\end{equation}
Equation~\eqref{eq:symb-Ricci-2} says that the image of $\sigma_1 (\Div^*) (\xi)$ lies in the kernel of $\sigma_2 (D_g \tRc) (\xi)$. Thus, if this image is nonzero, then the principal symbol $\sigma_2 (D_g \tRc) (\xi)$ cannot satisfy the condition~\eqref{eq:parabolic} for strong parabolicity.

To simplify notation, let's write
\begin{align*}
    B & := - 2 \sigma_2 (D_g \tRc) (\xi) \colon \Sym^2 (T_p^* M) \to \Sym^2 (T_p^* M), \\
    A & := - \sigma_1 (\Div^*)(\xi) \colon T_p^* M \to \Sym^2 (T_p^* M).
\end{align*}
We have introduced the factor of $-2$ so that $B$ is precisely the principal symbol of the linearization of the right hand side of the Ricci flow.

The next three results say that $\im A = \ker B \neq \{ 0 \}$, and that the operator $B$ is positive self-adjoint on the orthogonal complement of $\im A$. That is, the failure of the Ricci flow $\partial_t g = - 2 \tRc$ to be strongly parabolic is \emph{entirely} due to the diffeomorphism invariance of the Ricci tensor. Recall that $\xi \in T_p^* M$ is nonzero.

\begin{prop} \label{prop:image-divstar}
The map $A$ is injective with $n$-dimensional image.
\end{prop}
\begin{proof}
By~\eqref{eq:divstar} we have
\begin{equation} \label{eq:Asymb}
    2 A X = \xi \otimes X + X \otimes \xi.
\end{equation}
Suppose $A X = 0$. Taking the inner product of~\eqref{eq:Asymb} with $g$, or equivalently taking its trace, we get $\langle X, \xi \rangle = 0$. Then applying both sides of~\eqref{eq:Asymb} to $\xi$ we get $0 = 2 (AX) (\xi) = \xi \langle X, \xi \rangle + |\xi|^2 X$, and thus $X = 0$.
\end{proof}

\begin{prop} \label{prop:kernel-DRc}
The map $B$ has kernel precisely equal to the image of $A$.
\end{prop}
\begin{proof}
We have already seen that $\im A \subseteq \ker B$. Suppose $h \in \ker B$. From~\eqref{eq:linearization-Ricci}, we have
\begin{equation} \label{eq:Bsymb-coords}
    (B h)_{jk} = g^{ab} ( - \xi_a \xi_j h_{bk} - \xi_a \xi_k h_{bj} + \xi_a \xi_b h_{jk} + \xi_j \xi_k h_{ab} ).
\end{equation}
The above can be written invariantly as
\begin{equation} \label{eq:Bsymb}
    Bh = - \xi \otimes h(\xi) - h(\xi) \otimes \xi + |\xi|^2 h + (\tr h) \xi \otimes \xi. 
\end{equation}
Suppose $B h = 0$. Taking the trace of the above gives $0 = - 2 h(\xi, \xi) + 2 |\xi|^2 \tr h$, so that $\tr h = \frac{1}{|\xi|^2} h(\xi, \xi)$ and thus
\begin{equation*}
    0 = - \xi \otimes h(\xi) - h(\xi) \otimes \xi + |\xi|^2 h + \frac{1}{|\xi|^2} h(\xi, \xi) \xi \otimes \xi. 
\end{equation*}
The above then yields
\begin{equation} \label{eq:kernel-Bsymb}
    h = \frac{1}{|\xi|^2} ( \xi \otimes h(\xi) + h(\xi) \otimes \xi) - \frac{1}{|\xi|^4} h(\xi, \xi) \xi \otimes \xi.
\end{equation}
Define $X := \frac{1}{|\xi|^2} h(\xi) - \frac{1}{2 |\xi|^4} h(\xi, \xi) \xi \in T_p ^* M$. It is then clear that $h = \xi \otimes X + X \otimes \xi = 2 A X$.
\end{proof}

\begin{remark} \label{rmk:B-not-self-adjoint}
We observe from~\eqref{eq:Bsymb} that for $h, f \in \cS^2$, we have
\begin{equation*}
    \langle Bh, f \rangle = - 2 \langle h(\xi), f(\xi) \rangle + |\xi|^2 \langle h, f \rangle + (\tr h) f(\xi, \xi),
\end{equation*}
so $\langle Bh, f \rangle \neq \langle h, Bf \rangle$. Thus $B$ is \emph{not} self-adjoint. But see Remark~\ref{rmk:B-self-adjoint} below.
\end{remark}

\begin{prop} \label{prop:symb-perp}
On the orthogonal complement of $\im A$, the map $B$ satisfies $\langle B h, h \rangle = |\xi|^2 \, |h|^2$.
\end{prop}
\begin{proof}
Let $h \in S^2 (T^*_p M)$. Let
\begin{equation} \label{eq:X-deTurk}
    X = \frac{1}{|\xi|^2} h(\xi) - \frac{1}{2 |\xi|^4} h(\xi, \xi) \xi  
\end{equation}
as in the proof of Proposition~\ref{prop:kernel-DRc}, so that
\begin{equation} \label{eq:X-symp-perp2}
    \xi \otimes X + X \otimes \xi = \frac{1}{|\xi|^2} ( \xi \otimes h(\xi) + h(\xi) \otimes \xi) - \frac{1}{|\xi|^4} h(\xi, \xi) \xi \otimes \xi.
\end{equation}
Define
\begin{align} \nonumber
    \breve{h} & = h - (\xi \otimes X + X \otimes \xi) \\ \label{eq:h-breve}
    & = h - \frac{1}{|\xi|^2} \xi \otimes h(\xi) - \frac{1}{|\xi|^2} h(\xi) \otimes \xi + \frac{1}{|\xi|^4} h(\xi, \xi) \xi \otimes \xi.
\end{align}
We claim $\breve{h} \in (\im A)^{\perp}$. To see this, we compute
\begin{align*}
\langle \breve{h}, \xi \otimes Y + Y \otimes \xi \rangle & = \left \langle h - \frac{1}{|\xi|^2} \xi \otimes h(\xi) - \frac{1}{|\xi|^2} h(\xi) \otimes \xi + \frac{1}{|\xi|^4} h(\xi, \xi) \xi \otimes \xi, \xi \otimes Y + Y \otimes \xi \right \rangle \\
& = 2 h(\xi, Y) - 2 h(\xi, Y) - \frac{2}{|\xi|^2} \langle \xi, Y \rangle h(\xi, \xi) + \frac{2}{|\xi|^2} \langle \xi, Y \rangle h(\xi, \xi) \\
& = 0.
\end{align*}
Thus indeed $\breve{h}$ is the orthogonal projection of $h$ onto $(\im A)^{\perp}$. Next observe from~\eqref{eq:h-breve} that
\begin{equation*}
    \breve h(\xi) = h(\xi) - \frac{1}{|\xi|^2} h(\xi, \xi) \xi - h(\xi) + \frac{1}{|\xi|^2} h(\xi, \xi) \xi = 0
\end{equation*}
and
\begin{equation*}
    \breve h (\xi, \xi) = \langle \breve h (\xi), \xi \rangle = 0.
\end{equation*}
Using these and~\eqref{eq:Bsymb}, we compute that
\begin{equation*}
    B \breve h = |\xi|^2 \breve h + (\tr \breve h) \xi \otimes \xi,
\end{equation*}
and
\begin{equation*}
    \langle B \breve h, \breve h \rangle = |\xi|^2 |\breve h|^2 + (\tr \breve h) \breve h (\xi, \xi) = |\xi|^2 |\breve h|^2. \qedhere
\end{equation*}
\end{proof}

\begin{remark} \label{rmk:B-self-adjoint}
From the proof of Proposition~\ref{prop:symb-perp}, we see that for $\breve h, \breve f \in (\im A)^{\perp}$, we have
\begin{equation*}
    \langle B \breve h, \breve f \rangle = |\xi|^2 \langle \breve h, \breve f \rangle + (\tr \breve h) \breve f(\xi, \xi) = |\xi|^2 \langle \breve h, \breve f \rangle,
\end{equation*}
so $B$ \emph{is} self-adjoint when restricted to the subspace $(\im A)^{\perp}$. (Compare with Remark~\ref{rmk:B-not-self-adjoint} above.)
\end{remark}

We have shown that the map $- 2 \sigma_2 (D_g \tRc)(\xi)$ is \emph{positive definite self-adjoint} on $(\im A)^{\perp}$. This suggests that the Ricci flow $\partial_t g = - 2 \tRc$ somehow ``wants to be strongly parabolic'', and the only reason it fails to be is due to the diffeomorphism invariance. This discussion motivates the following generalization of Definition~\ref{defn:QL-parabolic}.

\begin{definition} \label{defn:weakly-parabolic}
Let $P$ be a nonlinear but quasilinear \emph{second order} differential operator on some open set $\cU$ in $\Gamma(E)$. The evolution equation
\begin{equation*}
   \partial_t s = P s 
\end{equation*}
is called \emph{weakly parabolic at $s \in \cU$} if and only if, for any $s \in \cU$, the map $B := (\sigma_2 D_s P)(\xi)$ has kernel precisely equal to the image of $A := - \sigma_1 (\Div^*) (\xi)$ and that the restriction of $B$ to $(\im A)^{\perp}$ is positive definite self-adjoint. (Strictly speaking, we need a \emph{uniform} lower bound $C$ such that~\eqref{eq:parabolic} with $s \in (\im A)^{\perp}$ holds \emph{at all points} of $M$, but this will be automatic if $M$ is compact, which we assume throughout.)
\end{definition}

Note that the observation in Remark~\ref{rmk:harmonic-coords} says that, \emph{in particular specially chosen local coordinates}, the Ricci flow resembles an honest heat equation, which is strongly parabolic. But \emph{choosing particular local coordinates} should be thought of as ``breaking the diffeomorphism invariance'', also known as \emph{gauge-fixing}. This strongly hints that one can establish some kind of correspondence between the Ricci flow and an honest strongly parabolic flow by somehow ``breaking the diffeomorphism invariance''. We do this in the next section.

\section{The DeTurck trick and short time existence for the Ricci flow} \label{sec:DeTurck}

The discussion in Section~\ref{sec:diff-inv} suggests that we look for a \emph{new} flow of the form $\partial_t g = P g$ where $P$ is a nonlinear but quasilinear second order differential operator whose linearization has principal symbol $\sigma_2 (D_g P) (\xi)$ which is precisely the map $h \mapsto |\xi|^2 h$. In such case, the flow $\partial_t g = P g$ would be strongly parabolic and have STE. Recalling that $B = - 2 \sigma_2 (D_g \tRc) (\xi)$, and comparing with~\eqref{eq:Bsymb}, if we write $P g = - 2 \tRc_g + Q g$, then we need
\begin{equation} \label{eq:symbQ-1}
    \sigma_2 (D_g Q) (\xi) (h) = \xi \otimes h(\xi) + h(\xi) \otimes \xi - (\tr h) \xi \otimes \xi. 
\end{equation}
The above is the principal symbol of the linearization
\begin{equation*}
    ((D_g Q)(h))_{jk} = g^{ab} \nabla_j \nabla_a h_{bk} + g^{ab} \nabla_k \nabla_a h_{bj} - g^{ab} \nabla_j \nabla_k h_{ab}.
\end{equation*}
We can write this as
\begin{equation} \label{eq:DeTurck-Q}
    ((D_g Q)(h))_{jk} = \nabla_j Y_k + \nabla_k Y_j = (\cL_Y g)_{jk},
\end{equation}
where $Y$ is the vector field
\begin{equation*}
    Y_k = g^{ab} \nabla_a h_{bk} - \frac{1}{2} \nabla_k (\tr h) = g^{ab} \left( \nabla_a h_{bk} - \frac{1}{2} \nabla_k h_{ab} \right).
\end{equation*}
Using~\eqref{eq:linearization-LC}, we can rewrite this as
\begin{equation} \label{eq:DeTurck-Y}
    Y_{\l} = \frac{1}{2} g^{ab} \Big( \nabla_a h_{b\l} + \nabla_b h_{a\l} - \nabla_{\l} h_{ab} \Big) = g_{k \l} g^{ab} ((D_g \LC^0) (h))^k_{ab}.
\end{equation}
The above computation motivates the following definition. Fix a reference metric $g_0$ on $M$. Consider the map $\W \colon \cU \to \Gamma(TM)$ given by
\begin{equation} \label{eq:DeTurck-W}
    (\W(g))^k = g^{ab} (\LC^0 (g))^k_{ab} = g^{ij} ( (\Gamma^g)^{k}_{ij} - (\Gamma^{g_{0}})^{k}_{ij}) .
\end{equation}
That is, to construct $\W(g)$, we take the trace with respect to $g$ of $\LC^0 (g) \in \Gamma(T^* M \otimes T M \otimes T^* M)$ to obtain a vector field. It is then easy to see from~\eqref{eq:DeTurck-W} and~\eqref{eq:DeTurck-Y} that the first order nonlinear but quasilinear map $\W$ has linearization $(D_g \W)(h) = Y + \lot$.

Now define $Q \colon \cU \to \cS^2$ to be the composition
\begin{equation*}
    Qg = \cL_{\W(g)} g = - 2 (\Div^* \circ \W) (g),
\end{equation*}
where the second equality above is by~\eqref{eq:divstar-defn}. This is a second order nonlinear but quasilinear differential operator. Since $\Div^*$ is linear, taking the linearization at $g$ yields
\begin{equation*}
    (D_g Q) (h) = - 2 (\Div^* \circ D_g \W) (h) = - 2 \Div^* Y + \lot = \cL_Y g + \lot.
\end{equation*}
Thus the principal symbol of the above does indeed have the required property. To summarize, we have shown that the geometric flow
\begin{equation} \label{eq:Ricci-DeTurck}
    \partial_t g = - 2 \tRc_g + \cL_{\W(g)} g = - 2 \tRc_g - 2 (\Div^* \circ \W)(g)
\end{equation}
is strongly parabolic and thus has short time existence and uniqueness.

The flow~\eqref{eq:Ricci-DeTurck} is called the \emph{Ricci-DeTurck flow}. It is clearly not diffeomorphism invariant. That is, $Q(F^* g) \neq F^* (Qg)$. This is because $\W$ is defined with respect to the choice of a fixed reference metric $g_0$, which we might as well take to be the initial metric $g(0)$ of the flow.

We have thus succeeded in finding a (more or less canonical, modulo lower order terms) modification of the Ricci flow which \emph{is} strongly parabolic. What remains to show is that we can use this Ricci-DeTurck flow to prove that the Ricci flow itself has short time existence and uniqueness, even though it is only weakly parabolic. We will do this by in some sense establishing a correspondence between solutions of these two flows, in the following result.

\begin{theorem}[\cite{DeTurck, Hamilton-3}] \label{thm:RF-STE}
Let $g_0$ be a Riemannian metric on $M$. There exists $\eps > 0$ and a unique smooth solution $g \colon [0, \eps) \to \cU$ satisfying
\begin{equation*}
    \partial_t = - 2 \tRc_g \quad \text{on $[0, \eps)$}, \qquad g(0) = g_0.
\end{equation*}
That is, the Ricci flow has short time existence and uniqueness.
\end{theorem}
\begin{proof}
Since the Ricci-DeTurck flow~\eqref{eq:Ricci-DeTurck} is strongly parabolic, there exists $\eps > 0$ and a unique smooth solution to~\eqref{eq:Ricci-DeTurck} with initial condition $g(0) = g_0$, by Theorem~\ref{thm:parabolic-STE}. Define a one-parameter family of diffeomorphisms $F_t \colon M \to M$ by the flow of the non-autonomous vector field $-\W(g(t))$. That is,
\begin{equation} \label{eq:DeTurck-diffeos}
    \partial_t F_t(p) = - \W(g(t))_{F_t(p)} \quad \text{for $t \in [0, \eps)$}, \qquad F_0 = \id_M.
\end{equation}
Because $M$ is compact, the flow of the non-autonomous vector field $-\W(g(t))$ exists and is defined on all of $M$ for as long as $\W(g(t))$ is defined, that is for $t \in [0, \eps)$. Now consider the metrics $\wt{g}(t) = F_t^* g(t)$ for $t \in [0, \eps)$. We claim that $\wt{g}$ satisfies the Ricci flow on $[0, \eps)$ with initial condition $g_0$. Since $F_0 = \id_M$, we have $\wt{g}(0) = g(0) = g_0$. We also have
\begin{align} \nonumber
\partial_t \wt{g} & = \partial_t (F_t^* g) = F_t^* (\cL_{-\W(g)} g) + F_t^* (\partial_t g) \\ \nonumber
& = F_t^* ( - \cL_{\W(g)} g + (-2 \tRc_g + \cL_{\W(g)} g) ) \\ \label{eq:DeTurck-equivalence}
& = - 2 F_t^* (\tRc_g) = - 2 \tRc_{F_t^* g} = - 2 \tRc_{\wt{g}},
\end{align}
and the claim is proved. Thus a smooth solution $\wt{g}$ to the Ricci flow exists on $[0, \eps)$.

It remains to establish uniqueness. To do this, we first investigate what the flow~\eqref{eq:DeTurck-diffeos} corresponds to in terms of $\wt{g}(t)$. From~\eqref{eq:DeTurck-W} we have
\begin{equation*}
    \partial_t F_t = - \W = - g^{ab} ((\Gamma_g)^k_{ab} - (\Gamma_0)^k_{ab}),
\end{equation*}
where $(\Gamma_0)^k_{ab}$ are the Christoffel symbols of $g_0$. Since $\wt{g} = F_t^* g$, we have $g = (F_t^{-1})^* \wt{g}$, and we can write the above as
\begin{equation*}
    \partial_t F_t = ( (F_t^{-1})^* \wt{g} )^{ab} (- (\Gamma_{(F_t^{-1})^* \wt{g}})^k_{ab} + (\Gamma_0)^k_{ab}),
\end{equation*}
which by Proposition~\ref{prop:map-Laplacian-diffeo} becomes
\begin{equation} \label{eq:RF-HMHF}
    \partial_t F_t = \Delta_{\wt{g}(t), g_0} F_t.
\end{equation}
That is, the diffeomorphisms $F_t$ evolve according to the harmonic map heat flow~\eqref{eq:HMHF} with domain metric $\wt{g}(t)$ and codomain metric $g_0$. Stated more precisely, the harmonic map heat flow of $F_t$ with domain metric $\wt{g}(t)$ and codomain metric $g_0$ is equivalent to the flow of the vector field $-\W(g(t))$ where $g(t) = (F_t^{-1})^* \wt{g} (t)$. In this case, one of the metrics for the harmonic map heat flow is itself changing in time, but one can show that the flow still has short time existence and uniqueness. (That is, it is still strictly parabolic.)

Now suppose $\wt{g}_1$ and $\wt{g}_2$ are two solutions of the Ricci flow on $[0, \eps)$ with initial condition $g_0$. Let $F_1 (t)$ and $F_2 (t)$ be the solutions to the harmonic map heat flow~\eqref{eq:RF-HMHF} for these metrics, with initial conditions $F_1 (0) = F_2 (0) = \id_M$. For $k = 1, 2$, define $g_k (t) = (F_k^{-1} (t))^* \wt{g}_k (t)$. The computations in~\eqref{eq:DeTurck-equivalence} are reversible to show that $g_1 (t)$ and $g_2 (t)$ are both solutions of the Ricci-DeTurck flow with the same initial condition $g_0$. By uniqueness of the Ricci-DeTurck flow, $g_1 (t) = g_2 (t) = g(t)$ on $[0, \eps)$. Thus, the harmonic map heat flows~\eqref{eq:RF-HMHF} for $F_1 (t)$ and $F_2 (t)$ are flows of the \emph{same} vector field $- \W(g(t))$, so again by uniqueness we have $F_1(t) = F_2 (t)$ on $[0, \eps)$ and thus finally $\wt{g}_1 (t) = \wt{g}_2 (t)$ on $[0, \eps)$.
\end{proof}

\begin{remark} \label{rmk:DeTurck-Einstein-tensor}
The original formulation of the DeTurck trick demonstrated an interesting connection to the \emph{Einstein} operator $\G \colon \cS^2 \to \cS^2$ given by $\G(h) = h - \frac{1}{2} (\tr h) g$. Note that $\G(g)$ is the \emph{Einsten tensor} of $g$, which is divergence-free by the contracted Riemannian second Bianchi identity. It was shown that the linearization of the Ricci operator is
\begin{equation*}
    (D_g \tRc) (h) = - \frac{1}{2} \Delta_L h - (\Div^* \circ \Div) \G (h),
\end{equation*}
where $\Delta_L$ is the Lichnerowicz Laplacian. (See~\cite[Section 3.3]{Chow-Knopf} for more details.)
\end{remark}

\begin{remark} \label{rmk:STE-noncompact}
Somewhat later, Shi~\cite{Shi} established short-time existence and uniqueness for the Ricci flow on \emph{complete noncompact} Riemannian manifolds with \emph{bounded curvature}.
\end{remark}

\section{Some known analytic results on the Ricci flow} \label{sec:RF-analysis}

In this section we give a very brief and very incomplete survey of some known analytic results on the Ricci flow, essentially to motivate analogous questions for flows of $\Gt$-structures later. References for some of these results are Cao--Chen--Zhu~\cite{CCZ}, Chow--Lu--Ni~\cite{CLN}, and Topping~\cite{Topping}.

\textbf{Preservation of initial properties.} Sometimes one can show that if the initial metric $g_0$ possesses a certain property, often but not always expressed by an \emph{inequality}, then that property is \emph{preserved along the flow}. That is, the metric $g(t)$ also possesses this property for all time for which the flow exists. These kinds of ``preservation of property'' results are almost always proved using the \emph{maximum principle}, which is one of the key analytic tools used in the study of geometric flows.

Some examples of such results are:
\begin{itemize}
    \item A lower bound for scalar curvature $R$ is preserved by the Ricci flow. If $R(g_0) \geq C$, then $R(g(t)) \geq C$ for all $t$ for which the solution exists.
    \item In dimension $n=3$, positive Ricci curvature is preserved by the Ricci flow. If $\tRc(g(0)) \geq C g_0$ for some $C > 0$, then $\tRc(g(t)) \geq C g(t)$ for all $t$ for which the solution exists. This is a very special feature of dimension $3$, because in this case one may recover the Riemann curvature tensor algebraically from the Ricci tensor and the metric.
    \item In any dimension, positivity of the Riemann curvature as a self-adjoint operator on the space of $2$-forms is preserved by the Ricci flow. (Recall that a metric has positive curvature operator $\tRm$ if the eigenvalues of $\tRm$ are positive.)
    \item In higher dimensions $n \geq 4$, certain curvature preservation results are known in some cases. See, for example, B\"ohm--Wilking~\cite{Bohm-Wilking}.
    \item If the initial metric $g_0$ is part of the package of a $\U(m)$-structure $(g, J, \omega)$ on $M^{2m}$, then the Ricci flow of metrics corresponds to a flow of $2$-forms $\omega(t)$ via $\omega(u, v) = g(J u, v)$, where the almost complex structure $J$ is fixed. In the case that $J$ is \emph{integrable} (so we have a holomorphic manifold), and the $2$-form $\omega$ is closed, so that $g$ is a \emph{K\"ahler metric}, then the flow of $2$-forms stays in the same cohomology class. That is, $[\omega(t)] = [\omega(0)]$ for all $t$ for which the solution exists. For this reason, Ricci flow on a K\"ahler manifold is called \emph{K\"ahler-Ricci flow}. In this case $\partial_t \omega_t = - \rho_{\omega(t)}$ via $\omega(u, v) = g(J u, v)$, where $\rho(u, v) = \tRc(Ju, v)$ is the Ricci form of $g$. Moreover, the $\partial\ol{\partial}$-lemma of K\"ahler geometry allows us to write the flow in this case as the flow of a \emph{scalar K\"ahler potential} function. See Cao~\cite{Cao} or Tosatti~\cite{Topping} for more details.
\end{itemize}

\textbf{Characterization of the blow-up time, long time existence, and convergence.} The Ricci flow exists for as long as the Riemann curvature tensor remains bounded. More precisely, if a solution exists on a \emph{maximal} time interval $[0, T)$ with $0 < T \leq \infty$, and if $T < \infty$, then
\begin{equation*}
    \lim_{t \to T^-} \sup_{p \in M} | \tRm_{g_p (t)} |_{g_p (t)} = \infty.
\end{equation*}
(In fact, this has been improved to the \emph{Ricci tensor} blowing up at the singular time. See Sesum~\cite{Sesum} for details.)

Thus, if one can somehow prove in a particular case that the Riemann curvature tensor remains uniformly bounded in space and time, so that it \emph{cannot} blow up, then the Ricci flow has \emph{long time existence} in this case. Even in cases when the Ricci flow has long time existence (LTE), it need not necessarily converge to some limit metric as $t \to \infty$. There are many cases where this is known (strictly speaking, one needs to actually consider the \emph{normalized} Ricci flow so that the total volume remains constant). For example:
\begin{itemize}
    \item In dimension $n=2$, the Ricci flow always has LTE and (if properly normalized) it converges to a metric of constant sectional curvature. (This gives a proof of the classical \emph{Uniformization Theorem}.)
    \item In dimension $n=3$, if the initial metric $g_0$ has positive Ricci curvature, then the Ricci flow has LTE and (if properly normalized) it converges to a metric of constant positive sectional curvature. This is a theorem of Hamilton~\cite{Hamilton-3}.
    \item In higher dimensions $n \geq 4$, partial results are known in some cases. See, for example, the recent survey by Bamler~\cite{Bamler}.
    \item Suppose that $(M^{2m}, g_0, J, \omega_0)$ is a K\"ahler manifold, and that its first Chern class $c_1 (M)$ vanishes. Then the K\"ahler-Ricci flow exists for all time and converges to a Ricci-flat K\"ahler metric, also known as a \emph{Calabi-Yau} metric. This is a parabolic proof of the $c_1 = 0$ case of the Calabi--Yau Theorem. See Cao~\cite{Cao} for details.
\end{itemize}

\section{Is Ricci flow the only game in town?} \label{sec:RB-flow}

Recall that in Section~\ref{sec:RF-motivation} we argued that any reasonable geometric flow of Riemannian metrics, which has a hope of being qualitatively ``heat like'' or diffusive, must be of the form
\begin{equation} \label{eq:general-RBF}
    \partial_t g = a \tRc + b R g,
\end{equation}
for some constants $a, b \in \R$, where $\tRc$ is the Ricci curvature and $R$ is the scalar curvature. This is because $\{ \tRc, Rg \}$ is a basis for the space of second order differential invariants of a Riemannian metric $g$ which are symmetric $2$-tensors. The Ricci flow corresponds to $a=-2, b=0$. In this section, we consider more general values of $a, b$ for which a DeTurck type trick can be applied, to get short time existence and uniqueness. We do this to motivate the analogous problem for flows of $\Gt$-structures, which we discuss in Section~\ref{sec:flows2}.

For simplicity, we will keep $a = -2$, and ask what values of $b$ sufficiently close to $0$ still work to obtain short time existence and uniqueness via a DeTurck type trick. This flow is called the \emph{Ricci--Bourguignon flow}.

Thus let $P g = - 2 \tRc_g + b R g$, and define $C := (\sigma_2 P) (\xi)$. From~\eqref{eq:linearization-scalar}, we immediately see that
\begin{equation*}
    (\sigma_2 (Rg)(\xi)) (h) = - |\xi|^2 (\tr h) g + h(\xi, \xi) g.
\end{equation*}
Hence, in the notation of Section~\ref{sec:diff-inv}, we have
\begin{align} \nonumber
   Ch & = B h + b (- |\xi|^2 (\tr h) g + h(\xi, \xi) g ) \\ \label{eq:RB-temp1}
   & = |\xi|^2 h - \xi \otimes h(\xi) - h(\xi) \otimes \xi + (\tr h) \xi \otimes \xi - b |\xi|^2 (\tr h) g + b h(\xi, \xi) g.
\end{align}

First we show that, as in the Ricci flow case, the kernel of the principal symbol of the right hand side of the flow $\partial_t g = -2 \tRc + b R g$ is again entirely due diffeomorphism invariance.

\begin{prop} \label{prop:kernel-DP-RB}
If $n \neq 2$ or if $b \neq 1$, then the map $C$ has kernel precisely equal to the image of $A$. (If $n=2$ \emph{and} $b=1$, then $P = - 2 \tRc + R g = 0$ identically, so this flow is constant and we can ignore this trivial case.)
\end{prop}
\begin{proof}
The diffeomorphism invariance of the map $g \mapsto Rg$ again yields $\im A \subset \ker C$, but this can also be easily checked directly.

Conversely, suppose $C h = 0$. Taking the trace of~\eqref{eq:RB-temp1} yields
\begin{align*}
    0 & = |\xi|^2 (\tr h) - 2 h(\xi, \xi) + |\xi|^2 (\tr h) - n b |\xi|^2 (\tr h) + n b h(\xi, \xi) \\
    & = (2 - n b) |\xi|^2 (\tr h) + (nb - 2) h(\xi, \xi).
\end{align*}
Thus, if $b \neq \frac{2}{n}$, we get $\tr h = \frac{1}{|\xi|^2} h(\xi, \xi)$. Alternatively, from~\eqref{eq:RB-temp1} we could compute $0 = (Ch)(\xi, \xi)$, which gives
\begin{equation*}
    0 = |\xi|^2 h(\xi, \xi) - 2 |\xi|^2 h(\xi, \xi) + |\xi|^4 \tr h - b |\xi|^4 \tr h + b |\xi|^2 h(\xi, \xi),
\end{equation*}
which also implies that $\tr h = \frac{1}{|\xi|^2} h(\xi, \xi)$ if $b \neq 1$. Substituting this back into~\eqref{eq:RB-temp1} yields
\begin{equation*}
    h = \frac{1}{|\xi|^2} \left( \xi \otimes h(\xi) + h(\xi) \otimes \xi - \frac{1}{|\xi|^2} h(\xi, \xi) \xi \otimes \xi \right),
\end{equation*}
so that if we define $X := \frac{1}{|\xi|^2} h(\xi) - \frac{1}{2 |\xi|^4} h(\xi, \xi) \xi$ as before, then we again have $h = 2 A X$.
\end{proof}

Again motivated by what worked in the Ricci flow case, we attempt to use the same vector field $\W(g)$ to obtain a ``Ricci--Bourguignon--DeTurck flow''. That is, with $\W(g)$ as in~\eqref{eq:DeTurck-W}, we consider the flow
\begin{equation} \label{eq:RB-DeTurck}
        \partial_t g = - 2 \tRc_g + b R g + \cL_{\W(g)} g = - 2 \tRc_g + b R g - 2 (\Div^* \circ \W)(g).
\end{equation}
If we can show that~\eqref{eq:RB-DeTurck} is strongly parabolic, then the same argument as for the Ricci flow will establish short time existence and uniqueness for the Ricci--Bourguignon flow. (Of course, we intuitively expect that this will work only for a range of $b$ close to $b=0$).

Comparing with Section~\ref{sec:DeTurck}, if we consider the flow
\begin{equation*}
\partial_t g = \wt{P} g = P g + Q g = - 2 \tRc_g + b R g + Q g,
\end{equation*}
with $Q$ satisfying~\eqref{eq:symbQ-1}, then from~\eqref{eq:RB-temp1} we get
\begin{equation} \label{eq:symb-RB-DeTurck}
    \wt{C} h := \sigma_2 (D_g \wt{P}) (\xi) (h) = |\xi|^2 h - b |\xi|^2 (\tr h) g + b h(\xi, \xi) g.
\end{equation}

\begin{prop} \label{prop:RB-DeTurck}
The flow~\eqref{eq:RB-DeTurck} is strongly parabolic if
\begin{equation*}
    \left| b + \frac{2(n-1)}{n} \right| < \frac{2}{n} \sqrt{n^2 - n + 1}.
\end{equation*}
(Note that $b=0$ does indeed lie in this range).
\end{prop}
\begin{proof}
We need to show that $\langle \wt{C} h, h \rangle$ is strictly positive for all $h \neq 0$. (Recall that we are always assuming that $M$ is compact, so the uniformity will follow automatically).

Write $h = \lambda g + h^0$ where $\lambda = \frac{1}{n} (\tr h)$ and the trace-free part $h^0$ is orthogonal to $g$. Then $|h|^2 = n \lambda^2 + |h^0|^2$ and $h(\xi, \xi) = \lambda |\xi|^2 + h^0 (\xi, \xi)$. Using these and~\eqref{eq:symb-RB-DeTurck}, we compute
\begin{align*}
    \langle \wt{C} h, h \rangle & = \langle |\xi|^2 h - b |\xi|^2 (\tr h) g + b h(\xi, \xi) g, h \rangle \\
    & = |\xi|^2 |h|^2 - b |\xi|^2 (\tr h)^2 + b h(\xi, \xi) (\tr h) \\
    & = |\xi|^2 (n \lambda^2 + |h^0|^2) - b |\xi|^2 n^2 \lambda^2 + b(\lambda |\xi|^2 + h^0(\xi, \xi)) n \lambda \\
    & = |\xi|^2 \lambda^2 (n - b n^2 + b n) + |\xi|^2 |h^0|^2 + b n \lambda h^0 (\xi, \xi).
\end{align*}
Substituting the sharp inequality $h^0 (\xi, \xi) = \langle h^0 (\xi), \xi \rangle \geq - |h^0 (\xi)| \, |\xi| \geq - |h^0| \, |\xi|^2$ into the above, we get
\begin{equation*}
    \langle \wt{C} h, h \rangle \geq |\xi|^2 \lambda^2 (n - b n^2 + b n) +  |\xi|^2 |h^0|^2 - b n \lambda |\xi|^2 |h^0|.
\end{equation*}
Writing $|h^0| = \mu$ for simplicity, the above becomes
\begin{equation*}
     \langle \wt{C} h, h \rangle \geq |\xi|^2 ( \lambda^2 (n - b n^2 + b n) + \mu^2 - b n \lambda \mu ).
\end{equation*}
The proof will be complete if we can show that the quadratic form
\begin{equation*}
    \lambda^2 (n - b n^2 + b n) + \mu^2 - b n \lambda \mu
\end{equation*}
is positive definite. This quadratic form has matrix
\begin{equation*}
    \begin{pmatrix} 1 & - \frac{1}{2} b n \\ - \frac{1}{2} b n & n - b n^2 + b n \end{pmatrix}
\end{equation*}
which is positive definite if and only if $n - b n^2 + b n - \frac{1}{4} b^2 n^2 > 0$, or equivalently, if and only if
\begin{equation*}
    b^2 + \frac{4}{n} (n - 1) b - \frac{4}{n} < 0.
\end{equation*}
This occurs if and only if $b$ lies in the interval
\begin{equation*}
    \left( - \frac{2(n-1)}{n} - \frac{2}{n} \sqrt{n^2 - n + 1},  - \frac{2(n-1)}{n} + \frac{2}{n} \sqrt{n^2 - n + 1}\right)
\end{equation*}
as claimed.
\end{proof}

\begin{remark} \label{rmk:RB-STE}
In~Catino--Cremaschi--Djadli--Mantegazza--Mazzieri~\cite[Theorem 2.1]{CCDMM} the authors claim that the Ricci--Bourguignon flow is amenable to a DeTurck type trick with the same vector field $\W(g)$ as in the Ricci flow, if $b < \frac{1}{n-1}$. (Our $b$ is their $2 \rho$ in~\cite{CCDMM}). However, their proof is incorrect. They correctly compute the eigenvalues of $\wt{C}$, and show that they are all positive if $b < \frac{1}{n-1}$. However, this is not sufficient for $\langle \wt{C} h, h \rangle$ to be positive for all $h \neq 0$, because $\wt{C}$ is \emph{not self-adjoint}. Let $\wt{C} = \wt{C}^+ + \wt{C}^-$ be the decomposition of $\wt{C}$ into self-adjoint and skew-adjoint parts. Then $\langle \wt{C} h, h \rangle = \langle \wt{C}^+ h, h \rangle$, and positivity of this for $h \neq 0$ is equivalent to all the eigenvalues of $\wt{C}^+$ being positive. This is a different condition. As an example, consider the matrix $G = \begin{pmatrix} 1 & 4 \\ 0 & 1 \end{pmatrix}$, which has all positive eigenvalues. But for $v = \begin{pmatrix} 1 \\ - 1 \end{pmatrix}$, we have $\langle G v, v \rangle = - 2 < 0$.
\end{remark}

\section{Exercises} \label{sec:Ricci-flow}

\begin{exercise}
Let $(M, g)$ be a Riemannian manifold and let $F \colon M \to M$ be a \emph{diffeomorphism}. Then $F^* g$ is again a metric on $M$. Let $\tRm_g$, $\tRc_g$, $R_g$ be the Riemann, Ricci, and scalar curvatures of $g$, respectively. Show that
\begin{equation*}
    \tRm_{F^* g} = F^* \tRm_g, \qquad \tRc_{F^* g} = F^* \tRc_g, \qquad R_{F^* g} = F^* R_g.
\end{equation*}
\end{exercise}

\begin{exercise}
Recall from~\eqref{eq:DeTurck-W} that
\begin{equation*}
    (\W(g))^k = g^{ij} ( (\Gamma^g)^{k}_{ij} - (\Gamma^{g_{0}})^{k}_{ij})
\end{equation*}
with respect to a reference metric $g_0$, so $\W$ is first order nonlinear but quasilinear map $\W \colon \cU \to \Gamma(TM)$ where $\cU$ is the space of Riemannian metrics on $M$. Show $\W$ has linearization $(D_g \W)(h) = Y + \lot$ where
\begin{equation*}
    Y_{\l} = \frac{1}{2} g^{ab} \Big( \nabla_a h_{b\l} + \nabla_b h_{a\l} - \nabla_{\l} h_{ab} \Big).
\end{equation*}
\end{exercise}

\chapter{Introduction to $\Gt$-structures and flows of $\Gt$-structures} \label{chapter:G2}

In this chapter, we give an introduction to the geometry of $\Gt$-structures, including the important \emph{contraction identities} which allow us to effectively compute many geometric quantities without need to resort to abstract representation theory. We mention some important results in $\Gt$-geometry, including the Joyce existence theorem and the Hitchin variational characterization, which together motivate the definition of the $\Gt$-Laplacian flow introduced by Bryant. We then survey several natural flows of $\Gt$-structures and discuss various known results in each case. For a more leisurely introduction to $\Gt$-structures, see Karigiannis~\cite{K-intro}.

\section{The basics of $\Gt$-structures} \label{sec:g2-basics}

Let $M^7$ be a smooth $7$-manifold. A $\Gt$-structure on $M$ is a reduction of the structure group of the frame bundle from $\GL(7, \R)$ to the Lie group $\Gt$, which is a Lie subgroup of $\SO(7)$. As such, a $\Gt$-structure induces a Riemannian metric and an orientation. In fact, we can give a very concrete description of a $\Gt$-structure that is much more suitable for explicit computation and which allows us to think of flows of $\Gt$-structures in a similar way as flows of Riemannian metrics.

To make this precise, let $e_1, \ldots, e_7$ be the standard oriented orthonormal basis of $\R^7$, and define a $3$-form $\ph_0$ on $\R^7$ by
\begin{equation} \label{eq:standard-ph}
\begin{aligned}
    \ph_0 & = e_1 \w e_2 \w e_3 + e_1 \w (e_4 \w e_5 - e_6 \w e_7) \\
    & \qquad {} + e_2 \w (e_4 \w e_6 - e_7 \w e_5) + e_3 \w (e_4 \w e_7 - e_5 \w e_6).
\end{aligned}
\end{equation}
The form $\ph_0$ is expressed in the above particular way so that it can be easily remembered because $\ph_0 = e_1 \w e_2 \w e_3 + e_1 \w \omega_1 + e_2 \w \omega_2 + e_3 \w \omega_3$ where $\omega_1, \omega_2, \omega_3$ are the standard basis of anti-self-dual $2$-forms on $\R^4$ with oriented orthonormal basis $e_4, e_5, e_6, e_7$.

Let $\ps_0 = \star_0 \ph_0$ be the Hodge dual $4$-form. One can verify easily that
\begin{equation} \label{eq:standard-ps}
\begin{aligned}
    \ps_0 & = e_4 \w e_5 \w e_6 \w e_7 - e_2 \w e_3 \w (e_4 \w e_5 - e_6 \w e_7) \\
    & \qquad {} - e_3 \w e_1 \w (e_4 \w e_6 - e_7 \w e_5) - e_1 \w e_2 \w (e_4 \w e_7 - e_5 \w e_6),
\end{aligned}
\end{equation}
or, equivalently, that $\ps_0 = \vol_4 - e_2 \w e_3 \w \omega_1 - e_3 \w e_1 \w \omega_2 - e_1 \w e_2 \w \omega_3$, where $\vol_4 = e_3 \w e_5 \w e_6 \w e_7$ is the volume form on $\R^4$ with oriented orthonormal basis $e_4, e_5, e_6, e_7$.

It is straightforward to verify that for $X, Y \in \R^7$, we have
\begin{equation} \label{eq:metric-ph}
    (X \hk \ph_0) \w (Y \hk \ph_0) \w \ph_0 = - 6 \langle X, Y \rangle \vol
\end{equation}
in terms of the standard inner product and volume form. The group $\Gt$ can be \emph{defined} to be
\begin{equation*}
    \Gt = \{ A \in \GL(7, \R) : A^* \ph_0 = \ph_0 \},
\end{equation*} 
where $(A^* \ph_0) (X_1, \dots, X_7) = \ph_0 (A X_1, \dots, A X_7)$. One can show from the above definition that $\Gt \subset \SO(7)$. In fact, $\Gt$ is a compact, connected, simply-connected simple Lie group of dimension $14$, and is one of the five exceptional Lie groups. It is also the \emph{automorphism group} of the \emph{octonions}, an $8$-dimensional non-associative real division algebra. But none of these facts are relevant for us here. (See Bryant~\cite{Bryant-some-remarks} or Karigiannis~\cite{K-intro} for further details.)

For us, then, a $\Gt$-structure on a $7$-manifold $M$ is a smooth $3$-form $\ph$ with the property that, at every point $p \in M$, there exists a basis for $T_p^* M$ in which $\ph_p$ takes the standard form~\eqref{eq:standard-ph}. This basis can be taken to be of the form $dx^1|_p, \ldots, dx^7|_p$ for local coordinates $x^1, \ldots, x^n$ near $p$, but we can only demand agreement with~\eqref{eq:standard-ph} \emph{at $p$}.

The existence of a $\Gt$-structure on $M$ is a purely topological question. It turns out that $M$ admits $\Gt$-structures if and only if $M$ is both \emph{orientable} and \emph{spinnable}, or equivalently if and only if its first two Stiefel--Whitney classes vanish. (See Joyce~\cite{Joyce} or Lawson-Michelsohn~\cite{Lawson-Michelsohn} for details).

It is worth explicitly showing how $\ph$ determines a Riemannian metric and an orientation in a nonlinear way. Let $x^1, \ldots, x^7$ be any local coordinates on $M$. For~\eqref{eq:metric-ph} to hold for $\ph$, we must have
\begin{equation*}
    - 6 B_{ij} dx^1 \w \cdots \w dx^7 := (\partial_i \hk \ph) \w (\partial_j \hk \ph) \w \ph = - 6 g_{ij} \sqrt{\det g} dx^1 \w \cdots \w dx^7,
\end{equation*}
from which we obtain $B_{ij} = g_{ij} \sqrt{\det g}$. Taking determinants gives $\det B = (\det g)^{\frac{9}{2}}$ and thus $\sqrt{\det g} = (\det B)^{\frac{1}{9}}$ and, finally, $g_{ij} = (\det B)^{-\frac{1}{9}} B_{ij}$. Note that $B_{ij}$ is cubic in $\ph$, so $g_{ij}$ is a very nonlinear function of $\ph$.

In fact, the above construction gives a concrete way of understanding when a $3$-form $\ph$ on a $7$-manifold is a $\Gt$-structure. First, in any local coordinate chart the $7 \times 7$ matrix $B_{ij}$ defined above must be invertible, and then $(\det B)^{- \frac{1}{9}} B_{ij}$ must be positive definite. It is clear from this that being a $\Gt$-structure is an \emph{open} condition. That is, if $M$ admits a $\Gt$-structure $\ph$, then it admits \emph{many} $\Gt$-structures, as any $3$-form sufficiently close to $\ph$ will also be a $\Gt$-structure.

Let $\ps = \star \ph$ be the Hodge dual $4$-form with respect to the metric and orientation induced by $\ph$. One can explicitly verify several fundamental \emph{contraction identities} for $\ph$ and $\ps$, expressed here in terms of a local orthonormal frame for $g$. The most frequently needed contractions are:
\begin{equation} \label{eq:contractions}
\begin{aligned}
\ph_{ijk} \ph_{abk} & = g_{ia} g_{jb} - g_{ib} g_{ja} - \ps_{ijab}, \\
\ph_{ijk} \ph_{ajk} & = 6 g_{ia}, \\
\ph_{ijk} \ps_{abck} & = g_{ia} \ph_{jbc} + g_{ib} \ph_{ajc} + g_{ic} \ph_{abj} - g_{ja} \ph_{ibc} - g_{jb} \ph_{aic} - g_{jc} \ph_{abi}, \\
\ph_{ijk} \ps_{abjk} & = - 4 \ph_{iab}, \\ 
\ps_{ijkl} \ps_{abkl} & = 4 g_{ia} g_{jb} - 4 g_{ib} g_{ja} - 2 \ps_{ijab}, \\ 
\ps_{ijkl} \ps_{ajkl} & = 24 g_{ia}. \\
\end{aligned}
\end{equation}
It is worth emphasizing that all these identities follow from the first one, and that the first one essentially encodes the nonassociativity of octonion multiplication, or, equivalently, the failure of the iteration of the cross product in $\R^7$ to satisfy an identity identical to the iteration of the cross product in $\R^3$. Moreover, it is not an exaggeration to say that all nontrivial results about the local differential geometry of $\Gt$-structures can be traced back to the first identity in~\eqref{eq:contractions}.

\begin{remark} \label{rmk:fund-g2-coordinate-free}
The cross product associated to a $\Gt$-structure $\ph$ is the skew-symmetric bilinear map $\times \colon \mathfrak{X} \times \mathfrak{X} \to \mathfrak{X}$ given by $X \times Y = Y \hk X \hk \ph = \ph(X, Y, \cdot)$. The first identity in~\eqref{eq:contractions} can be written invariantly as
\begin{equation*}
    \langle X \times Y, Z \times W \rangle = \langle X \w Y, Z \w W \rangle - \ps(X, Y, Z, W). 
\end{equation*}
\end{remark}

Given a $\Gt$-structure $\ph$ on $M$, the exterior algebra bundle $\Lambda^{\bullet} (T^* M)$ on $M$ decomposes into an orthogonal direct sum of subbundles corresponding to the decomposition of $\Lambda^{\bullet} (\R^7)^*$ into irreducible representations of the Lie algebra $\mathfrak{g}_2$. Detailed discussion of these decompositions can be found in Dwivedi--Gianniotis--Karigiannis~\cite{DGK-flows2}. We describe these orthogonal direct sum bundle decompositions by equivalently describing the orthogonal direct sum decompositions of their associated spaces of smooth sections. We will content ourselves with the following summary given both invariantly and in terms of a local orthonormal frame.

First, the space of $2$-forms decomposes as $\Omega^2 = \Omega^2_7 \oplus \Omega^2_{14}$ where
\begin{align*}
    \Omega^2_7 & = \{ \beta \in \Omega^2 : \beta = X \hk \ph, X \in \mathfrak{X} \} & \Omega^2_{14} & = \{ \beta \in \Omega^2 : \beta \w \ps = 0 \} \\
    & = \{ \beta \in \Omega^2 : \star (\ph \w \beta) = - 2 \beta \} & & = \{ \beta \in \Omega^2 : \star (\ph \w \beta) = \beta \} \\
    & = \{ \beta \in \Omega^2 : \beta_{ij} \ps_{ijab} = - 4 \beta_{ab} \} & & = \{ \beta \in \Omega^2 : \beta_{ij} \ps_{ijab} = 2 \beta_{ab} \}.
\end{align*}
Note that $\Omega^2_7$ is isomorphic to $\mathfrak{X}$. It follows that we have an orthogonal decomposition
\begin{equation*}
    \cT^2 = \cS^2 \oplus \Omega^2 = \langle g \rangle \oplus \cS^2_0 \oplus \Omega^2_7 \oplus \Omega^2_{14}.
\end{equation*}

As in~\eqref{eq:diamond-defn}, consider the linear map $\cT^2 \to \Omega^3$ given by $A \mapsto A \diamond \ph$. In terms of a local orthonormal frame, this is
\begin{equation*}
    (A \diamond \ph)_{ijk} = A_{im} \ph_{mjk} + A_{jm} \ph_{imk} + A_{km} \ph_{ijm}.
\end{equation*}
One can show that the kernel of this map is $\Omega^2_{14}$. This is just an explicit demonstration that at any point $p \in M$, the space $\Lambda^2_{14} (T^*_p M)$ is identified with the Lie subalgebra $\mathfrak{g}_2$ under the identification of of $\Lambda^2 (T^*_p M)$ with $\mathfrak{so}(7)$ using the metric $g_p$. Moreover, on the complement $\langle g \rangle \oplus \cS^2_0 \oplus \Omega^2_7$, the map $A \mapsto A \diamond \ph$ is a linear isomorphism onto $\Omega^3$.

It follows that the space of $3$-forms decomposes as $\Omega^3 = \Omega^3_1 \oplus \Omega^3_{27} \oplus \Omega^3_7$ where
\begin{align*}
    \Omega^3_1 & = \{ A \diamond \ph : A \in \langle g \rangle \} & \Omega^3_{27} & = \{ A \diamond \ph : A \in \cS^2_0 \} & \Omega^3_7 & = \{ A \diamond \ph : A \in \Omega^2_7 \} \\
    & = \{ f \ph : f \in \Omega^0 \} & & = \{ \gamma \in \Omega^3 : \gamma_{ijk} \ph_{ijk} = 0, \gamma_{ijk} \ps_{ijk\l} = 0 \} & & = \{ (X \hk \ph) \diamond \ph : X \in \mathfrak{X} \}.
\end{align*}
Note that $\Omega^3_1$ is isomorphic to $\Omega^0$ and that $\Omega^3_7$ is isomorphic to $\mathfrak{X}$. In particular, any $3$-form corresponds uniquely to a pair $(h, X)\in \cS^2 \oplus \mathfrak{X}$.

One can also show using the contraction identities that $\Omega^3_7 = \{ Y \hk \ps : X \in \mathfrak{X} \}$, and in fact
\begin{equation} \label{eq:Lambda37}
    Y \hk \ps = - \frac{1}{3} (Y \hk \ph) \diamond \ph.
\end{equation}

\section{The torsion of a $\Gt$-structure} \label{sec:g2torsion}

Let $\ph$ be a $\Gt$-structure on $M$, which induces a metric $g$ and thus a Levi-Civita covariant derivative $\nabla$. For any vector field $X$, the covariant derivative $\nabla_X \ph$ is a $3$-form, so it must be of the form $- \frac{1}{3} \wh{T}(X) \diamond \ph$ for some linear map $\wh{T} \colon \mathfrak{X} \to \langle g \rangle \oplus \cS^2_0 \oplus \Omega^2_7$. (The factor of $- \frac{1}{3}$ is for convenience below). In fact, we claim that $\wh{T}(X) \in \Omega^2_7$ for all $X \in \mathfrak{X}$. To see this, we differentiate the second identity in~\eqref{eq:contractions} to obtain
\begin{equation*}
    \nabla_{\l} \ph_{ijk} \ph_{ijm} = - \ph_{ijk} \nabla_{\l} \ph_{ijm}.
\end{equation*}
Computing the inner product $\langle \nabla_{\l} \ph, A \diamond \ph \rangle$, and using the above, we have
\begin{align*}
    \langle \nabla_{\l} \ph, A \diamond \ph \rangle & = \nabla_{\l} \ph_{ijk} (A \diamond \ph)_{ijk} \\
    & = \nabla_{\l} \ph_{ijk} (A_{im} \ph_{mjk} + A_{jm} \ph_{imk} + A_{km} \ph_{ijm}) \\
    & = 3 \nabla_{\l} \ph_{ijk} A_{km} \ph_{ijm} \\
    & = - 3 \nabla_{\l} \ph_{ijm} A_{km} \ph_{ijk} \\
    & = - \langle \nabla_{\l} \ph, A^T \diamond \ph \rangle,
\end{align*}
and thus $\nabla_i \ph$ is orthogonal to $A \diamond \ph$ if $A$ is symmetric, as claimed.

\begin{definition}
The tensor $\wh{T} \in \Gamma(T^* M \otimes \Lambda^2_7 (T^* M)$ given by $- \frac{1}{3} \wh{T}(X) \diamond \ph = \nabla_X \ph$ is called the \emph{torsion} of the $\Gt$-structure. If $\wh{T} = 0$, then $\ph$ is called \emph{torsion-free}, and it follows from the standard theory of connections that the holonomy of the Riemannian metric $g$ of $\ph$ is contained in the subgroup $\Gt$ of $\SO(7)$. In this case, we call the pair $(M, \ph)$ a \emph{torsion-free} $\Gt$-manifold. (In much of the older literature, these are simply called $\Gt$-manifolds, but it is actually preferable to use the term $\Gt$-manifold for a manifold equipped with a $\Gt$-structure, which in general may have torsion).
\end{definition}

Note that due to the isomorphism $\Omega^2_7 \cong \Omega^1$, we can equivalently express the torsion $\wh{T}$ as a $2$-tensor $T \in \cT^2$ by $T(X) \hk \ps = \nabla_X \ph = - \frac{1}{3} (\wh{T} (X)) \diamond \ph$. That is, putting $Y = T(X)$ in~\eqref{eq:Lambda37}, we have
\begin{equation*}
    T(X) \hk \ph = \wh{T}(X) \in \Omega^2_7 \quad \text{for all $X \in \mathfrak{X}$}.
\end{equation*}
The following are results and remarks about the torsion of a $\Gt$-structure $\ph$, some of which are classical. Detailed proofs and some historical discussions can be found in Karigiannis~\cite{K-flows1} and Dwivedi--Gianniotis--Karigiannis~\cite{DGK-flows2}.
\begin{itemize}
\item The $\Gt$-structure $\ph$ is torsion-free if and only if $d \ph = 0$ and $d \ps = 0$. Note that the first equation is linear, but the second equation is nonlinear because $\ps = \star_{\ph} \ph$ depends nonlinearly on $\ph$. The essential reason for this fact is the following. First, the two conditions $d \ph = 0$ and $d \ps = 0$ have some overlap. The condition $d\ph = 0$ is an equation in $\Omega^3 = \Omega^3_1 \oplus \Omega^3_{27} \oplus \Omega^3_7$, so it has components $1 + 27 + 7$. The condition $* d \ps = 0$ is an equation in $\Omega^2 = \Omega^2_7 \oplus \Omega^2_{14}$, so it has components in $7 + 14$. The two $7$ components are multiples of each other, and up to constants the $1+27+7+14$ components of $d\ph, \star d \ps$ are the $1+27+7+14$ components of the torsion $T$. In particular, the torsion $T$ of a $\Gt$-structure has four independent components, so there are $2^4 = 16$ ``classes'' of $\Gt$-structures. 

\item If $d \ph = 0$ only, then we say $\ph$ is a \emph{closed} $\Gt$-structure. If $d \ps = 0$ only, then we say $\ph$ is a \emph{coclosed} $\Gt$-structure.

\item There exists a very important relation between the covariant derivative $\nabla T$ of the torsion of $\ph$ and the Riemann curvature of the induced metric, given in terms of a local orthonormal frame by
\begin{equation} \label{eq:g2bianchi}
    \nabla_i T_{jk} - \nabla_j T_{ik} = T_{ip} T_{jq} \ph_{pqk} + \frac{1}{2} R_{ijpq} \ph_{pqk}
\end{equation}
This is called the ``$\Gt$-Bianchi identity'', since it can be proved in a similar way to the classical Riemannian Bianchi identities by considering the infinitesimal characterization of diffeomorphism invariance of the torsion. That is, by differentiating $T_{F_t^* \ph} = F_t^* T_{\ph}$ for $F_t$ the flow of a vector field. It can also be proved by directly expanding the Ricci identity $\nabla^2_{X, Y} \ph - \nabla^2_{Y, X} \ph = \tRm(X, Y) \ph$ using $\nabla_X \ph = T(X) \hk \ps$. The $\Gt$-Bianchi identity~\eqref{eq:g2bianchi} is fundamental for understanding the independent second order differential invariants of a $\Gt$-structure, as we discuss in Section~\ref{sec:flows2}.

\item One can show from~\eqref{eq:g2bianchi} that the \emph{Ricci curvature} can be expressed completely in terms of the torsion $T$ and its covariant derivative $\nabla T$, in a nonlinear way. Explicitly, we have
\begin{equation} \label{eq:Ricci-ph}
\begin{aligned}
    R_{ij} & = - \tfrac{1}{2} (\nabla_p T_{iq} \ph_{pjq} + \nabla_p T_{jq} \ph_{piq}) - \tfrac{1}{2} (\nabla_i T_{pq} \ph_{jpq} + \nabla_j T_{pq} \ph_{ipq}) \\
    & \qquad {} - \tfrac{1}{2} (T_{im} T_{pq} \ps_{pqmj} + T_{jm} T_{pq} \ps_{pqmi}) + \tfrac{1}{2} (\tr T) (T_{ij} + T_{ji}) - \tfrac{1}{2} (T_{im} T_{mj} + T_{jm} T_{mi}).
\end{aligned}
\end{equation}
It follows that the metric of a torsion-free $\Gt$-structure is \emph{Ricci-flat}. Another interesting class occurs when $T_{ij} = \lambda g_{ij}$ for some constant $\lambda \neq 0$. In this case, $d \ph = 4 \lambda \ps$ and $d \ps = 0$. These are called \emph{nearly parallel} $\Gt$-structures, and their Ricci tensors satisfy $\tRc = 6 \lambda^2 g$, so the metric is \emph{positive Einstein}.
\end{itemize}

\section{The $\Gt$ Laplacian flow and its motivation} \label{sec:Lap-flow}

Now that we know a bit about $\Gt$-structures, it is natural to ask if we can start with some $\Gt$-structure $\ph_0$, and \emph{flow} it to a ``better'' $\Gt$-structure. Ideally, perhaps we would like to flow to a torsion-free $\Gt$-structure, but let's keep the discussion deliberately vague.

The most well-studied and successful flow of $\Gt$-structures so far is the $\Gt$ Laplacian flow. Before we define this, let us consider some motivation, which are two important foundational theorems by Joyce and Hitchin. We state them here somewhat imprecisely.

\begin{theorem}[Joyce~\cite{Joyce}] \label{thm:Joyce}
Let $\ph_0$ be a \emph{closed} $\Gt$-structure on a \emph{compact} $M$. Suppose that $\ph$ is ``almost'' torsion-free in a very precise sense. (There should exist a closed $4$-form $\xi$ that is sufficiently close to $\star_{\ph_0} \ph_0$ in the $C^0$, $L^2$, and $L^{14}_1$ norms.) Then there exists a torsion-free $\Gt$-structure $\ph$ closed to $\ph_0$ and in the same cohomology class $[\ph_0] \in H^3 (M, \R)$.
\end{theorem}

\begin{remark} \label{rmk:Joyce}
In fact Theorem~\ref{thm:Joyce} is the only existence theorem known for compact torsion-free $\Gt$-manifolds, and all known such constructions use this result in various ways.
\end{remark}

\begin{theorem}[Hitchin~\cite{Hitchin}] \label{thm:Hitchin}
Let $\ph_0$ be a \emph{closed} $\Gt$-structure on a \emph{compact} $M$. Consider the space $\cC$ of $\Gt$-structures in the cohomology class $[\ph_0] \in H^3 (M, \R)$, and define a functional $\cF \colon \cC \to \R$ by $\cF(\ph) = \int_M \vol_{\ph} = \mathrm{Vol}(M, \ph)$, the volume of $M$ with respect to the metric induced by $\ph$. Then the critical points of $\cF$ are precisely the torsion-free $\Gt$-structures in the class $[\ph_0]$, and the functional has a local maximum at such critical points.
\end{theorem}

\begin{remark} \label{rmk:Hitchin}    
In fact, Hitchin showed that the functional $\cF$ is a \emph{Morse--Bott} functional, in the sense that its critical points are nondegenerate in directions transverse to the action of the diffeomorphism group. It is also worth pointing out that there may not exist \emph{any} critical points in $\cC$, and that it is still unknown whether a given cohomology class can admit more than one torsion-free $\Gt$-structure. Theorem~\ref{thm:Hitchin} only gives \emph{local uniqueness} of torsion-free $\Gt$-structures in a given cohomology class.
\end{remark}

The above two theorems provide strong motivation to consider the following situation. If we \emph{start} with a closed $\Gt$-structure $\ph_0$, we may be able to define a flow $\ph(t)$ within its cohomology class $[\ph_0] \in H^3 (M, \R)$ in such a way that the volume $\mathrm{Vol}(M, \ph(t))$ is increasing, hopefully to a local max and thus a torsion-free $\Gt$-structure. In fact, this is indeed a great idea. It was first considered by Bryant~\cite{Bryant-some-remarks} before either of the two above theorems appeared. (Although it was not realized that Bryant's flow was the gradient flow of $\cF$ until later.)

More precisely, define a flow of $\Gt$-structures to be the (positive) gradient flow of the Hitchin volume functional $\cF$, thought of as a function on the space of exact $3$-forms, by $\cF(d \eta) = \mathrm{Vol}(M, \ph_0 + d \eta)$. One can show that this flow is
\begin{equation} \label{eq:Lap-flow}
    \partial_t \ph = \Delta_d \ph, \qquad \text{$\ph(0) = \ph_0$, with $d \ph_0 = 0$},
\end{equation}
where $\Delta_d = d d^* + d^* d$ is the Hodge Laplacian with respect to the metric $g(t)$ of $\ph(t)$, so this flow is indeed a \emph{nonlinear} but quasilinear second order flow. The flow~\eqref{eq:Lap-flow} is called the $\Gt$ Laplacian flow, although perhaps it should be called the Bryant--Hitchin flow. 

Note that by examining equations~\eqref{eq:BW} and~\eqref{eq:AL}, it would \emph{seem} at first glance that the flow~\eqref{eq:Lap-flow} has the wrong sign to exhibit heat-like (diffusive) behaviour. However, something somewhat mysterious happens due to the fact that the Hodge star $\star_{\ph}$ depends nonlinearly on $\ph$, which ends up making the equation behave well, in the following sense.
\begin{theorem}[Bryant--Xu~\cite{Bryant-Xu}] \label{thm:Bryant-Xu}
The $\Gt$ Laplacian flow~\eqref{eq:Lap-flow} starting from a closed $\Gt$-structure $\ph_0$ has short time existence and uniqueness. Moreover under this flow, the $\Gt$-structure $\ph(t)$ stays in the cohomology class $[\ph_0]$ and in particular remains closed.
\end{theorem}

However, as stated, Theorem~\ref{thm:Bryant-Xu} is somewhat misleading. What Bryant--Xu \emph{actually proved} is the following. Let $\ph_0$ be a closed $\Gt$-structure, and let $\ph_t = \ph_0 + d \sigma_t$ for some family of $2$-forms $\sigma_t$ on $M$. Consider the flow
\begin{equation} \label{eq:BX-actual}
    \partial_t \ph_t = \partial_t (d \sigma_t) = d d^{\star_{\ph_t}} (\ph_0 + d \sigma_t), \qquad \sigma_0 = 0. 
\end{equation}
Bryant--Xu proved that the flow~\eqref{eq:BX-actual} has short time existence and uniqueness. This flow clearly implies that $\ph_t = \ph_0 + d \sigma_t$ evolves by the $\Gt$ Laplacian flow $\partial_t \ph_t = \Delta_t \ph_t$ and that $d \ph_t = 0$ for all $t$ such that the flow exists.

However (although it is probably unlikely), we \emph{do not know for sure} that the $\Gt$-Laplacian flow~\eqref{eq:Lap-flow} with initial condition $\ph_0$ which is closed does not have any other solutions which do not stay in the cohomology class $[\ph_0]$. In other words, the formulation~\eqref{eq:BX-actual} actually by construction forces the flow to preserved the closed condition, and in this setting the solution exists (for short time) and is unique. Moreover, by forcing the ``closedness'' to be preserved in this way, we actually have a flow of \emph{exact $3$-forms}, and their argument for establishing STE is somewhat non-standard.

To get a better sense of how this situation differs from similar results in other contexts, we note that \emph{usually}, one proves that a certain property is \emph{preserved} along a geometric flow by using the \emph{maximum principle}. For example:
\begin{itemize}
    \item That positivity of scalar curvature is preserved under the Ricci flow is proved using the maximum principle. Note that STE and uniqueness for Ricci flow is known without any a priori conditions on the initial metric. See Chow--Knopf~\cite{Chow-Knopf} for details.
    \item In a K\"ahler--Einstein manifold, the condition of a submanifold being \emph{Lagrangian} is preserved under the mean curvature flow, and this is proved using the maximum principle. Note that STE and uniqueness for the mean curvature flow is known without any a priori conditions on the initial submanifold. See Smoczyk~\cite{Smoczyk} for details.
    \item If the initial metric $g_0$ is a K\"ahler metric, then under the Ricci flow the metric \emph{stays} K\"ahler for all time for which the flow exists. There exists a direct proof of this using the maximum principle. There is also a proof similar in spirit to the Bryant--Xu proof of Theorem~\ref{thm:Bryant-Xu}. More precisely, let $g_0$ be a K\"ahler metric on the complex manifold $(M^{2m}, J)$ with associated K\"ahler form $\omega_0$. Let $\omega_t = \omega_0 + \sqrt{-1} \partial \bar\partial f_t$ for some family of functions $f_t$ on $M$. Consider the flow
\begin{equation} \label{eq:Cao-KRF}
    \partial_t \omega_t = \partial_t \sqrt{-1} \partial \bar \partial f_t = - 2 \rho_{\omega_t},
\end{equation}
    where $\rho_{\omega_t}$ is the Ricci form of the K\"ahler metric $g_t (\cdot, \cdot) = \omega_t( \cdot, J \cdot)$. Cao~\cite{Cao} proved that the flow~\eqref{eq:Cao-KRF} has STE and uniqueness. In this case, as in~\eqref{eq:BX-actual}, the preservation of K\"ahlerity was forced by the formulation of the flow. However, in this case, this argument leaves no possible gap, for the following reason. We \emph{already know} that the Ricci flow always has STE and uniqueness. The solution of~\eqref{eq:Cao-KRF} also solves the ordinary Ricci flow with the same initial condition, so it \emph{must be} the unique short time solution which we know exists. There cannot exist other solutions to the Ricci flow with this initial condition which do not stay K\"ahler.
\end{itemize}
Thus, in order for the Bryant--Xu approach to short time existence of the $\Gt$ Laplacian flow to be on the same level of completeness as the Cao approach to the K\"ahler-Ricci flow described above, we would need to know that the flow $\partial_t \ph = \Delta_d \ph$ \emph{always} has STE and uniqueness without any particular conditions on the initial $\Gt$-structure. It would then follow by uniqueness and the Bryant--Xu result that the condition of being closed is preserved by the $\Gt$ Laplacian flow. (Although if one could prove STE in general for the $\Gt$ Laplacian flow with no assumption on the initial torsion, then it is likely that one could also prove preservation of the closed condition using the maximum principle, again in analogy with the Ricci flow starting with a K\"ahler metric.) This has not been done (and may not even be true). More realistic is the following open question.

\begin{question} \label{quest:Lap}
    Does there exist some flow of $\Gt$-structures $\partial_t \ph = P \ph$ which \emph{always} has STE and uniqueness without any assumotion on the initial torsion, such that (i) if $\ph_0$ is closed, then $\ph_t$ is closed for all $t$ for which the flow exists, and (ii) when $\ph_0$ is closed, the flow reduces to the $\Gt$ Laplacian flow $\partial_t \ph = \Delta_d \ph$.
\end{question}

Nevertheless, even with this slightly unsatisfying issue regarding STE for the $\Gt$ Laplacian flow, it remains the most well studied flow, and it exhibits many very nice properties. Many foundational results were proved in a series of papers~\cite{LW1, LW2, LW3} by Lotay--Wei, including:
\begin{itemize}
    \item Characterization of the blow-up time. The $\Gt$ Laplacian flow exists for all time such that $|\tRm|^2 + |\nabla T|^2$ remains bounded.
    \item Stability. If $\ph_0$ is sufficiently close to a torsion-free $\Gt$-structure $\wt{\ph}$ in the cohomology class $[\ph_0]$, then the $\Gt$ Laplacian flow has \emph{long time existence} and \emph{converges} to a torsion-free $\Gt$-structure that is in the diffeomorphism orbit of $\wt{\ph}$. (Note that by Joyce's Theorem~\ref{thm:Joyce}, if $\ph_0$ is closed and ``close'' to being torsion-free, then there exists a nearby torsion-free $\Gt$-structure $\wt{\ph}$ in the class $[\ph_0]$, but one should be able to prove this independently using the $\Gt$ Laplacian flow.)
    \item For all $t$ for which the solution to the $\Gt$ Laplacian flow exists, the solution $\ph_t$ is \emph{real analytic}.
\end{itemize}

Some work has also been done for \emph{dimensional reduction} of the $\Gt$ Laplacian flow:
\begin{itemize}
    \item Consider the case where $M^7 = T^3 \times X^4$ where $X^4$ is equipped with a \emph{hypersymplectic structure}. This means that $X^4$ is a compact oriented $4$-manifold equipped with a triple of symplectic forms $\omega_1, \omega_2, \omega_3$ such that the symmetric $3 \times 3$ matrix $Q$ defined by $2 Q_{ij} \mu = \omega_i \w \omega_j$ for any volume form $\mu$ is \emph{positive definite}. Then one defines a $T^3$-invariant closed $\Gt$-structure $\ph_0$ on $M$ by
    \begin{equation*}
    \ph_0 = dt^1 \w dt^2 \w dt^3 - dt^1 \w \omega_1 - dt^2 \w \omega_2 - dt^3 \w \omega_3,
    \end{equation*}
    where $t^1, t^2, t^3$ are the standard (periodic) coordinates on $T^3$. Fine--Yao studied the $\Gt$ Laplacian flow in this setting, in the papers~\cite{Fine-Yao, Fine-Yao-survey}. The flow induces a ``hypersymplectic flow'' of the triple $\omega_1, \omega_2, \omega_3$ on $X^4$, and moreover in this case the flow exists for as long as the scalar curvature remains bounded.
    \item Lambert--Lotay~\cite{Lambert-Lotay} related the $\Gt$ Laplacian flow on $M^7 = B^3 \times T^4$, where $B^3 \subseteq \R^3$ is simply connected, and $\ph_0$ is a $T^4$-invariant closed $\Gt$-structure, to \emph{spacelike mean curvature flow} of $B$ in $\R^{3,3}$.
    \item Picard--Suan~\cite{Picard-Suan} consider the case where $M^7$ is a product of a flat torus with a K\"ahler $2$-fold or $3$-fold $X$ which has vanishing first Chern class. This admits a canonical closed $\Gt$-structure, and they show that the $\Gt$ Laplacian flow in this context reduces to a ``Monge--Amp\`ere type'' flow of $\SU(2)$- or $\SU(3)$-structures on $X$ which have long time existence and convergence to a limit in the diffeomorphism orbit of the unique Ricci-flat K\"ahler metric known to exist by Yau's Theorem.
\end{itemize}

\section{Other geometric flows of $\Gt$-structures} \label{sec:other-flows-G2}

In this section we give a brief survey, with many references, of several other geometric flows of $\Gt$-structures. An excellent survey on the state of knowledge of geometric flows of $\Gt$-structures circa 2017 is Lotay~\cite{Lotay-survey}.

\textbf{The coflow and modified coflow.} Motivated by the success of the $\Gt$ Laplacian flow, it is natural to consider the ``dual'' setting. That is, start with a \emph{coclosed} $\Gt$-structure, represented by a closed positive $4$-form $\ps_0$ (a $4$-form is positive if it is the Hodge star of a $\Gt$-structure) and evolve it by
\begin{equation*}
    \partial_t \ps = \Delta_d \ps_t, 
\end{equation*}
the Hodge Laplacian flow of the $4$-form. It is reasonable to hope that this ``Laplacian coflow'' might have short time existence and uniqueness, and moreover preserve the coclosed condition by staying in the cohomology class $[\ps_0] \in H^4 (M, \R)$. This flow was introduced (albeit with the opposite sign) in Karigiannis--McKay--Tsui~\cite{KMT-coflow} and was dubbed the \emph{coflow} of $\Gt$-structures, but no analytic properties of this flow were addressed in that paper.

Later, Grigorian~\cite{Grigorian-modified} showed that the Bryant--Xu approach to the coflow \emph{does not work}. Indeed it is still unknown whether the coflow (with initial closed $4$-form) has STE in general or not.

\begin{remark} \label{rmk:4form-existence}
 This phenomenon seems to be related to the following curious fact: attempting to prove an analogue of Joyce's existence Theorem~\ref{thm:Joyce} for $4$-forms appears to be much more difficult. That is, if we assume that $\ps_0$ is a closed positive $4$-form on a compact $M$, and suppose that it is sufficiently close to torsion-free in some precise analytic sense, then can we always find a torsion-free $\Gt$-structure $\ph$ whose dual $4$-form $\ps$ is close to $\ps_0$ and in the same cohomology class $[\ps_0] \in H^4 (M, \R)$? This problem is more difficult, because we are modifying our initial $4$-form by an exact $4$-form, and this space is more complicated than the space of exact $3$-forms. The author is currently investigating this question with Shubham Dwivedi.
\end{remark}

In the same paper~\cite{Grigorian-modified}, Grigorian introduced the \emph{modified coflow}, defined to be
\begin{equation*}
    \partial \ps_t = \Delta_d \ps_t + 2 d \big( (C - \tr_{g_t} T_t) \star_t \ps_t \big), \qquad d \ps_0 = 0,
\end{equation*}
where $C$ is a constant. Grigorian proved that the Bryant--Xu appraoch works in this case, so the modified coflow has STE and uniqueness, and stays in the cohomology class $[\ps_0]$. It is not clear what the fixed points of this flow are, but this is currently being studied in both this case and another flow called the $\Gt$ \emph{anomaly flow} (which we do not discuss here) by the author with Sebastien Picard and Caleb Suan. See also the survey~\cite{Grigorian-survey-cc} by Grigorian for a general discussion of the state of knowledge of flows of coclosed $\Gt$-structures circa 2017.

\textbf{The Dirichlet energy gradient flow.} Because the ``best'' class of $\Gt$-structures are those with zero torsion, it is a natural idea to consider the \emph{Dirichlet energy functional} $\ph \mapsto \int_M |T_{\ph}|^2 \vol_{\ph}$, which is the (square of) the $L^2$ norm of the torsion, and to consider its \emph{negative gradient flow}. This was studied by Weiss--Witt in~\cite{WW1, WW2}, where they proved that this flow does indeed have short time existence and uniqueness with no assumption on the initial torsion. (See also Section~\ref{sec:flows2}, particularly Theorem~\ref{thm:DGK} for a generalization of this result.) These ideas were later extended to a ``spinorial energy flow'' on spin manifolds by Ammann--Weiss--Witt~\cite{AWW}.

\textbf{The isometric flow of $\Gt$-structures.} The nonlinear map $\ph \mapsto g_{\ph}$ that sends a $\Gt$-structure to its induced Riemannian metric is \emph{not} injective. In fact, if $g$ is a metric on a manifold $M^7$ that admits $\Gt$-structures, then at any given point $p \in M$, the space of positive $3$-forms in $\Lambda^3 (T_p^* M)$ inducing $g_p$ is an $\R \mathbb{P}^7$. This was first explained by Bryant~\cite{Bryant-some-remarks}. It is natural to ask if there is a ``best'' $\Gt$-structure inducing a given metric $g_0$, and furthermore a natural way to define ``best'' here is to attempt to minimize the Dirichlet energy functional \emph{restricted to the set of $\Gt$-structures inducing $g_0$.} This leads to the \emph{isometric flow} or $\Div T$ flow of $\Gt$-structures, which is
\begin{equation} \label{eq:isometric-flow}
    \partial_t \ph_t = (\Div T_t) \hk \ps_t,
\end{equation}
first introduced by Grigorian~\cite{Grigorian-isometric}. Short time existence and uniqueness for this flow is easy to prove, because this flow is actually \emph{strongly parabolic}. (It is easy to see that it cannot be invariant under the full diffeomorphism group because the metric $g_0$ is held fixed.) See Bagaglini~\cite{Bagaglini} for an early proof.

A great many analytic results for the isometric flow have been proved, by Grigorian~\cite{Grigorian-isometric} and Dwivedi--Gianniotis--Karigiannis~\cite{DGK-isometric}. These include: derivative estimates, characterization of the singular time, compactness of the solution space, almost monotonicity of a localized energy, an $\eps$-regularity theorem, LTE and convergence given small initial entropy, and structure of the singular set. See also the survey~\cite{Grigorian-survey-isometric} by Grigorian on the isometric flow.

The isometric flow has been generalized to other $G$-structures by Loubeau--S\'a Earp~\cite{LS1}, where it is called the \emph{harmonic flow of $G$-structures}, as it is shown that in many ways it is quite similar to the harmonic map heat flow. This has been applied to $\Spin(7)$-structures by Dwivedi--Loubeau--S\`a Earp~\cite{DLS} and to $\Sp(2)$-structures by Loubeau--Moreno--S\'a Earp--Saavedra~\cite{LMSS}. In a recent paper Fadel--Loubeau--Moreno--S\'a Earp~\cite{FLMS} prove more analytic results about the harmonic flow of $G$-structures.

Of course, the isometric flow cannot be used to flow to \emph{torsion-free} $\Gt$-structures, because those have Ricci-flat metrics, and the metric is fixed in the isometric flow. However, in Section~\ref{sec:flows2} we discuss a ``coupling'' of Ricci-flow of metrics with the isometric flow of $\Gt$-structures which is a very promising candidate for a natural geometric flow of $\Gt$-structures.

\begin{remark} \label{rmk:Amanda}
The author's current PhD student Amanda Petcu is working on problems related to flows of $\Gt$-structures. Given a compact oriented $4$-manifold $X^4$ equipped with a hypersymplectic triple $\omega_1, \omega_2, \omega_3$, then we can define a $T^3$-invariant \emph{coclosed} $\Gt$-structure on $T^3 \times X^4$ by
\begin{equation*}
    \ps_0 = \vol_X - dt^2 \w dt^3 \w \omega_1 - dt^3 \w dt^1 \w \omega_2 - dt^1 \w dt^w \w \omega_3,
\end{equation*}
 where $\vol_X$ is the Riemannian volume form on $X$ associated to the hypersymplectic metric. Petcu has shown that in this setting, the original coflow and the modified coflow \emph{agree}, so they both have STE, and in fact in this case the induced flow on the triple $\omega_1, \omega_2, \omega_3$ is again the \emph{hypersymplectic flow} of Fine--Yao~\cite{Fine-Yao, Fine-Yao-survey}. Petcu is also removing the assumption that the $\omega_i$ are closed and studying the isometric flow in this context.
\end{remark}

\section{Some recent results on geometric flows of $\Gt$-structures} \label{sec:flows2}

In this section we survey some of the main results of the recent preprint~\cite{DGK-flows2} entitled ``\emph{Flows of $\Gt$-structures, II: Curvature, torsion, symbols, and functionals}'' by the author in collaboration with Shubham Dwivedi and Panagiotis Gianniotis. This paper is a sequel to the earlier paper~\cite{K-flows1} by the author which initiated the study of general flows of $\Gt$-structures.

Is the $\Gt$ Laplacian flow the best flow to evolve a $\Gt$-structure towards a torsion-free $\Gt$-structure?
\begin{itemize}  \setlength\itemsep{-1mm}
\item Recall that the known proof of STE in this case is unsatisfying, given that the preservation of closedness is built-in from the outset.
\item Moreover, it is not clear if \emph{starting closed and preserving the cohomology class} is the right thing to do. We don't know if there is at most one torsion-free $\Gt$-structure in a given cohomology class.
\end{itemize}

\begin{remark} \label{rmk:compare-CT}
Compare this situation with K\"ahler/Calabi--Yau geometry: Yau's theorem says we can start with a closed $\omega$ and find a (globally unique!) Ricci-flat K\"ahler form $\tilde \omega$ in the cohomology class $[\omega] \in H^2 (M, \R)$. There is also a parabolic proof by Cao~\cite{Cao} using K\"ahler--Ricci flow. Both approaches very heavily rely on the $\del \bar{\del}$-lemma of K\"ahler geometry. There is no such result in $\Gt$-geometry. The problem is that while the complex and symplectic geometry of a K\"ahler manifold decouple, no such decoupling happens in $\Gt$-geometry.
\end{remark}

What would be a ``good'' geometric flow of $\Gt$-structures?
\begin{itemize} \setlength\itemsep{-1mm}
    \item Any reasonable geometric flow of $\Gt$-structures should of course have STE and uniqueness.
    \item Ideally (in analogy with Ricci flow), it should be amenable to a DeTurck type trick yielding equivalence with a strongly parabolic (heat-like) flow.
    \item It should be of the form $\partial_t \ph_t = P \ph_t$, where $\ph \mapsto P \ph$ is some \emph{second order quasilinear differential invariant} of $\Gt$-structures.
\end{itemize}

Thus, in order to identify all such possible heat-like flows of $\Gt$-structures, we need to find all the (independent) second order differential invariants of a $\Gt$-structure, \emph{which are $3$-forms}. Recall from Section~\ref{sec:g2-basics} that the space of $3$-forms on $M$ decomposes as
\begin{equation*}
    (\Omega^1 \oplus \Omega^3_{27}) \oplus \Omega^3_7 \cong \cS^2 \oplus \mathfrak{X}.
\end{equation*}
That is, every $3$-form corresponds uniquely to a pair $(h, X)$ where $h$ is a symmetric $2$-tensor and $X$ is a vector field. So we need to identify all the (independent) second order differential invariants of a $\Gt$-structure which are either symmetric $2$-tensors or vector fields.

Recall that the torsion $2$-tensor $T$ of a $\Gt$-structure $\ph$ decomposes as
\begin{equation*}
T = \underset{\text{symmetric}}{(T_1 + T_{27})} + \underset{\text{skew-symm}}{(T_7 + T_{14})}
\end{equation*}
where
\begin{align*}
    T_1 & = \tfrac{1}{7} (\tr T) g, & T_{27} & \in \mathcal{S}^2_0, & T_7 & = (\vec{T} \hk \ph) \in \Omega^2_7, & T_{14} & \in \Omega^2_{14}.
\end{align*}
here $\vec{T}$ is the vector field corresponding to the section $T_7 \in \Omega^2_7$.

We can similarly decompose the Riemann curvature tensor $\tRm$ into irreducible $\Gt$-representations. Recall that a Riemann curvature tensor in any dimension $n$ \emph{orthogonally} decomposes as $\tRm = a_n R g \owedge g + b_n g \owedge \tRc^0 + W$ where $a_n, b_n$ are dimensional constants, $\owedge$ denotes the Kulkarni-Nomizu product, $\tRc^0 = \tRc - \frac{1}{n} R g$ is the trace-free part of the Ricci tensor, and $W$ is the Weyl curvature tensor. In the presence of a $\Gt$-structure, the Weyl curvature tensor decomposes into three independent pieces. A decscription of this using abstract representation theory is given by Cleyton--Ivanov~\cite{Cleyton-Ivanov}, and an explicit computational approach using the $\Gt$-structure contraction identities is given by Dwivedi--Gianniotis--Karigiannis~\cite{DGK-flows2}. The result is that the Riemann curvature tensor $\tRm$ of $g_{\ph}$ decomposes as
\begin{equation*}
    \tRm = \underset{\text{Ricci curvature}}{( \tfrac{1}{84} R g \owedge g + \tfrac{1}{5} \tRc^0 \owedge g )} + \underset{\text{Weyl curvature}}{( W_{27} + W_{64} + W_{77} )}
\end{equation*}
where $(W_{27})_{ijkl}$ is a curvature tensor which can be expressed in terms of $\ph$, $g$, and a symmetric traceless $2$-tensor $(W_{27})_{ab}$ extracted from the Weyl curvature using $\ph$. The important thing is that $W_{27}$ \emph{is a type of $3$-form}.

Thus there are three independent symmetric $2$-tensors $Rg$, $\tRc^0$, $W_{27}$ coming from the Riemann curvature tensor, and there are no vector fields.

It is computationally convenient to introduce the symmetric $2$-tensor $F_{ij}$ given by
\begin{equation*}
    F_{ij} = R_{abcd} \ph_{abi} \ph_{cdj}.
\end{equation*} 
With some work one can show that
\begin{equation*}
    W_{27} = F + \tfrac{2}{7} R g - \tfrac{4}{5} \tRc^0.
\end{equation*} 
We therefore can and will consider the simpler basis $\{ R g, \tRc, F \}$ for the span of $\{ R_g, \tRc^0, W_{27} \}$.

However, in addition to the Riemann curvature tensor, there is another second order differential invariant of a $\Gt$-structure, namely the covariant derivative $\nabla T$ of the torsion tensor. Note that
\begin{equation*}
    \nabla T \in \Gamma( T^* M \otimes T^* M \otimes T^* M),
\end{equation*}
so to decompose $\nabla T$ into components belonging to orthogonal irreducible $\Gt$-representations, we need to know the splitting of $T^* M \otimes T^* M \otimes T^* M = T^* M \otimes (T^* M \otimes T^* M)$, which we write symbolically as 
\begin{equation*}
    7 \otimes (1 \oplus 27 \oplus 7 \oplus 14).
\end{equation*}
It turns out that
\begin{equation} \label{eq:rep-th}
\begin{aligned}
    7 \otimes 1 & \cong 7, & \qquad 7 \otimes 27 & \cong 77^* \oplus 7 \oplus 64 \oplus 27 \oplus 14, \\
    7 \otimes 7 & \cong 1 \oplus 27 \oplus 7 \oplus 14, & \qquad 7 \otimes 14 & \cong 7 \oplus 27 \oplus 64.
\end{aligned}
\end{equation}
Here $64$ and $77^*$ are irreducible $\Gt$-representation spaces of those dimensions. (The $64$ is the same space that $W_{64}$ lies in, while the $77^*$ is a \emph{different} $77$-dimensional representation from the one that $W_{77}$ lies in.) The upshot is that we get \emph{a lot} of components of $\nabla T$ which can be identified with $3$-forms. Namely,
\begin{equation*}
    \text{$\nabla T$ contains four vector fields, three traceless symmetric $2$-tensors, and one function.}
\end{equation*}
We can explicitly identify these tensors (always up to lower order terms of the form $T \ast T$, which we are ignoring for the purposes of this discussion), as follows:
\begin{itemize} \setlength\itemsep{-1mm}
    \item The four vector fields contained in $\nabla T$ are: $\Div T$, $\Div T^t$, $\nabla(\tr T)$, and $(\nabla T) \hk \ps$. We can write these in a local orthonormal frame as $\nabla_i T_{ip}$, $\nabla_i T_{pi}$, $\nabla_p T_{ii}$, and $\nabla_i T_{jk} \ps_{ijkp}$.
    \item The three traceless symmetric $2$-tensors contained in $\nabla T$ are the trace-free parts of the symmetrizations of $(\nabla_p T_{jk}) \ph_{qjk}$, $(\nabla_i T_{pk}) \ph_{iqk}$$, $$(\nabla_i T_{jp}) \ph_{ijq}$. (These are all of the possible contractions of $\nabla T$ with $\ph$ on two pairs of the three pairs of indices.) Moreover, again up to lower order terms, it can shown that the first of these three is also be expressible as the trace-free part of $\cL_{\vec{T}} g$.
    \item The function contained in $\nabla T$ is the trace of any of the above tensors, equivalently $\langle \nabla T, \ph \rangle = \nabla_i T_{jk} \ph_{ijk}$.
\end{itemize}

So far we have assembled a veritable zoo of second order differential invariants of $\ph$ that can be identified with $3$-forms, and thus can be used as ingredients to construct a heat-like flow of $\Gt$-structures. However, these second order invariants \emph{are not all independent}. Recall that the $\Gt$-Bianchi identity~\eqref{eq:g2bianchi} gives a nontrivial relation between $\tRm$ and $\nabla T$. We repeat it here:
\begin{equation*}
\nabla_i T_{jl} - \nabla_j T_{il} = T_{ia} T_{jb} \ph_{abl} + \tfrac{1}{2} R_{ijab} \ph_{abl}. \end{equation*}
The important thing to note is that the $\Gt$-Bianchi identity is skew in $i, j$, so it is a section of $\Gamma(T^* M \otimes \Lambda^2 (T^* M))$. This means that in terms of irreducible $\Gt$-representations, it follows from~\eqref{eq:rep-th} that the $\Gt$-Bianchi identity is a relation of the form
\begin{equation*}
    7 \otimes (7 \oplus 14) = (1 \oplus 27 \oplus 7 \oplus 14) \oplus (7 \oplus 27 \oplus 64).
\end{equation*}
The above decomposition can be carefully analyzed to reveal:
\begin{itemize}
    \item There is one relation between the functions contained in $\tRm$ and $\nabla T$. Explicitly, it says that, up to lower order terms and overall scalar factors, $\langle \nabla T, \ph \rangle$ is the same as the scalar curvature $R$.
    \item There are two relations between the vector fields contained in $\nabla T$. Explicitly, it turns out that $(\nabla T) \hk \ps$ is entirely lower order (hence can be ignored for the purposes of constructing the leading order behaviour of the flow), and that, up to lower order terms and overall scalar factors, $\nabla (\tr T)$ is the same as $\Div T^t$.
    \item There are two relations between the traceless symmetric $2$-tensors contained in $\tRm$ and $\nabla T$. Explicitly, it turns out that, again up to lower order terms, we can express the last two of the symmetric $2$-tensors coming from $\nabla T$ in terms of the symmetric $2$-tensors $\tRc$, $F$, and $\cL_{\vec{T}} g$.
\end{itemize}
This discussion can be summarized in the following theorem.

\begin{theorem}[Dwivedi--Giannotis--K., 2023]
A basis for the \emph{independent} second order differential invariants of a $\Gt$-structure $\ph$ are:
\begin{equation*}
\underset{\textrm{symmetric $2$-tensors}}{Rg, \, \, \tRc, \, \, F, \, \, \cL_{\vec{T}} g}, \qquad \underset{\textrm{vector fields}}{\Div T, \, \, \Div T^t}.
\end{equation*}
\end{theorem}

This theorem shows that, up to lower order terms which can be shown to all be schematically of the form $T \ast T$, any second order nonlinear but quasilinear flow of $\Gt$-structures can be expressed as
\begin{equation} \label{eq:general-flow2}
    \partial_t \ph = (c_1 R g + c_2 \tRc + c_3 F + c_4 \cL_{\vec{T}} g) \diamond \ph + (c_5 \Div T + c_6 \Div T^t) \hk \ps, 
\end{equation}
for some choice of real constants $c_1, \ldots, c_6$.

We want to determine sufficient conditions on these coefficients so that the flow~\eqref{eq:general-flow2} admits a DeTurck type trick which can then be used to prove short time existence and uniqueness by establishing a correspondence with a strongly parabolic flow.

One can show that if $\ph_t$ evolves by~\eqref{eq:general-flow2}, then its induced metric $g_t$ evolves by
\begin{equation*}
    \partial_t g = 2 (c_1 R g + c_2 \tRc + c_3 F + c_4 \cL_{\vec{T}} g).
\end{equation*}
Thus, since the Ricci flow is very well behaved, it makes sense for us to fix $c_2 = -1$, so that the induced flow of metrics is a perturbation of the Ricci flow. For simplicity (although it is not necessary), let us also choose $c_1 = 0$. Relabelling the coefficients to match~\cite{DGK-flows2}, we are considering the flow
\begin{equation} \label{eq:general-flow3}
    \partial_t \ph = (- \tRc + \lambda F + a \cL_{\vec{T}} g) \diamond \ph + (b_1 \Div T + b_2 \Div T^t) \hk \ps + \l ot,
\end{equation}
for $a, \lambda, b_1, b_2 \in \R$. Then we have the following result.

\begin{theorem}[Dwivedi--Gianniotis--Karigiannis~\cite{DGK-flows2}] \label{thm:DGK}
Suppose that
\begin{equation} \label{eq:DGK-inequalities}
0 \leq b_1 - a-1 < 4, \quad b_1 + b_2 \geq 1, \quad \text{and} \quad |\lambda| < \tfrac{1}{4} (1 - \tfrac{1}{4} (b_1 - a-1)).
\end{equation}
Then the flow~\eqref{eq:general-flow3} has short time existence and uniqueness, \emph{by a slight modification of the DeTurck trick}.
\end{theorem}

There are various interesting remarks we can make about this result.
\begin{itemize} \setlength\itemsep{-1mm}
    \item First, we note that we do \emph{not} make any assumptions on the torsion of the intial $\Gt$-structure $\ph_0$. For example, we do not assume $\ph_0$ is closed or coclosed.
    \item If $a = -\frac{1}{2}, \lambda = 0, b_1 = 1, b_2 = 0$, then we get the negative gradient flow of the Dirichlet energy, which was previous studied by Weiss--Witt~\cite{WW1, WW2} and shown to have STE/uniqueness.
    \item If we choose $a = \lambda = b_1 = b_2 = 0$, which would just give ``Ricci flow'' of $\ph$, then the inequalities~\eqref{eq:DGK-inequalities} are \emph{not} satisfied, so it seems that this ``pure Ricci flow'' of $\Gt$-structures is not well behaved. (Although the inequalities in Theorem~\ref{thm:DGK} are sufficient but may not be necessary.)
    \item By contrast, if we choose If $a = \lambda = b_2 = 0$ and $b_1 = 1$, the inequalities~\eqref{eq:DGK-inequalities} are satisfied and we have STE/uniqueness. This choice gives
    \begin{equation} \label{eq:Ricci-isometric}
        \partial_t \ph_t = - 2 \tRc \diamond \ph + (\Div T) \hk \ps + \l ot,
    \end{equation}
    which is some sort of coupling of Ricci flow with the isometric flow~\eqref{eq:isometric-flow}. This observation suggest that the ``best'' way to try to get to torsion-free might be to evolve the $\Gt$-structure $\ph$ in such a way that the induced metric $g$ evolves by Ricci flow, while the component of the evolution of $\ph$ which does not affect the metric is precisely that which would decrease the Dirichlet energy as quickly as possible if the metric were held fixed.
    \item Gao Chen~\cite{GaoChen} defines a flow of $\Gt$-structures to be a \emph{reasonable flow} if it admits short-time existence and uniqueness, and if it has $\lambda = a = 0$ in~\eqref{eq:general-flow3}, with $b_1, b_2$ arbitrary. Chen proves general Shi-type derivative estimates for such flows. It would be interesting to see if these arguments could be extended to all the flows satisfying the inequalities~\eqref{eq:DGK-inequalities} of Theorem~\ref{thm:DGK}.
    \item It can be shown that the $\Gt$ Laplacian flow $\partial_t \ph = \Delta_d \ph$, with no assumption on the initial torsion, does \emph{not} satisfy the inequalities~\eqref{eq:DGK-inequalities}. (But recall again that these conditions are sufficient but not necessary.) Thus Theorem~\ref{thm:DGK} cannot be directly applied to establish STE for the $\Gt$ Laplacian flow \emph{in general} as discussed just before Question~\ref{quest:Lap}. However, one could attempt to use Theorem~\ref{thm:DGK} to \emph{answer} Question~\ref{quest:Lap}. That is, we can look for a flow that \emph{does} satisfy~\eqref{eq:DGK-inequalities} and which reduces to the $\Gt$ Laplacian flow if $d \ph = 0$. (One would then also have to prove that the closed condition is preserved along such a flow, likely using the maximum principle.)
\end{itemize}

\begin{proof}[Brief sketch of the main ideas for the proof of Theorem~\ref{thm:DGK}.]
By explicitly computing the principal symbols of the operators $\tRc$, $\cL_{\vec{T}} g$, $F$, $\Div T$, and $\Div T^t$, we show that if the inequalities~\eqref{eq:DGK-inequalities} are satisfied, then the failure of strong parabolicity of~\eqref{eq:general-flow3} is \emph{precisely due} to the diffeomorphism invariance. To show this, we analyze the usual \emph{Bianchi operator} $B_1 \colon \cS^2 \to \Omega^1$ given by $(B_1 h)_k = \xi_a h_{ak} - \frac{1}{2 }\xi_k (\tr h)$ as in Ricci flow, \emph{as well as} a $\Gt$-specific operator $B_2 \colon \Omega^1 \to \Omega^1$ given by $(B_2 X)_k = \xi_a X_b \ph_{abk}$.

We then perform a slight variation of the DeTurck trick. We modify the right hand side of the flow~\eqref{eq:general-flow3} by adding a term $\cL_{\W(\ph)} \ph$ with $\W(\ph) \in \Omega^1$ given 
by
\begin{equation*}
(\W(\ph))^k = g^{ij} ( \Gamma^k_{ij} - (\Gamma^0)^k_{ij} ) - 2 a \vec{T}^k.
\end{equation*}
Note that this is a modification of the usual DeTurck vector field $\W(g)$ by the addition of the term $- 2 a \vec{T}$. This yields a strongly parabolic flow ``equivalent'' to the original flow as in the classical DeTurck trick.

The proof of uniqueness is similar to the classical case, with one additional step, because we don't immediately get the harmonic map heat flow for the diffeomorphisms, since the vector field $\W(\ph)$ is different.
\end{proof}

While Theorem~\ref{thm:DGK} shows that there are many flows of $\Gt$-structures which are amenable to a DeTurck type trick to establish STE/uniqueness, this question is completely unaffected by the nature of the lower order terms, which are all schematically of the form $T \ast T$. (Slightly more explicitly, they are sums of various contractions of two components of the torsion tensor with various combinations of $\ph, \ps, g$.)

Although STE/uniqueness only depends on the second order terms, many properties of a geometric flow (such as the characterization of fixed points) are \emph{very sensitive to the lower order terms}. 

\begin{question} \label{qu:coupling}
For the coupling of Ricci flow with isometric flow given by~\eqref{eq:Ricci-isometric}, is it possible to choose the lower order terms so that the fixed points of the flow are torsion-free? This would be some sort of \emph{preferred} coupling of Ricci flow with isometric flow.
\end{question} 

\begin{question} \label{qu:lot}
If the flow induces an evolution of the pointwise norm of the torsion which is of the form of a heat equation
\begin{equation*}
    \frac{\partial}{\partial t} |T|^2 = \Delta |T|^2 + \l ot,
\end{equation*}
with ``good'' signs on the lower order terms, then one can use the maximum principle to show that nice things happen. Is it possible to choose the lower order terms in~\eqref{eq:Ricci-isometric}, or more generally in~\eqref{eq:general-flow3}, to achieve this?
\end{question}

Of course, given a particular choice of geometric flow of $\Gt$-structures, such as one of the flows in Theorem~\ref{thm:DGK}, then one can ask all the usual questions:
\begin{itemize} \setlength\itemsep{-1mm}
    \item What is the characterization of the singular time?
    \item What types of singularities can occur? Are they related to solitons? If so, what kinds of solitons can exist?
    \item When do we get long time existence? Convergence? Stability?
    \item Are there any nice \emph{monotonicity formulas} for some kind of entropy?
\end{itemize}

We note that a recent paper of Dwivedi~\cite{Dwivedi-Spin7} obtains partial results for flows of $\Spin(7)$-structures similar to those obtained in~\cite{DGK-flows2} for flows of $\Gt$-structures.

\section{Exercises} \label{sec:G2}

\begin{exercise}
Let $(M, \ph)$ be a manifold with $\Gt$-structure. Using the contraction identities~\eqref{eq:contractions}, verify that:
    \begin{enumerate}[{$[$}a{$]$}]
    \item the space $\Omega^2_{14}$ is the kernel of the map $\cT^2 \to \Omega^3$ given by $A \mapsto A \diamond \ph$,
    \item for any $Y \in \mathfrak{X}$, we have $Y \hk \ps = - \frac{1}{3} (Y \hk \ph) \diamond \ph$.
    \end{enumerate}
\end{exercise}

\begin{exercise}
By covariantly differentiating the characterization of the torsion tensor $T$ of a $\Gt$-structure $\ph$ as a $2$-tensor, namely $\nabla_i \ph_{jkl} = T_{ip} \ps_{pjkl}$, and using the Ricci identity, verify the ``$\Gt$-Bianchi identity''
    \begin{equation} \label{eq:G2-Bianchi-ex}
        \nabla_i T_{jk} - \nabla_j T_{ik} = T_{ip} T_{jq} \ph_{pqk} + \frac{1}{2} R_{ijpq} \ph_{pqk}
    \end{equation}
    which relates the Riemann curvature tensor with the torsion and its covariant derivative $\nabla T$.
\end{exercise}

\begin{exercise}
By contracting the $\Gt$-Bianchi identity~\eqref{eq:G2-Bianchi-ex} appropriately and using the contraction identities~\eqref{eq:contractions}, derive the formula~\eqref{eq:Ricci-ph} for the Ricci tensor of $g_{\ph}$ in terms of the torsion. Using this formula, verify that if $d \ph = 4 \lambda \ps$ and $d \ps = 0$, then $\tRc = 6 \lambda^2 g$. This shows that \emph{nearly parallel $\Gt$-structures} are always \emph{positive Einstein}.
\end{exercise}

\hskip -0.25in \textbf{Warped product $\Gt$-structures.} The remaining exercises in this chapter concern ``warped product'' $\Gt$-structures on $N^6 \times L^1$, where $N^6$ is equipped with an $\SU(3)$-structure. A good reference for the reader is Karigiannis--McKay-Tsui~\cite{KMT-coflow}. An $\SU(3)$-structure on a $6$-manifold $N$ is determined by an almost complex structure $J$, a $J$-compatible Riemannian metric $h$, and a nonvanishing smooth complex-valued $(3,0)$-form $\Upsilon$, such that the associated K\"ahler form $\omega (X, Y) = h (JX, Y)$ satisfies
\begin{equation*}
    \vol_N = \tfrac{1}{3!} \omega^3 = \tfrac{i}{8} \Upsilon \w \ol{\Upsilon} = \tfrac{1}{4} \real (\Upsilon) \w \imag (\Upsilon).
\end{equation*}

\begin{exercise}
Let $M^7 = N^6 \times L^1$, where $L^1$ is $\R$ or $S^1$, and denote by $r$ the standard coordinate on $L^1$. Consider a smooth, nowhere vanishing complex-valued function $F(r)$ and a smooth, everywhere positive function $G(r)$ on $L^1$.  
\begin{enumerate}[{$[$}a{$]$}]
    \item Verify that the following $3$-form defines a $\Gt$-structure on $M$:
    \begin{equation} \label{eq: warped_G2_struct}
    \ph = \real (F^3 \Upsilon) - G |F|^2 dr \w \omega,
    \end{equation}
and that $\ph$ induces the metric $g = G^2 dr^2 + |F|^2 h$ and the volume form $\vol_g = G |F|^6 dr \w \vol_h$.
    \item Show that the Hodge star operator induced by~\eqref{eq: warped_G2_struct} satisfies 
    \begin{equation*}
    \star_g \alpha = (-1)^k|F|^{6-2k} Gdr \w \star_h \alpha, \qquad \star_g (dr \w \alpha) = |F|^{6-2k} G^{-1} \star_h \alpha, \qforq \alpha \in \Omega^k (N).
    \end{equation*}
    In particular, the dual $4$-form $\ps = \star_g \ph$ is
    \begin{equation*}
    \ps = - G dr \w \imag (F^3 \Upsilon) - \tfrac{1}{2} |F|^4 \omega.
    \end{equation*}
    \end{enumerate}
\end{exercise}

\begin{exercise}
The $\SU(3)$-structure $(N^6, J, g, \omega, \Upsilon)$ is called \emph{Calabi-Yau} (CY) if $d\omega = 0$ and $d \Upsilon = 0$; and \emph{nearly K\"ahler} (NK) if
    \begin{equation*}
    d \omega = - 3 \real (\Upsilon), \qquad d \imag (\Upsilon) = 2 \omega^2.
    \end{equation*}
    \begin{enumerate}
    \item Check that
    \begin{align*}
    d \ph  & = - \left(\frac{(F^3)'}{2} \Upsilon + \frac{(\ol{F}^3)'}{2} \ol{\Upsilon} \right) \w dr + G |F|^2 dr \w d \omega + \frac{F^3}{2} d \Upsilon + \frac{\ol{F}^3}{2} d \ol{\Upsilon} \\
    d \ps & = - \frac{iGF^3}{2} dr \w d \Upsilon + \frac{i G \ol{F}^3}{2} dr \w d \ol{\Upsilon} - (|F|^4)' dr \w\frac{\omega^2}{2} -|F|^4 d \left(\frac{\omega^2}{2} \right).
    \end{align*}
            
    \item Show that~\eqref{eq: warped_G2_struct} is torsion-free iff $F$ is constant when $N^6$ is CY, and iff $F(r) = r$ when $N^6$ is NK. What can you say about $G(r)$?

    \item Writing $F(r)= \ell(r) e^{i\theta(r)}$ for some $\ell, \theta \in C^{\infty}(L^1)$, show that the torsion forms of~\eqref{eq: warped_G2_struct} are:
    \begin{align*}
    \text{when $N$ is CY:} & \qquad
    \tau_0 = \frac{12}{7G} \theta', \qquad \tau_1 = d (\log \ell), \qquad \tau_2 = 0; \\
    \text{when $N$ is NK:} & \qquad 
    \tau_0 = \frac{12}{7} \left(\frac{\theta'}{G} + \frac{2\sin(3\theta)}{\ell} \right), \qquad \tau_1 = \left(\frac{\ell' - G\cos(3\theta)}{\ell} \right) dr, \qquad \tau_2 = 0.
    \end{align*}
    We also have $\tau_3 = \star_g (d\ph) - \tau_0 \ph - 3 \star_g (\tau_1 \w \ph)$.
    \end{enumerate}
\end{exercise}

\begin{exercise}
A $1$-parameter family of $\Gt$-structures $\{\ph(t)\}_{[0,T)}$ on $M^7$ solves the~\emph{Laplacian coflow} if
    \begin{equation} \label{eq: Laplacian coflow}
    \frac{\partial}{\partial t} \ps = \Delta_d \ps, \qquad \text{for all $t \in [0,T)$},
    \end{equation}
    where $\ps = \star_{\ph(t)}\ph(t)$ and $\Delta_d = d d^* + d^* d$ is the Hodge Laplacian of  $g(t) := g_{\ph(t)}$.
    \begin{enumerate}
    \item Assuming that~\eqref{eq: warped_G2_struct} is coclosed, verify that
    \begin{align*}
    \text{when $N$ is CY:} \qquad \Delta_d \ps = & - \left( \frac{i(F^3)'}{2G} \right)' dr \w \Upsilon + \left( \frac{i(\ol{F}^3)'}{2G} \right)'dr \w \ol{\Upsilon}; \\
    \text{when $N$ is NK:} \qquad \Delta_d \ps = & - A dr\w \Upsilon - \ol{A} dr\w \ol{\Upsilon} - B \frac{\omega^2}{2},
    \end{align*}
    where
    \begin{equation*}
    A = \left(\frac{i(F^3)'}{2G} - \frac{3i \ell^2}{2} \right)' + 6 G \ell \sin(3\theta) \qquad \text{and} \qquad B = \left( -\frac{4}{G} (\ell^3 \cos(3\theta))' + 12 \ell^2 \right).
    \end{equation*}

    \item Under the flow Laplacian coflow~\eqref{eq: Laplacian coflow}, show that $F := \ell e^{i\theta}$ and $G$ satisfy the evolution equations:
    \begin{align*}
    \text{when $N$ is CY:}\qquad \ell = 1 &, \qquad \frac{\partial\theta}{\partial t} = - \Delta \theta, \qquad \frac{\partial G}{\partial t} = 9G |\nabla \theta|^2; \\
    \text{when $N$ is NK:} \qquad \frac{\partial \ell }{\partial t} = & -\Delta \ell + \frac{3 (1+|\nabla \ell |^2)}{\ell}, \qquad \frac{\partial\theta}{\partial t} = - \Delta \theta + \frac{\sin(6\theta)}{\ell^2}, \\ \frac{\partial G}{\partial t} = & \left(9 |\nabla \theta|^2 + \frac{3 |\sin(3\theta)|^2}{\ell^2} \right) G.
    \end{align*}
    \emph{Hint:} The rough Laplacian $\Delta=\nabla^*\nabla$ on functions is just $- \Delta_d$. It follows that for $f \in C^{\infty}(L^1)$, we have
    \begin{equation*}
    \Delta f = \frac{f''}{G^2} + \frac{6 \ell'f'}{\ell G^2} - \frac{f'G'}{G^3} \qquad \text{and} \qquad |\nabla f|^2 = \frac{(f')^2}{G^2}.
    \end{equation*}
    \end{enumerate}
\end{exercise}

\begin{exercise}
Let $(M^7, \ph)$ be a manifold with coclosed $\Gt$-structure. The $3$-form $\ph$ is called a \emph{soliton of the Laplacian coflow} if its dual $\ps = \star \ph$ satisfies the stationary equation
    \begin{equation}
    \label{eq: soliton equation}
    \Delta_d \ps = \cL_X \ps + \lambda \ps \qforq \lambda \in \R \qandq X \in \mathfrak{X}(M).
    \end{equation}
    The soliton is \emph{expanding}, \emph{steady}, or \emph{shrinking} if $\lambda>0$, $\lambda=0$ or $\lambda<0$, respectively.
    \begin{enumerate}
    \item Show that there are no compact shrinking solitons of the Laplacian coflow, and that the only compact steady solitons of~\eqref{eq: Laplacian coflow} are given by torsion-free $\Gt$-structures. 
    
    \item Given $X = s(r) \frac{\partial}{\partial r}$, for $s \in C^{\infty}(L^1)$, show that the \emph{coclosed} solitons of the form~\eqref{eq: warped_G2_struct} are precisely given by:
    \begin{align*}
    \text{when $N$ is CY:} & \qquad \lambda =0, \qquad \ell = 1, \qquad \theta = \frac{2}{3} \arctan \left( ce^{br} \right), \qquad s = b \left( \frac{1 - c^2 e^{2br}}{1 + c^2 e^{2br}} \right); \\
    \text{when $N$ is NK:}& \qquad \begin{cases} & \ell' = \cos(3\theta), \\
    & 0 = (\ell^3 \sin(3\theta))'' - 12 \ell \sin(3\theta) - \lambda \ell^3 \sin(3\theta) - (k' \ell ^3 \sin(3\theta))', \\
    & 0 = (\ell^3 \cos(3\theta))' - 3 \ell^2 - \frac{1}{4} \lambda \ell^4 - k' \ell^3 \cos(3\theta), \end{cases}
    \end{align*}
    where $k(r) = \int_{r_0}^r s(u) du$.
    \end{enumerate} 
\end{exercise}

\addcontentsline{toc}{chapter}{References}
\bibliographystyle{plain}
\bibliography{bridges-karigiannis.bib}


\end{document}